\documentclass[12pt, twoside, openany]{report}
\usepackage{amsfonts}
\usepackage{amsmath,amssymb,amsthm, amscd, amsfonts, graphicx,cite, CJK,float}
\usepackage{mathrsfs}
\usepackage{url}
\usepackage{fancy}
\usepackage{indentfirst}
\usepackage{enumerate}
\usepackage{latexsym}
\usepackage[only,ninrm,elvrm,twlrm, sixrm,egtrm,tenrm]{rawfonts}
\usepackage[noend]{algorithmic}
\usepackage{algorithm}
\usepackage{graphicx,psfrag}
\usepackage{psfrag,eepic,epsfig}
\usepackage{sectsty}
\usepackage{algorithm}
\usepackage{algorithmic}

\chapterfont{\LARGE\centering}
\sectionfont{\large\centering}
\makeatletter \@addtoreset{figure}{chapter} \makeatother

\makeatletter
\long\def\@makecaption#1#2{%
   \vskip 10\p@
   \setbox\@tempboxa\hbox{{#1}\ \ #2}%
   \ifdim \wd\@tempboxa >\hsize
       {#1}\ \ #2\par
   \else
       \hbox to\hsize{\hfil\box\@tempboxa\hfil}%
   \fi}
\makeatother

\makeatletter
\let\@afterindentrestore\@afterindentfalse
\let\@afterindentfalse\@afterindenttrue
\@afterindenttrue \makeatother

\let\origdoublepage\cleardoublepage
\newcommand{\clearemptydoublepage}{%
  \clearpage
  {\pagestyle{empty}\origdoublepage}%
}
\let\cleardoublepage\clearemptydoublepage

\newtheorem{theorem}{Theorem}[chapter]

\newtheorem{lemma}[theorem]{Lemma}
\newtheorem{proposition}[theorem]{Proposition}
\newtheorem{corollary}[theorem]{Corollary}
\newtheorem{example}{Example}[chapter]
\newtheorem{remark}{Remark}[chapter]
\newtheorem{definition}{Definition}[chapter]

\renewenvironment{example}{\smallskip\noindent{\bf Example. }}{\smallskip}
\renewenvironment{remark}{\smallskip\noindent{\bf Remark. }}{\smallskip}

\def\qed{\nobreak\quad\raise -2pt\hbox{\vrule\vbox to 10pt{\hrule width 6pt
\vfill\hrule}\vrule}\par\vspace{2ex}}

\newcommand{\dif}{\text{d}}
\renewcommand{\dim}{\text{dim}}

\newcommand{\re}{\text{Re}}
\renewcommand{\i}{\text{i}}
\renewcommand{\exp}{\text{e}}
\newcommand{\im}{\text{Im}}

\renewcommand{\theequation}%
{\arabic{chapter}.\arabic{section}.\arabic{equation}}

\def\qed{\hfill \rule{4pt}{7pt}}

\newcommand{\newchap}[1]{ \chapter{#1}
\chead{\fancyplain{}{\small Chapter \thechapter. {#1} }}
\lhead{\fancyplain{}{\small }} \rhead{\fancyplain{}{\small }} }

\begin{document}

\newpage

\thispagestyle{empty} \vspace*{.5cm} \centerline{\Large \bf
Doctoral Dissertation} \vspace{2cm}
\begin{center}
    {\huge\bf Combinatorial Analysis for Pseudoknot RNA with Complex Structure}
\end{center}
\vskip 1in  \begin{center}
        \large\rm By\\
       Yangyang Zhao
    \end{center}

    \vfill
\begin{center}
        A thesis submitted to the Graduate School of Nankai University\\
        in partial fulfillment of the requirement for the degree of\\
        Doctor of Philosophy\\
        [1.5cm]Nankai University \\
        Tianjin, People's Republic of China\\
        March 2012
    \end{center}
    \vskip0.75in
    \begin{center}
        \rm \copyright\ Copyright  2012
    \end{center}

\cleardoublepage

\pagenumbering{Roman}


\newpage
\def\papername{Abstract}\label{abstractpapge}
\pagestyle{fancyplain} \chead{\fancyplain{}{\small Abstract}}
\lhead{\fancyplain{}{\small }} \rhead{\fancyplain{}{\small }}
 \vskip 24pt
\addcontentsline{toc}{chapter}{Abstract (in English)}

\begin{center}
{\bf \fontsize {18}{18}\sffamily Abstract \\ [18pt]}
\end{center}
There exists many complicated $k$-noncrossing pseudoknot RNA structures in nature based on some special conditions. The special characteristic of RNA structures gives us great challenges in researching the enumeration, prediction and the analysis of prediction algorithm. We will study two kinds of typical $k$-noncrossing pseudoknot RNAs with complex structures separately.

The main content of Chapter 1 introduces the background and the significance of the project. We also present the chief results of this thesis.

In Chapter 2, we mainly illustrate the basic concepts, including the following items:
\begin{enumerate}
\item {\bf The construction of RNA}, containing the secondary structures, pseudoknot structures and substructure of RNA.
\item {\bf Symbolic enumeration method}. It is one of the most important method for studying the enumeration of RNA structures. All the generating functions are concluded by this method.
\item {\bf Combinatorial analysis}, basic tool for asymptotically analyzing the enumeration of RNA structures. We mainly show $D$-finiteness, $P$-recursive, singularity analysis and distribution analysis in detail.
\end{enumerate}

In Chapter 3, we discuss a kind of complicated structures named modular $k$-noncrossing diagram. Such diagrams frequently appear in RNA prediction. The prediction program based on modular $k$-noncrossing diagram is much more difficult to code. Therefore, we need to make the diagram clear, including the way of building it and the time complexity we build. To reach the purpose, we invent a whole new method to enumerate and find the asymptotic behavior of modular $k$-noncrossing diagrams. The method is made with reconstructing $V_k$-shape, recurrence, differential equation and symbolic method. Finally we get the result that the asymptotic formula of the number of modular, $k$-noncrossing diagram with $n$ nucleotides, which is,
\begin{equation*}
{\sf Q}_k(n)\sim  c_{k}\, n^{-((k-1)^2+(k-1)/2)}\,
(\gamma_{k}^{-1})^n ,\quad k\geq 3,
\end{equation*}
where $\gamma_k$ is the minimum real solution of $\vartheta(z)=\rho_k^2$  and  $c_k$ is some positive constant, see eq.~(\ref{E:gut1}).

In Chapter 4, we emphasize the research on $k$-noncrossing skeleton structure, especially when $k=3$. The skeleton structure is well known in Huang's folding algorithm for $3$-noncrossing RNA structures. The algorithm constructs a skeleton tree first, in which every leaf is a skeleton structure. The structure of the skeleton tree is so complicated that the complexity of the algorithm is only concluded by experiment data in their paper. We will give a proof for the complexity and obtain the asymptotic formula of the number of canonical $3$-noncrossing skeleton diagrams with $n$ vertices,
\begin{equation*}
{\sf S}^{[4]}_{3}(n)\sim  C'\, n^{-5}\,
(\eta^{-1})^n,
\end{equation*}
where $C'\approx7892.16$, $\eta\approx 0.4934$. Afterwards, we study the statistical properties of arcs in canonical $3$-noncrossing skeleton diagrams. We prove the central limit theorem for arcs distributions in terms of their bivariate generating function. Our result allows to estimate the arc numbers in a random canonical $3$-noncrossing skeleton diagrams.

\vskip 18pt


\newpage
\pagestyle{fancyplain} \chead{\fancyplain{}{\small Contents}}
\lhead{\fancyplain{}{\small }} \rhead{\fancyplain{}{\small }}

\setcounter{tocdepth}{1} \tableofcontents

\thispagestyle{fancyplain}

\contentsline{toc}{\numberline{}\bf \hspace{-0.8cm}
Bibliography}{\hfill\bf \pageref{bibliographypage}}\\ \vskip
-1.2cm \contentsline{toc}{\numberline{}\bf \hspace{-0.8cm}
Acknowledgement}{\hfill\bf \pageref{acknowledgepapge}}\\ \vskip
-1.2cm \contentsline{toc}{\numberline{}\bf \hspace{-0.8cm}
Resum\'{e}}{\hfill\bf \pageref{resumepapge}}

\newpage
\newchap{Introduction}
\setcounter{page}{0} \pagenumbering{arabic}
\thispagestyle{fancyplain}

\section{Background}
The Ribonucleotide acid (RNA) is a single stranded molecule of four different nucleotides \textbf{A}, \textbf{C}, \textbf{G} and \textbf{U}, together with Watson-Crick (\textbf{A}-\textbf{U}, \textbf{G}-\textbf{C}) and (\textbf{U}-\textbf{G}) base pairs. It is one of the three major macromolecules that are essential for all known forms of life. The sequence of nucleotides allows RNA to encode genetic information. Over the decades of years, for the need of medicine and pharmacy, scientists give an intense interesting in the structure of RNA. The traditional way of researching RNA structure is cutting RNA by RNA enzyme or some reagent and analyzing by sequencing gel, which is time-consuming, high-tech demanding and expensive. To solve these problems, people turned to focus on using computer to predict a specific RNA structure.

\restylefloat{figure}\begin{figure}[h!t!b!p]
\centering
\scalebox{0.8}[0.8]{\includegraphics{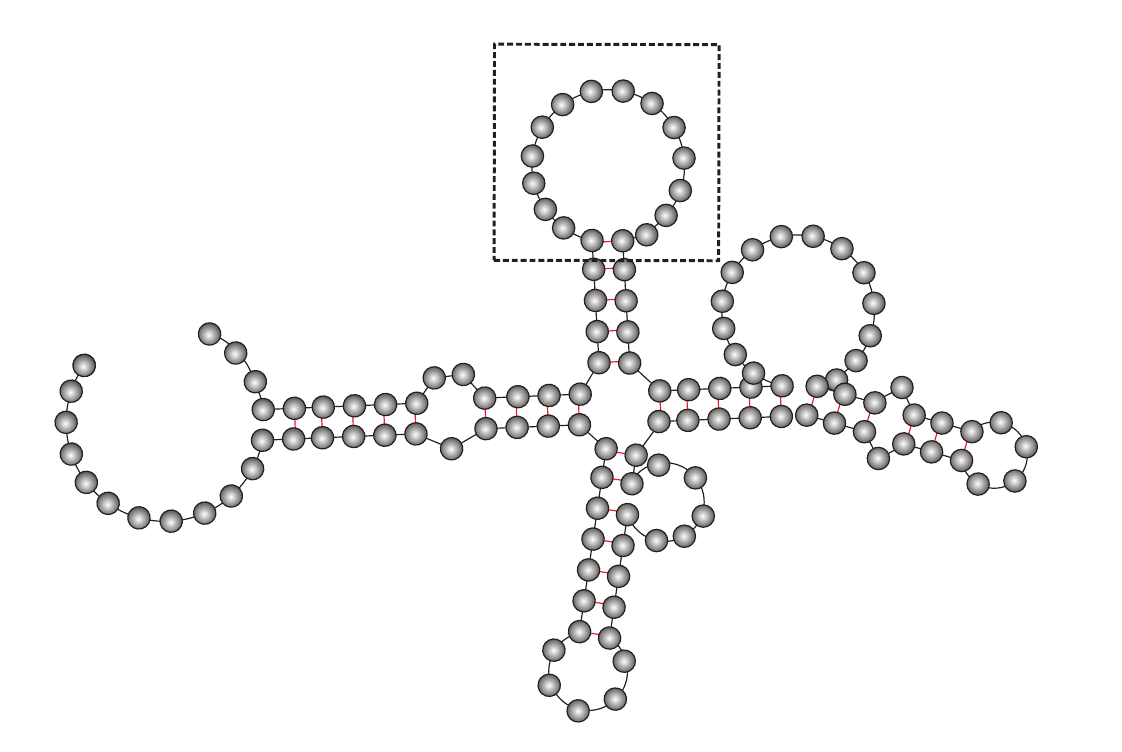}}
\caption{\small HAR1F RF00635 RNA secondary structure represented as a  planar graph. There are more than two isolated vertices in a hairpin marked by a dashed-line box.}
\label{F:secondary}
\end{figure}

RNA structures are separated into two categories, which are secondary structures and pseudoknot structures respectively. The secondary structures are the first been studied. In 1978, Waterman etc. completely investigate the combinatorial properties of secondary structure \cite{Waterman:78b}\cite{Waterman:79}\cite{Waterman:93}\cite{Waterman:94a}. Assume that the number of secondary structures with $n$ nucleotides having arcs of length at least $\lambda$ is ${\sf T}^{[\lambda]}_2(n)$. We have the recursion that
\begin{equation}
{\sf T}^{[\lambda]}_2(n)={\sf T}^{[\lambda]}_2(n-1)+\sum_{i=0}^{n-\lambda-1}{\sf T}^{[\lambda]}_2(i){\sf T}^{[\lambda]}_2(n-i-2),
\end{equation}
where ${\sf T}^{[\lambda]}_2(n)$ for $0\leq n\leq \lambda$. The recursion above tells us the way of constructing an arbitrary secondary structure and it is also the fundamental of efficient minimum free energy folding (mfe) of RNA secondary structure \cite{Hofacker:02,Hofacker:03,Hofacker:94}. The RNA secondary structure does not allow one or two isolated nucleotides in any one of hairpin loops, see Fig. \ref{F:secondary}\footnote{The planar graph of HAR1F RF00635 is drawn by reference to \url{http://en.wikipedia.org/wiki/File:HAR1F_RF00635_rna_secondary_structure.jpg}.}.

\restylefloat{figure}\begin{figure}[h!t!b!p]
\centering
\includegraphics{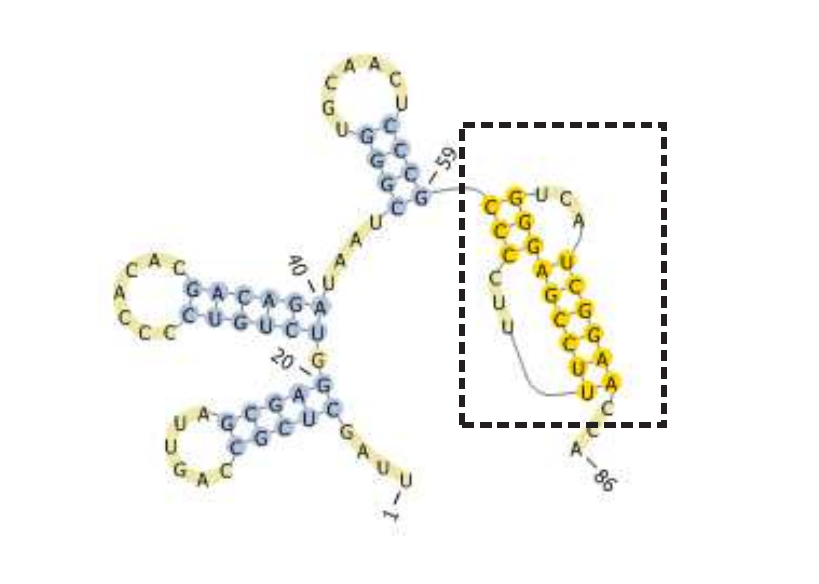}
\caption{\small A planar graph of turnip yellow mosaic virus. The substructure marked by a dashed-line box is the {\it pseudoknot}.}\label{F:yellow_mosaic}
\end{figure}

There exists other RNA structures having cross-serial nucleotide interactions which is different from secondary structure. We call them pseudoknot structure \cite{Stadler:99}. The pseudoknot structure was first discovered in yellow mosaic virus in \cite{Staple:05}, see Fig. \ref{F:yellow_mosaic}\footnote{The figure is cited from \url{http://pseudoviewer.inha.ac.kr/Examples.htm}.}. Unlike the secondary structure, they are more complicated and hard to find a bijection between them and some combinatorial structure. It gives people who want to predict pseudoknot structures a big trouble. In 1999, Elena Rivas and Sean R. Eddy gave us a classical dynamic programming algorithm for RNA pseudoknot structure \cite{Rivas:99}. Meanwhile, Lyns{\o} etc. \cite{lyngso:00} advanced a new predicting method based on a energy models. However, for the lack of the knowledge about the exact combinatorial meaning of pseudoknot, they cannot control the crossing number. For this reason, people began to study the combinatorial properties of RNA pseudoknot structure \cite{Waterman:80}. In the year of 2007, the paper \cite{Chen} written by Chen and Stanley etc. showed us an involution combinatorial structure, ``$k$-noncrossing matching". It actually builds a connection between a RNA pseudoknot structure and a $k$-noncrossing matching. Based on this idea, Reidys etc. generalize the $k$-noncrossing matching to specified $k$-noncrossing diagram, which have bijection to the RNA pseudoknot structure \cite{Reidys:08}\cite{Reidys:07lego}\cite{Reidys:08han}. Furthermore, the generating functions of some kinds of $k$-noncrossing diagram were computed the asymptotic formula, which gives us a proof of growth rate when the number of nucleotides is increasing.

\restylefloat{figure}\begin{figure}[h!t!b!p]
\centering
\scalebox{1.2}[1.2]{\includegraphics{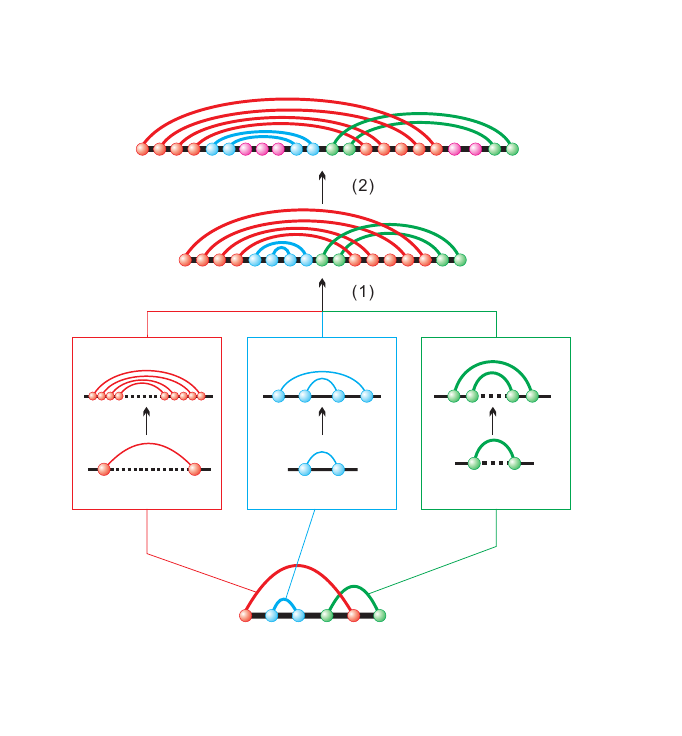}}
\caption{\small Modular $k$-noncrossing diagrams: the inflation
method. A modular $3$-noncrossing diagram (top) is derived by
inflating its ${\sf V}_3$-shape (bottom) in two steps. First we
individually inflate each shape-arc into a more complex
configuration and second insert isolated vertices (purple).}
\label{F:inflating}
\end{figure}

Around the year of 2010, Jin \cite{Reidys:07pseu} focused on deriving the generating functions and asymptotic formula of $3$-noncrossing diagram, while Han \cite{Reidys:08han} focused on the diagram with $4$ or higher cross number. Ma etc, \cite{Reidys:08ma} had further studies on some $k$-noncrossing diagrams satisfying some constraints, i.e.\textit{ $k$-noncrossing $\tau$-canonical diagrams}. The constraints specify minimum arc-length and stack-length of the diagrams, due to the biophysical constraints. All these kinds of diagrams belong to the class of diagrams, whose arc-length $\lambda$ satisfies $1\leq \lambda\leq \tau+1$. It means that the method used to compute the generating function and asymptotic formula of such a class of diagrams cannot be applied to some $k$-noncrossing $\tau$-canonical diagram with arc length larger than $\tau+1$, such as $k$-noncrossing 2-canonical diagrams with arc length not less than 4, which is also known as modular $k$-noncrossing diagram. We have to develop a new way to treat this kind of diagrams. Fortunately, the ${\sf V}_k$-shapes will help us find the answer. It is like a core of a $k$-noncrossing diagram. With an action called ``inflation", every $k$-noncrossing diagrams can be produced by adding isolated vertices and stacks based on some ${\sf V}_k$-shape. The ${\sf V}_k$-shape in \cite{Reidys:09shape} can only be inflated to a $k$-noncrossing $\tau$-canonical diagram with the arc length $\lambda$ satisfying $1\leq \lambda\leq \tau+1$. We have to reform the traditional ${\sf V}_k$-shape, called colored ${\sf V}_k$-shape, so as to adapt to the new problem. The main idea is to build modular $k$-noncrossing diagrams via inflating their colored ${\sf V}_k$-shapes \cite{Zhao:10}, see Fig. \ref{F:inflating}. The inflation gives
rise to ``stem-modules'' over shape-arcs and is the key for the
symbolic derivation of ${\bf Q}_k(z)$. One additional observation
maybe worth to be pointed out: the computation of the generating
function of colored shapes in Section~\ref{S:color}, hinges on the
intuition that the crossings of short arcs are relatively simple and
give rise to manageable recursions. The coloring of these shapes
then allows to identify the arc-configurations that require special
attention during the inflation process. Our results are of
importance in the context of RNA pseudoknot structures
\cite{Rietveld:82} and evolutionary optimization
\cite{Reidys:02}.  The method of generating modular $k$-noncrossing diagrams gives contributions to analyze the joint structure \cite{Thomas:11}\cite{Thomas:12}.

\restylefloat{figure}\begin{figure}[ht]
\centerline{\includegraphics[width=1\textwidth]{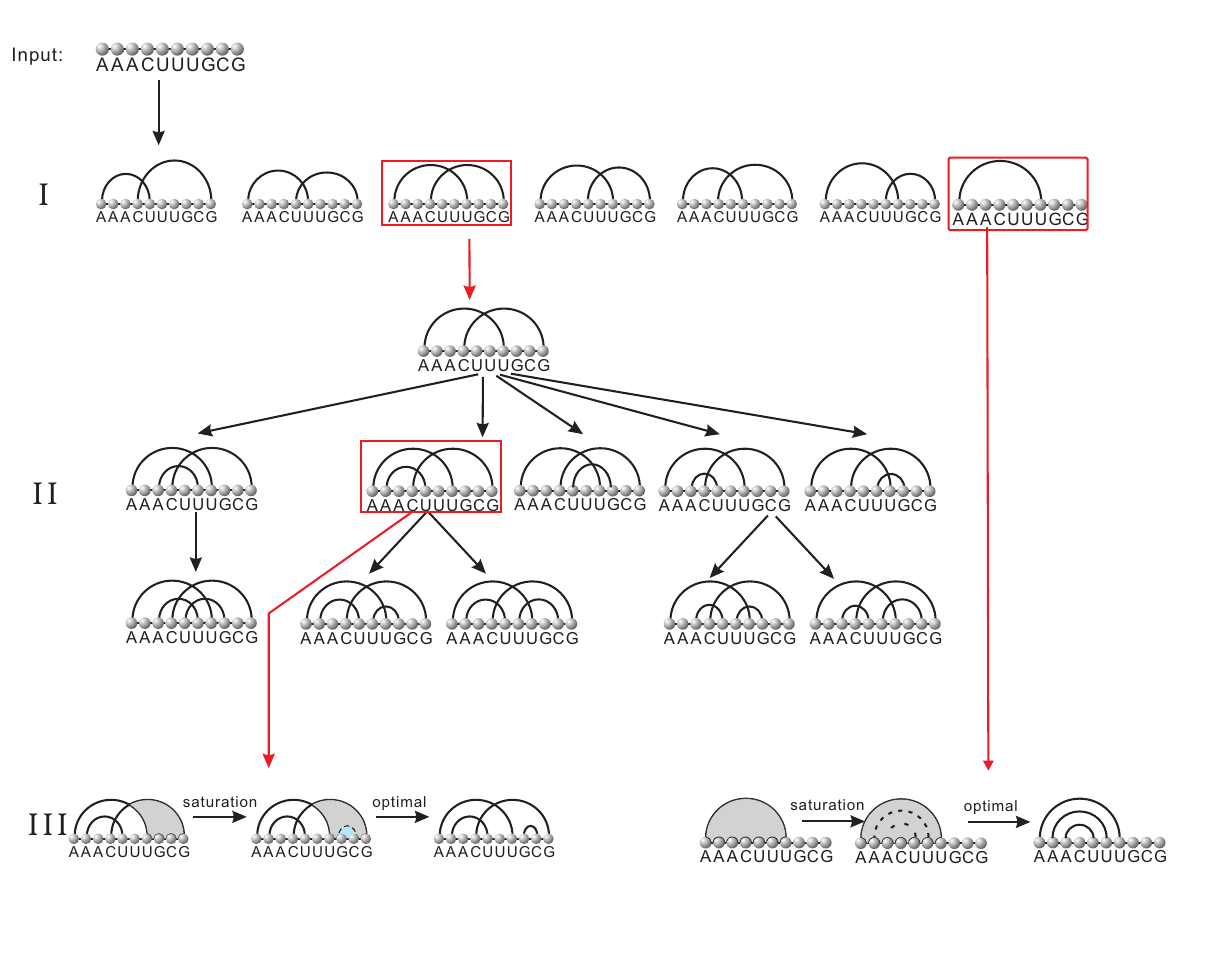}
\hskip8pt}
\caption{\small An outline of {\tt cross}\cite{Huang:09}: the generation of motifs (I), the
construction of skeleta-trees, that are rooted in irreducible shadows (II) and
the saturation (III). During the latter we derive via DP-routines optimal
fillings of intervals of skeleta. The red arrows represent the processing of
two motifs, one of which leads to the generation of a skeleton tree, while
the other leads directly to the saturation routine.
}
\label{F:sketch}
\end{figure}

Besides modular $k$-noncrossing diagram, another diagram, called {skeleton diagram}, is worth to study. It has a special characteristic that every arc is crossed by another arc. In the world of nature, the skeleton diagram rarely appears independently. However, skeleton diagram is an important substructure in every pseudoknot RNA structure. It is like a skeleton covering outside a pseudoknot diagram, and small substructures are filled in the ``gap" of the skeleton diagrams, like muscle, see Fig. \ref{F:skeleton-intro}. Making the use of this special property, we can design a dynamic program to predict RNA structure. In 2008, Huang etc. \cite{Huang:09} designed a folding $3$-noncrossing algorithm, {\sf cross}, based on the skeleton diagrams. The main idea is generating motifs first, then constructing a skeleta trees, rooted in irreducible shadows and saturation, during which via DP-routines, optimal fillings of skeleta-intervals are derived, see Fig. \ref{F:sketch}. The key difference to any other pseudoknot folding algorthm is the fact that {\sf cross} has a transparent, combinatorially specified, output class. {\sf Cross} is by design an algorithm of exponential time complexity by virtue of its construction of its skeleta-trees. Only in its saturation phase it employs vector versions of DP-routines. The fact shows us we must make clear the time complexity of generating skeleta-trees if we want to know the time complexity of the algorithm. It turns to be a problem that what the statistics property of the skeleta-trees is. Since any skeleton diagram can be found in a skeleta-tree, we need to find a way to enumerate the skeleton diagrams.

\restylefloat{figure}\begin{figure}[h!t!p!b]
\centering
\includegraphics{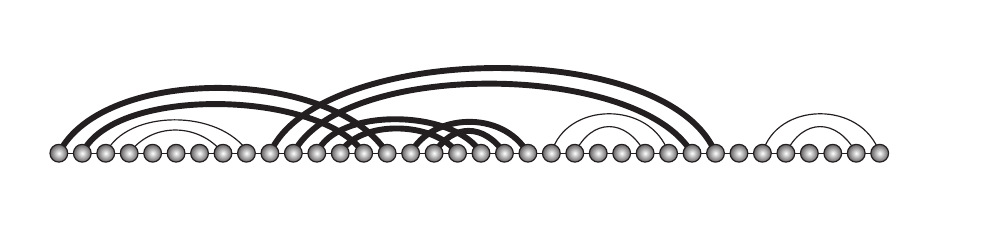}
\caption{A skeleton diagram in a pseudoknot structure. The skeleton diagram is marked by thick line. The substructures marked by thin line are embedded in the intervals generated through splitting nucleotide (grey balls) sequence by the skeleton diagram.\label{F:skeleton-intro}}
\end{figure}

\section{Outline of the thesis}
Besides introducing some basic concepts in this thesis, we mainly discuss about two subjects, which are the modular, $k$-noncrossing diagrams and the $3$-noncrossing skeleton diagrams, respectively. The thesis is organized as following.

In Chapter 1, we introduce the background and motivation of the thesis and illustrate the outline of the result and key method. The basic concepts and fundamental theorems which are highly used in the thesis, are present in Chapter 2.

The main results will be shown with the beginning of Chapter 3. The combinatorial analysis of modular, $k$-noncrossing diagram is the main content of Chapter 3 . Based on ${\sf V}_k$-shape, we compute the generating function of colored shape, ${\bf W}_k(x,y,w)$ and ${\bf I}_k(x,y,z,w)$, which are
\begin{equation}\label{chapter1:3gf}
{\bf W}_k(x,y,w)=(1+x)v \, {\bf F}_k\left(x(1+x)v^2\right),
\end{equation}
 and
\begin{equation}\label{chapter1:5gf}
{\bf I}_k(x,y,z,w,t)=\frac{1+x}{\theta}{\bf F}_k
\left(\frac{x(1+(2w-1)x+(t-1)x^2)}{\theta^2}\right),
\end{equation}
where ${\bf F}_k$ is the generating function of $k$-noncrossing matching (see Section ) and
\begin{eqnarray*}
 v&=&\left((1-w)x^3+(1-w)x^2+(2-y)x+1\right)^{-1}\\
 \theta&=&1-(y-2)x+(2w-z-1)x^2+(2w-z-1)x^3
\end{eqnarray*}
for $k>2$.

With help of the above two generating functions of colored shapes and symbolic enumeration method, we obtain the generating function of modular, $k$-noncrossing diagrams when $k>2$, which is
\begin{eqnarray}\label{chapter1:gut0}
{\bf Q}_{k}(z) &=&
\frac{1-z^2+z^4}{q(z)}
{\bf F}_k\left(\vartheta(z)\right),
\end{eqnarray}
where
\begin{eqnarray}
q(z)& = & 1-z-z^2+z^3+2z^4+z^6-z^8+z^{10}-z^{12}
\nonumber \\
\vartheta(z) & = &
\frac{z^4(1-z^2-z^4+2z^6-z^8)}{q(z)^2}.
\end{eqnarray}
Eq. (\ref{chapter1:gut0}) shows us the generating function of modular, $k$-noncrossing diagrams is composed of an rational expression and a composite function. Using the asymptotic method introduced in Chapter 2, we are able to conclude the asymptotic formula of modular, $k$-noncrossing diagrams with $n$ vertices, which is
\begin{equation*}
{\sf Q}_k(n)\sim  c_{k}\, n^{-((k-1)^2+(k-1)/2)}\,
(\gamma_{k}^{-1})^n ,\quad \text{for some {$c_{k}^{}>0$}}.
\end{equation*}

After we finish discussing about modular, $k$-noncrossing diagram, we start to focus another interesting diagram, $3$-noncrossing skeleton diagram, which will be studied in Chapter 4 \cite{Zhao:12}. Similarly as modular, $k$-noncrossing diagram, the generating function of $3$-noncrossing skeleton diagram should be given by inflating its shape (we call it $3$-noncrossing skeleton shape), and the conclusion is
\begin{equation}\label{chapter1:s34GF}
{\bf S}^{[4]}_{3}(z)=(1-z){\bf G}\left(\left({{\sqrt{w_0(z)}z}\over{1-z}}\right)^2\right),
\end{equation}
where
\begin{eqnarray*}
{\bf G}(x)={\bf S}(x)-1-x,
w_0(z)={z^4 \over {1-z^2+z^6}}.
\end{eqnarray*}
Here ${\bf S}(x)$ is the generating function of $3$-noncrossing skeleton shape. Before doing asymptotic of $3$-noncrossing skeleton diagram, we must obtain the generating function of ${\bf S}(x)$ and doing combinatorial analysis. The main idea is utilizing the relationship between skeleton diagram and irreducible structure \cite{Jin:10} to compute the generating function of skeleton matching (a matching satisfying the condition of skeleton). Then we ``shrink" the skeleton matching to the skeleton shape. Unlike the $3$-noncrossing matching, the generating function of $3$-noncrossing skeleton matching has no explicit formula. Thanks to the Bender's theorem \cite{Bender:86} and Eric's method \cite{Fusy:10}, we finally solve the asymptotic problem of $3$-noncrossing skeleton matching, which is
\begin{equation*}
{\sf S}(h)\ \sim\ C\cdot h^{-5}R^{-h}, \end{equation*}
where ${\sf S}(h)$ is the number of $3$-noncrossing skeleton matching with $h$ arcs and
\begin{eqnarray*}
C=\frac{24(-512+165\pi)^5}{3125(256-81\pi)^5\pi}\approx 3.03096,\quad
R=256\left(\frac{11}{8}-\frac{64}{15\pi}\right)^2\approx 0.0729.
\end{eqnarray*}

Using the asymptotic property of $3$-noncrossing matching and eq. (\ref{chapter1:s34GF}), we deduce that the asymptotic formula of $3$-noncrossing skeleton diagram is
\begin{equation*}
{\sf S}^{[4]}_{3}(n)\sim  C'\, n^{-5}\,
(\eta^{-1})^n ,\quad \text{for some {$C'>0$}},
\end{equation*}
where ${\sf S}^{[4]}_{3}(n)$ represents the number of canonical $3$-noncrossing skeleton diagrams with $n$ vertices and $C'\approx7892.16$, $\eta\approx 0.4934$.

\newchap{Basic concepts}
\thispagestyle{fancyplain}

In this chapter we provide the mathematical foundation of the thesis. The main object here is the generating function and asymptotic formula of $k$-noncrossing matchings, which will play a central role for in RNA pseudoknot structures.

We begin with the nomenclatures and notations needed for the thesis. It contains two kinds of expressions for RNA pseudoknot structure, the elements in the RNA structure, diagram,  and {\it $k$-noncrossing} {\it$\tau$-canonical}, etc.

Next we introduce the fundamental combinatorial tools for obtaining the generating functions of $k$-noncrossing diagrams. We use symbolic enumeration present by Flajolet \cite{Flajolet:07a} to get the generating function, avoiding the complicated recursive method.

After that we begin to discuss the combinatorial analyzing method for generating functions. The key method is the singularity analysis. However, before we introduce it, an important concept must be inducted, which is $D$-finiteness. The $D$-finiteness and singularity analysis are going to help us to analyze the generating function of $k$-noncrossing diagrams asymptotically.

We then conclude this chapter by the central limit theorem due to Bender \cite{Bender:73}. The central limit theorem shows that a sequence of random variables have a limit distribution Gaussian or normal distribution. The central limit theorem will be used in Chapter 4 to analyze the distribution of $3$-noncrossing skeleton diagrams with $n$ vertices and $h$ arcs.

\section{Nomenclatures and Notations}

An arbitrary RNA structures can be uniquely expressed by two forms respectively, which are planar graph and diagram. A {\it planar graph} of RNA structure is a graph that obtained by embedding the $3$-dimensional RNA structure into the plain, in which the nodes represent nucleotides and edges represent phosphodiester bond and Watson-Crick base pairs ({\bf A-U}, {\bf G-C} and {\bf U-G}), see Fig. \ref{F:yellow_mosaic}. If we ``straighten" the edges representing phosphodiester bond, we will get the other expression, {\it diagram}. The strict definition of diagram follows
\begin{definition}
A {\bf diagram} is a labeled graph over the vertex set $[n]={1,2,\ldots,n}$, represented by drawing its vertices on a horizonal line and its arcs $(i,j)$, $i<j$ in the upper half plain, where each vertex is connected with at most one arc, see Fig. \ref{F:intro-diagram}. A vertex in a diagram is called an {\bf isolated vertex} if there are not any arcs connecting with the vertex.
\end{definition}

\restylefloat{figure}\begin{figure}[h!t!b!p]
\centering \scalebox{0.8}[0.8]{\includegraphics{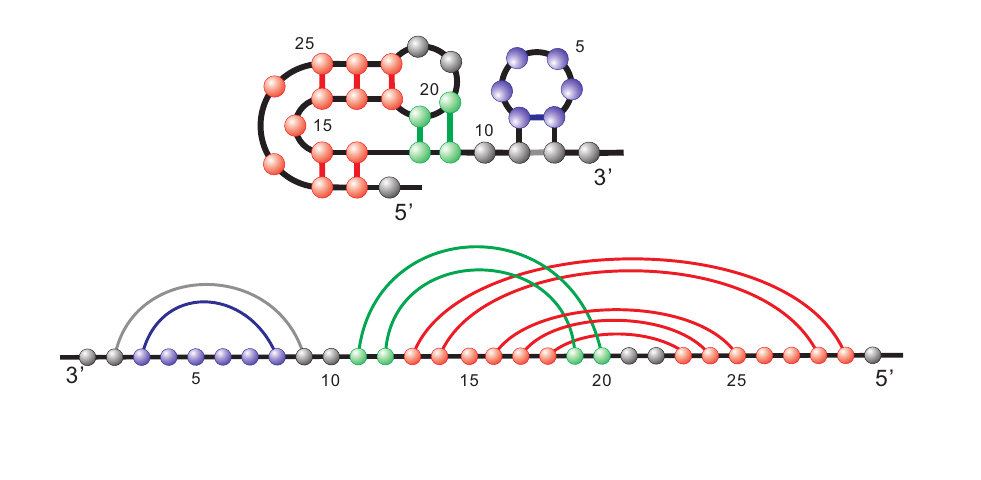}}
\caption{\small Features of a modular $3$-noncrossing diagram
represented as planar graph (top) and in diagram representation
(bottom). We display a stack of length two (green), a stem of size
two (red) and a $5$-arc (blue).} \label{F:intro-diagram}
\end{figure}

Obviously, vertices and arcs correspond to nucleotides and Waston-Crick base pairs, respectively. In a diagram, a sequence of arcs ${(i_s,j_s)}_{1\leq s\leq k}$ is called {\it $k$-nesting}, if they satisfy
\begin{equation*}
i_s<i_{s-1}<\ldots<i_2<i_1<j_1<j_2<\ldots<j_{s-1}<j_s.
\end{equation*}
Similarly,  a sequence of arcs ${(i_s,j_s)}_{1\leq s\leq k}$ is called {\it $k$-crossing} if
\begin{equation*}
i_1<i_2<\ldots<i_k<j_1<j_2<\ldots<j_k.
\end{equation*}
Similarly,
If a diagram have $k-1$-crossing arcs at most, we call it {\it $k$-noncrossing} diagram. Particularly, if $k$ equals 2, the diagram will degenerate to noncrossing diagram, which corresponds to secondary structure.

A {\it $k$-noncrossing, $\tau$-canonical} structure, is a diagram in which
\begin{itemize}
\item there exists no $k$-crossing arcs,
\item stack has at least size $\tau$, and
\item any arc $(i,j)$ has a minimum arc length $j-i\geq2$
\end{itemize}

We next specify further properties of $k$-noncrossing diagrams:
\begin{itemize}
\item a stack\index{stack!} of size\index{stack!-length} $\sigma$,
$S_{i,j}^{\sigma}$, is a
maximal sequence of ``parallel'' arcs,
\begin{equation*}
((i,j),(i+1,j-1),\dots,(i+(\sigma-1),j-(\sigma-1))).
\end{equation*}
We call a stack of size $\sigma$ a $\sigma$-stack\index{stack! $\sigma$-},
\item a stem\index{stem} of size $s$ is a sequence
\begin{equation*}
\left(S_{i_1,j_1}^{\sigma_1},S_{i_2,j_2}^{\sigma_2},
\ldots,S_{i_{s},j_{s}}^{\sigma_{s}}\right)
\end{equation*}
where $S_{i_{m},j_{m}}^{\sigma_m}$ is nested in $S_{i_{m-1},
j_{m-1}}^{\sigma_{m-1}}$ such that any arc nested in
$S_{i_{m-1},j_{m-1}}^{\sigma_{m-1}}$ is either
contained or nested in $S_{i_{m},j_{m}}^{\sigma_m}$, for $2\leq m\leq s$.
\end{itemize}
For an illustration of the above structural features, see
Fig.~\ref{F:intro-diagram}.
Note that given a stem
\begin{equation*}
\left(S_{i_1,j_1}^{\sigma_1},
S_{i_2,j_2}^{\sigma_2},\ldots,S_{i_{s},j_{s}}^{\sigma_{s}}\right)
\end{equation*}
the maximality of the stacks implies that any two nested stacks within
a stem, $S_{i_{m},j_{m}}^{\sigma_m}$ and $S_{i_{m-1},j_{m-1}}^{\sigma_{
m-1}}$ are separated by a nonempty interval of isolated vertices between
$i_{m-1}+(\sigma_{m-1}-1)$ and $i_m$ or $j_{m-1}-(\sigma_{m-1}-1)$ and
$j_m$, respectively.

\section{Symbolic enumeration method}

When we deal with the enumeration of diagrams, we actually deal with $discrete$ objects, which can be finitely described by constructions rules. For example, the vertices, arcs, stacks and stems in $k$-noncrossing RNA structures. Our major is to {\it enumerate} such objects according to some characteristics parameter(s).

\begin{definition}
A {\bf combinatorial class}, or simply a {\bf class}, is a finite or denumerable set on which a size function is defined, satisfying the following conditions:
(i) the size of an element is a non-negative integer;
(ii) the number of elements of any given size is finite.
\end{definition}

\begin{definition}
The {\bf counting sequence} of a combinatorial class is the sequence of integers $(a_n)_{n\geq 0}$ where $a_n={\rm card}(a_n)$ is the number of objects in class $\mathcal{A}$ that have size $n$.
\end{definition}

Next, for combinatorial enumeration purposes, it is convenient to identify combinatorial classes that are merely variants of one another.
\begin{definition}
Two combinatorial classes $\mathcal{A}$ and $\mathcal{B}$ are said to be (combinatorially) isomorphic, which is written $\mathcal{A}\cong\mathcal{B}$, iff their counting sequences are identical. This condition is equivalent to the existence of a bijection from $\mathcal{A}$ to $\mathcal{B}$ that preserves size, and one also says that $\mathcal{A}$ and $\mathcal{B}$ are bijectively equivalent.
\end{definition}

 Suppose we have a combinatorial class and the corresponding counting sequence. We define the ordinary generating function (OGF) as follows.

\begin{definition}
The {\bf ordinary generating function (OGF)} of a sequence $(a_n)$ is the formal power series
\begin{equation}
A(z)=\sum_{n=0}^\infty a_n z^n.
\end{equation}
The ordiary generating function (OGF) of a combinatorial class $\mathcal{A}$ is the generating function of the number $A_n={\rm card} (a_n)$. Equivalently, the OGF of class $\mathcal{A}$ admits the combinatorial form
\begin{equation}\label{chapter2:OGF}
A(z)=\sum_{\alpha\in \mathcal{A}}z^{|\alpha|}.
\end{equation}
It is also said that the variable $z$ marks size in the generating function
\end{definition}

The combinatorial form of an OGF in eq. (\ref{chapter2:OGF}) results straightforwardly from observing that the term $z^n$ occurs as many times as there are objects in $\mathbb{A}$ having size $n$.

We let generally $[z^n]f(z)$ denote the operation of extracting the coefficient of $z^n$ in the formal power series $f(z)=\sum f_n z^n$, so that
\begin{equation}
[z^n]\left(\sum_{n\geq 0}f_n z^n\right)=f_n
\end{equation}

\begin{example}
 {\bf Noncrossing matching.} Let $\mathcal{M}$ be the class of the noncrossing matchings. Then $\mathcal{M}$ can be represented as
 \begin{equation*}
 \mathcal{M}:=\{\varepsilon,\{(1,2)\},\{(1,2),(3,4)\},\{(1,4),(2,3)\},\ldots\},
 \end{equation*}
 where $\varepsilon$ is an empty noncrossing matching without arcs. For the matching $\gamma=\{(1,2),(3,4)\}$ we have $|\gamma|=2$. We already know the number of noncrossing matching with $n$ arcs is $c_n=\frac{1}{n+1}{{2n}\choose n}$. Therefore we have the following equations,
 \begin{equation*}
 M(z)=\sum_{\gamma\in\mathcal{M}}z^{\gamma}=\sum_{n\geq 0}c_n z^n=\frac{1-\sqrt{1-4z}}{2z}.
 \end{equation*}
\end{example}

Now we start to learn about the basic {\it constructions} that constitute the core of a specification language for combinatorial structures. For the demand of the thesis, we  mainly discuss about {\it disjoint union}, {\it Cartesian product} and {\it sequence construction}.

First, assume we are given a class $\mathcal{E}$ called the {\it neutral class} that consists of a single object of size $0$, and a class $\mathcal{Z}$ called an {\it atomic class} comprising a single element of size $1$. Clearly, the generating functions of a neutral class $\mathcal{E}$ and an atomic class $\mathcal{Z}$ are
\begin{equation*}
E(z)=1,\qquad Z(z)=z,
\end{equation*}
corresponding to the unit 1,and the variable $z$, of generating functions.
\begin{proposition}
Let $\mathcal{A}$, $\mathcal{B}$, $\mathcal{C}$ be combinatorial classes. $A_n$, $B_n$, $C_n$ are the corresponding cardinals of size $n$ and $A(x)$, $B(x)$, $C(x)$ are the corresponding ordinary generating functions.
\begin{enumerate}
\item {\bf Disjoint union.} $\mathcal{A}$, $\mathcal{B}$ and $\mathcal{C}$ satisfy
    \begin{equation}
    \mathcal{A}=\mathcal{B}\cup \mathcal{C}, \qquad\text{with}\quad \mathcal{B}\cap \mathcal{C}=\emptyset,
    \end{equation}
    with size defined in a consistent manner: for $\omega\in\mathcal{A}$,
    \begin{equation}
    |\omega|_\mathcal{A}=
    \begin{cases}
    |\omega|_\mathcal{B}&\quad \text{if $\omega\in\mathcal{B}$}\\
    |\omega|_\mathcal{C}&\quad \text{if $\omega\in\mathcal{C}$}.
    \end{cases}
    \end{equation}
    One has
    \begin{equation}
    A_n=B_n+C_n,
    \end{equation}
    which, at generating function level, means
    \begin{equation}
    A(z)=B(z)+C(z).
    \end{equation}
\item {\bf Cartesian product.} This construction applied to two classes $\mathcal{B}$ and $\mathcal{C}$ forms ordered pairs,
    \begin{equation}
    \mathcal{A}=\mathcal{B}\times\mathcal{C}\quad \text{iff}\quad \mathcal{A}=\{\alpha=(\beta,\gamma)\mid \beta\in\mathcal{B},\gamma\in\mathcal{C}\},
    \end{equation}
    with the size of a pair $\alpha=(\beta,\gamma)$ being defined by
    \begin{equation}
    |\alpha|_\mathcal{A}=|\beta|_\mathcal{B}+|\gamma|_\mathcal{C}.
    \end{equation}
    It is immediately seen that the counting sequences corresponding to $\mathcal{A}, \mathcal{B}, \mathcal{C}$ are related by the convolution relation
    \begin{equation}
    A_n=\sum_{k=0}^n B_k C_{n-k},
    \end{equation}
    which means admissibility. Furthermore, we recognize here the formula for a product of two power series,
    \begin{equation}
    A(z)=B(z)\cdot C(z).
    \end{equation}
\item {\bf Sequence construction.} If $\mathcal{B}$ is a class then the sequence class {\sc Seq}$(\mathcal{B})$ is defined as the infinite sum
    \begin{equation}\label{chapter2:sequence}
    \textsc{Seq}(\mathcal{B})=\{\epsilon\}+\mathcal{B}+\mathcal{B}\times\mathcal{B}
    +\cdots
    \end{equation}
    with $\epsilon$ being a neutral structure. In other words, we have
    \begin{equation*}
    \mathcal{A}=\left\{(\beta_1,\ldots,\beta_\ell)\mid \ell\geq 0, \beta_j\in\mathcal{B}\right\}.
    \end{equation*}
    According cartesian product and eq. (\ref{chapter2:sequence}), the generating function of $\mathcal{A}$ will hold as follows,
    \begin{equation}
    A(z)=1+B(z)+B(z)\cdot B(z)+\cdots=\frac{1}{1-B(z)}.
    \end{equation}
\end{enumerate}
\end{proposition}

\begin{example}
Interval coverings. Let $\mathcal{Z}:={a}$. Then $\mathcal{A}=\mathcal{Z}+\mathcal{Z}\times\mathcal{Z}$ is a set of two elements, $a$ and $(a,a)$, which we choose to note as $\{a, aa\}$. Then $\mathcal{C}=\textsc{Seq}(\mathcal{A})$ contains
\begin{equation*}
a, a\ a, aa, a\ aa, aa\ a, aa\ aa, a\ a\ a\ a,\ldots
\end{equation*}
with the notion of size adopted, the objects of size $n$ in $\mathcal{C}=\textsc{Seq}(\mathcal{Z}+\mathcal{Z}\times\mathcal{Z})$ are isomorphic to the covering of $[0,n]$ by intervals of length either 1 or 2. The OGF
\begin{equation*}
C(z)=1+z+2z^2+3z^3+5z^4+\cdots=\frac{1}{1-z-z^2}.
\end{equation*}
\end{example}

\section{$D$-finiteness and $P$-recursive}

$D$-finiteness and $P$-recursive are important concepts in the asymptotics of $k$-noncrossing diagrams. In Chapter 3, we will use these two concepts to do the asymptotic of modular $k$-noncrossing diagrams.  Based on the paper of Stanley \cite{Stanley:80}, an arbitrary power series, $S(x)=\sum_{n\geq 0}s_n x^n$ is analytically continuated in any simply-connected domain containing zero avoiding singularities if $S(x)$ is $D$-finite.

\begin{definition}[\cite{Stanley:00}]\label{chapter2:D-finite}
(a) A sequence $f(n)$ of complex number is said to be $P$-recursive,
if there are polynomials $p_0(n),\ldots,p_m(n)\in \mathbb{C}[n]$
with $p_m(n)\neq0$, such that for all $n\in\mathbb{N}$
\begin{equation}\label{chapter2:p-recursive}
p_m(n)f(n+m)+p_{m-1}(n)f(n+m-1)+\cdots+p_0(n)f(n)=0.
\end{equation}
(b) A formal power series
$F(x)=\sum_{n\geq0}f(n)x^n$ is rational, if
there are polynomials $A(x)$ and $B(x)$ in $\mathbb{C}[x]$ with
$B(x)\neq0$, such that
\begin{equation*}
F(x)=\frac{A(x)}{B(x)}.
\end{equation*}
(c) $F(x)$ is algebraic of degree $m$, if there exist polynomials
$q_0(x),\ldots,q_m(x)\in \mathbb{C}[x]$
with $q_m(x)\neq0$, such that
\begin{equation*}
q_m(x)F^{m}(x)+q_{m-1}(x)F^{m-1}(x)+\cdots+q_1(x)F(x)+q_0(x)=0.
\end{equation*}
(d) $F(x)$ is $D$-finite,
if there are polynomials $q_0(x),\ldots,q_m(x)\in \mathbb{C}[x]$
with $q_m(x)\neq0$, such that
\begin{equation}\label{chapter2:D-finite-eq}
q_m(x)F^{(m)}(x)+q_{m-1}(x)F^{(m-1)}(x)+\cdots+q_1(x)F'(x)+q_0(x)F(x)=0,
\end{equation}
where $F^{(i)}(x)=d^iF(x)/dx^i$, and $\mathbb{C}[x]$ is the
ring of polynomials in $x$ with complex coefficients.
\end{definition}

Let $\mathbb{C}(x)$ denote the rational function field, i.e.~the
field generated by taking equivalence classes of fractions of
polynomials. Let $\mathbb{C}_{\text{\rm alg}}[[x]]$ and
$\mathcal{D}$ denote the sets of algebraic power series over
$\mathbb{C}(x)$ and $D$-finite power series, respectively. The four definitions above actually can imply the following theorem.

\begin{theorem}[\cite{Stanley:00}]
\begin{enumerate}
\item A rational power series is algebraic.
\item Let $F\in \mathbb{C}[[x]]$ be algebraic of degree $m$. Then $F$ is $D$-finite.
\item Let $F=\sum_{n\geq 0}{f(n)x^n}\in \mathbb{C}[[x]]$, then $F$ is $D$-finite if and only if $f$ is $P$-recursive.
\end{enumerate}
\end{theorem}
\begin{proof}
\begin{enumerate}
\item Here we use the notations of Definition \ref{chapter2:D-finite} (b). Let $q_0(x)=-A(x)$, $q_1(x)=B(x)$, then we have $B(x)F(x)-A(x)=0$. Therefore $F$ is algebraic.
\item Using the notations of Definition \ref{chapter2:D-finite} (c), let $P(x,y)\in \mathbb{C}[x,y]$ and nonzero,
    \begin{equation}
    P(x,F)=q_m(x)F^{m}(x)+q_{m-1}(x)F^{m-1}(x)+\cdots+q_1(x)F(x)+q_0(x),
    \end{equation}
    then we have
    \begin{equation}
    0=\frac{\dif}{\dif x}P(x,F)=\frac{\partial P(x,y)}{\partial x}\big|_{y=F}+F'\frac{\partial P(x,y)}{\partial y}\big|_{y=F}.
    \end{equation}
    It therefore follows that
    \begin{equation}\label{chapter2:algebratoDf}
    F'=-\frac{\frac{\partial P(x,y)}{\partial x}\big|_{y=F}}{\frac{\partial P(x,y)}{\partial y}\big|_{y=F}}\in \mathbb{C}(x,F).
    \end{equation}
    Continually differentiating eq. (\ref{chapter2:algebratoDf}) with respect to $x$ shows by induction that $F^{(k)}\in\mathbb{C}(x,F)$ for all $k\geq0$. But $\dim_{\mathbb{C}(x)}\mathbb{C}(x,F)=m$, so $F, F', \ldots, F^{(m)}$ are linearly dependent over $\mathbb{C}(x)$, yielding an equation of the form eq. (\ref{chapter2:D-finite-eq}). The $D$-finiteness of $F$ is proved.
\item Suppose $F$ is $D$-finite, so that eq. (\ref{chapter2:D-finite-eq}) holds (with $q_m\neq 0$). Since
    \begin{equation}
    x^j F^{(i)}=\sum_{n\geq 0}(n+i-j)_i f(n+i-j)x^n,
    \end{equation}
    when we equate coefficients of $x^{n+k}$ in eq. (\ref{chapter2:D-finite-eq}) for fixed $k$ sufficiently large, we will obtain a recurrence of the form eq. (\ref{chapter2:p-recursive}). This recurrence will not collapse to $0=0$ because if $[x^j]q_m(x)\neq 0$, then $[n^d]p_{m-j+k}(n)\neq 0$.
    Conversely, suppose that $f$ satisfies eq. (\ref{chapter2:p-recursive}) with $p_m(n)\neq 0$. For fixed $i\in \mathbb{N}$, the polynomials $(n+i)_j$, $j\geq 0$, form a $\mathbb{C}$-basis for the space $\mathbb{C}[n]$ (since $\deg(n+i)_j)=j$). Thus $p_i(n)$ is a $\mathbb{C}$-linear combination of the polynomials $(n+i)_j$, so $\sum_{n\geq 0}p_i f(n+i)x^n$ is a $\mathbb{C}$-linear combination of series of the form $\sum_{n\geq 0}(n+i)_j f(n+i)x^n$. Now
    \begin{equation}
    \sum_{n\geq 0}(n+i)_j f(n+i)x^n=R_i(x)+x^{j-i}F^{(j)}
    \end{equation}
    for some $R_i(x)\in x^{-1}\mathbb{C}[x^{-1}]$ (i.e. $R_i(x)$ is a Laurent polynomial all of whose exponents are negative). Hence multiplying eq. (\ref{chapter2:p-recursive}) by $x^n$ and summing on $n\geq 0$ yields
    \begin{equation}\label{chapter2:p-recursive-proof}
    0=\sum a_{ij}x^{j-i}F^{(j)}+R(x),
    \end{equation}
    where the sum is finite, $a_{ij}\in \mathbb{C}$ and $R(x)\in x^{-1}\mathbb{C}[x^{-1}]$. One easily sees that not all $a_{ij}=0$. Now multiply eq. (\ref{chapter2:p-recursive-proof}) by $x^q$ for $q$ sufficiently large to get an equation of the form eq. (\ref{chapter2:D-finite-eq}).
    \end{enumerate}
\end{proof}

For easily understanding, we present an example.

\begin{example}
The {\it Motzkin number} $M_n$ are defined by
\begin{align*}
M(x)&=\sum_{n\geq 0}M_n x^n=\frac{1-x-\sqrt{1-2x-3x^2}}{2x^2}\nonumber\\
&=1+x+2x^2+4x^3+9x^4+21x^5+51x^6+\cdots,
\end{align*}
$M(x)$ is algebraic and $D$-finite since
\begin{equation*}
x^2 M(x)^2+(x-1)M(x)+1=0
\end{equation*}
and
\begin{equation*}
(6x+3)M(x)+(12x^2+7x-3)M'(x)+x(x+1)(3x-1)M''(x)=0
\end{equation*}
which implies $M_n$ is $P$-recursive since
\begin{equation*}
-(n^2+4n+3)M_{n+1}+(2n^2+5n+3)M_n+(3n^2+3n)M_{n-1}=0
\end{equation*}
\end{example}

We proceed by studying closure properties of $D$-finite power series which are of key important in the following chapters.

\begin{theorem}\label{T:d-finite_property}\cite{Stanley:00}
$P$-recursive sequences, $D$-finite and algebraic power
series\index{algebraic!power series}
have the following properties:\\
(a) If $f,g$ are $P$-recursive, then $f\cdot g$ is $P$-recursive.\\
(b) If $F,G\in\mathcal{D}$, and $\alpha,\beta\in \mathbb{C}$,
then $\alpha F+\beta G\in\mathcal{D}$ and $FG\in\mathcal{D}$. \\
(c) If $F\in\mathcal{D}$ and $G\in\mathbb{C}_{alg}[[x]]$ with
$G(0)=0$, then $F(G(x))\in\mathcal{D}$.
\end{theorem}
Here we omit the proof of (a) and (b) which can be found in
\cite{Stanley:00}. We present however a direct proof of
(c).
\begin{proof}
(c) We assume that $G(0)=0$ so that the composition $F(G(x))$ is
well-defined. Let $K=F(G(x))$. Then $K^{(i)}$ is a linear
combination of $F(G(x)),\,F'(G(x)),\,\ldots$ over
$\mathbb{C}[G,\,G',\ldots]$, i.e.~the ring of polynomials in
$G,\,G',\ldots$ with complex coefficients. \\
{\it Claim.} $G^{(i)}\in \mathbb{C}(x,G),\ i\geq 0$ and therefore
$\mathbb{C}[G,\,G',\ldots]\subset \mathbb{C}(x,G)$, where $\mathbb{C}(x,G)$
denotes the field generated by $x$ and $G$.\\
Since $G$ is algebraic, it satisfies
\begin{equation}\label{E:proof_alg}
q_d(x)G^{d}(x)+q_{d-1}(x)G^{d-1}(x)+\cdots+q_1(x)G(x)+q_0(x)=0,
\end{equation}
where $q_0(x),\ldots,q_d(x)\in \mathbb{C}[x],\ q_d(x)\neq0$ and $d$
is minimal, i.e.~$(G^{i}(x))_{i=0}^{d-1}$ is linear
independent over $\mathbb{C}[x]$. In other words, for all
$(\widetilde{q}_{i}(x))_{i=1}^{d-1}\neq 0$ we have
\begin{equation*}
\widetilde{q}_{d-1}(x)G^{d-1}(x)+\cdots+\widetilde{q}_1(x)G(x)+
\widetilde{q}_0(x)\neq 0.
\end{equation*}
We consider
\begin{equation*}
P(x,G)=q_d(x)G^{d}(x)+q_{d-1}(x)G^{d-1}(x)+\cdots+q_1(x)G(x)+q_0(x).
\end{equation*}
Differentiating eq.~(\ref{E:proof_alg}) once, we derive
\begin{equation*}
0=\frac{d}{dx}P(x,G)=\frac{\partial P(x,y)}{\partial x}\Big
|_{y=G}+G'\frac{\partial P(x,y)}{\partial y}\Big|_{y=G}.
\end{equation*}
The degree of $\frac{\partial P(x,y)}{\partial y}\Big|_{y=G}$ in $G$
is $d-1$ and $q_d(x)\neq0$, whence $\frac{\partial
P(x,y)}{\partial y}\Big|_{y=G}\neq 0$. We therefore arrive at
\begin{equation*}
G'=-~\frac{\frac{\partial P(x,y)}{\partial x}\Big
|_{y=G}}{\frac{\partial P(x,y)}{\partial y}\Big|_{y=G}}\,\in
\mathbb{C}(x,G).
\end{equation*}
Iterating the above argument, we obtain $G^{(i)}\in
\mathbb{C}(x,G),\ i\geq 0$ and therefore
$\mathbb{C}[G,\,G',\ldots]\subset \mathbb{C}(x,G)$, whence the
Claim.\\
Let $\widetilde{V}$ be the $\mathbb{C}(x,G)$ vector space spanned by
$F(G(x)),\,F'(G(x)),\,\ldots$. Since  $F\in\mathcal{D} $, we have
$\dim_{\mathbb{C}(x)}\langle F,F',\cdots\rangle<\infty$, immediately
implying the finiteness of $\dim_{\mathbb{C}(G)}\langle F(G),F'(G),
\cdots\rangle$. Thus, since $\mathbb{C}(G)$ is a subfield of
$\mathbb{C}(x,G)$, we derive
\begin{equation*}
\dim_{\mathbb{C}(x,G)}\langle F,F',\cdots\rangle <\infty
\end{equation*}
and consequently $ \dim_{\mathbb{C}(x,G)} \widetilde{V}<\infty$ and
$ \dim_{\mathbb{C}(x)}\mathbb{C}(x,G)<\infty$. As a result
\begin{equation*}
\dim_{\mathbb{C}(x)} \widetilde{V}=\dim_{\mathbb{C}(x,G)}
\widetilde{V}\cdot \dim_{\mathbb{C}(x)}\mathbb{C}(x,G)<\infty
\end{equation*}
follows and since each $K^{(i)}\in \widetilde{V}$, we conclude that
$F(G(x))$ is $D$-finite.
\end{proof}

Assume ${\bf F}_k=\sum_{n\geq 0}f_k(2n,0)x^n$, where $f_k(2n,0)$ is the number of $k$-noncrossing matching with $n$ arcs, or equivalently, $2n$ vertices. In Chen's paper \cite{Chen}, we have known that the exponential generating function of $k$-noncrossing matching is
\begin{equation}\label{chapter2:matching-gf}
\sum_{n\geq 0}f_k(2n,0)\frac{x^{2n}}{(2n)!}=\det [I_{i-j}(2x)-I_{i+j}(2x)]|^{k-1}_{i,j=1},
\end{equation}
where $I_m(x)$ is hyperbolic Bessel function of the first kind of order $m$. Particularly, when $k=3$, $f_k(2n,0)$ has a formula of general term, which is
\begin{equation*}
f_k(2n,0)=C(n)C(n+2)-C(n+1)^2,
\end{equation*}
where $C(n)$ is the $n$-th Catalan number. The generating function of $3$-noncrossing matching can be explicitly written by
\begin{equation}\label{E:F3-gf}
{\bf F}_3(x)=\frac{1}{4x^2}\left(1+6x-_2F_1(-1/2,-1/2;1;16x)-2x\cdot\ _2 F_1(-1/2,3/2;3;16x)\right),
\end{equation}
where $_2F_1(a,b;c;z)$ is hypergeometric function. Before the further discussion in the following chapters, we have to confirm the D-finiteness of ${\bf F}_k$.

\begin{corollary}\label{C:Bessel}
The generating function of $k$-noncrossing
matchings over $2n$ vertices, {${\bf
F}_k(z)=\sum_{n\ge 0}f_k(2n,0)\,z^{n}$} is $D$-finite.
\end{corollary}
\begin{proof}
Corollary~\ref{C:Bessel} gives the exponential generating function
of eq. (\ref{chapter2:matching-gf}). Recall that the Bessel function of the first kind
satisfies $I_n(x)=i^{-n}J_n(ix)$ and $J_n(x)$ is the solution of the
Bessel differential equation
\begin{equation*}
x^2\frac{d^2y}{d x^2}+x\frac{d y}{d x}+(x^2-n^2)y=0.
\end{equation*}
For every fixed $n\in\mathbb{N}$, $J_n(x)$ is $D$-finite.  Let
$G(x)=ix$. Clearly, $G(x)\in \mathbb{C}_{\text{\rm alg}}[[x]]$ and
$G(0)=0$, $J_n(ix)$ and $I_n(x)$ are accordingly $D$-finite in view
of the assertion (c) of Theorem~\ref{T:d-finite_property}.
Analogously we show that $I_n(2x)$ is $D$-finite for every fixed
$n\in\mathbb{N}$. Using eq.~(\ref{chapter2:matching-gf}) and assertion (b)
of Theorem~\ref{T:d-finite_property}, we conclude that
\begin{equation*}
{\bf H}_k(x)=\sum_{n\ge 0}\frac{f_k(2n,0)}{(2n)!}\,x^{2n}
\end{equation*}
is $D$-finite. In other words the sequence
$f(n)=\frac{f_k(2n,0)}{(2n)!}$ is $P$-recursive and furthermore
$g(n)=(2n)!$ is, in view of $(2n+1)(2n+2)g(n)-g(n+1)=0$,
$P$-recursive. Therefore, $f_k(2n,0)=f(n)g(n)$ is $P$-recursive.
This proves that {${\bf F}_k(z)=\sum_{n\geq 0}f_k(2n,0)z^{n}$} is
$D$-finite.
\end{proof}

\section{Singularity analysis of generating functions}

This section is meant to introduce the basic technology of singularity analysis, which is largely of a methodological nature. Precisely, the method of singularity analysis applies to functions whose singular expansion involves fractional powers and logarithms--one sometimes refers to such singularities as ``algebraic-logarithmic". It centrally relies on two ingredients.
\begin{itemize}
\item A {\it catalogue} of asymptotic expansions for coefficients of the standard functions that occur in such singular expansion.
\item {\it Transfer theorems}, which allow us to extract the asymptotic order of coefficients of error terms in singular expansions.
\end{itemize}

The singularity analysis are based on Cauchy's coefficient formula, used in conjunction with special contours of integration known as {\it Hankel contours}. The contours come very close to the singularities then go around a circle with a sufficient large radius: by design, they  capture essential asymptotic information contained in the functions' singularities.

\subsection{Basic singularity analysis theory}

Before discussing about the singularity analysis of general functions, we first have a glimpse of basic singularity analysis theory. We consider functions whose expansion at a singularity $\eta$ involves elements of the form
\begin{equation*}
\left(1-\frac{z}{\eta}\right)^{-\alpha}\left(\log \frac{1}{1-\frac{z}{\eta}}\right)^\beta.
\end{equation*}

Under suitable conditions to be discussed in detail in this section, any such element contributes a term of the form
\begin{equation*}
\eta^{-n}n^{\alpha-1}(\log n)^\beta,
\end{equation*}
where $\alpha$ and $\beta$ are arbitrary complex numbers.

Virtually, discussing a function with a singularity $\eta$ is equivalently discussing a function which has a singularity at $1$. Indeed, if $f(z)$ is singular at $z=\eta$, then $g(z)=f(z\eta)$ satisfies, by scaling rule of Taylor expansion,
\begin{equation*}
[z^n]f(z)=\eta^{-n}[z^n]f(z\eta)=\eta^{-n}[z^n]g(z),
\end{equation*}
where $g(z)$ is now singular at $z=1$. Based on this fact, we will focus on the function which has a singularity at $1$.

\restylefloat{figure}\begin{figure}[h!t!b!p]
\begin{center}
\includegraphics{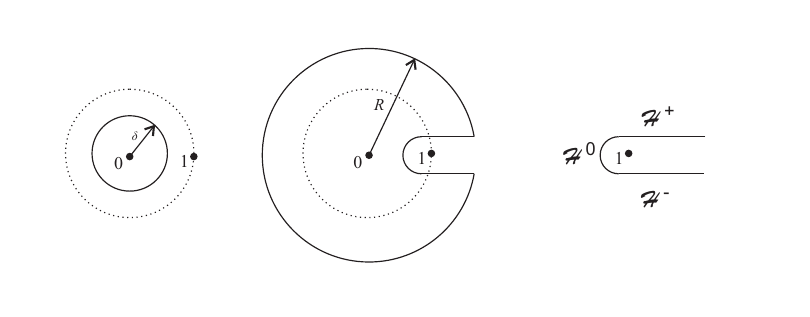}
\end{center}
\caption{Three different contours, $\mathcal{C}_0$, $\mathcal{C}_1$ and $\mathcal{C}_2\equiv \mathcal{H}(n)$ (from left to right) used for estimating the coefficients of functions from the standard function scale.\label{F:contour}}
\end{figure}

\begin{theorem}[\cite{Flajolet:90}]\label{T:standard-function-scale}
Let $\alpha$ be an arbitrary complex number in $\mathbb{C}\setminus \mathbb{Z}_{\leq 0}$. The coefficient of $z^n$ in
\begin{equation}\label{chapter2:standard-function-scale}
f(z)=(1-z)^{-\alpha}
\end{equation}
admits for large $n$ a complete asymptotic expansion in descending powers of $n$,
\begin{equation*}
[z^n]f(z)\sim
\frac{n^{\alpha-1}}{\Gamma(\alpha)}\left(1+\sum_{k=1}^\infty
\frac{e_k}{n^k}\right),
\end{equation*}
where $e_k$ is a polynomial in $\alpha$ of degree $2k$. In
particular:
\begin{align}\label{result1}
[z^n]f(z)\sim\frac{n^{\alpha-1}}{\Gamma(\alpha)}\bigg(1+&\frac{\alpha
(\alpha-1)}{2n}+\frac{\alpha(\alpha-1)(\alpha-2)(3\alpha-1)}{24n^2}+\cdots\bigg).
\end{align}
\end{theorem}

\begin{proof}
The first step is to express the coefficient $[z^n](1-z)^{-\alpha}$
as a complex integral by means of Cauchy's coefficient formula,
\begin{equation}\label{coef}
f_n=\frac1{2\pi \i}\int_{\mathcal{C}}(1-z)^{-\alpha}\frac{\dif z}{z^{n+1}},
\end{equation}
where $\mathcal{C}$ is a small enough contour that encircles the
origin, see Figure \ref{F:contour}. We can start with
$\mathcal{C}\equiv\mathcal{C}_0$, where $\mathcal{C}_0$ is the
positively oriented circle $\mathcal{C}_0=\{z:|z|=\frac1{2}\}$. The
second step is to deform $C_0$ into another simple closed curve
$\mathcal{C}_1$ around the origin that does not cross the half-line
$\text{Re}(z)\geq 1$: the contour $\mathcal{C}_1$ consists of a large
circle $\mathcal{R}$ of radius $R>1$ with a notch that comes back
near and to the left of $z=1$. Since
\begin{align*}
\left|\int_{\mathcal{R}}(1-z)^{-\alpha}\frac{\dif z}{z^{n+1}}\right|&\leq \int_{\mathcal{R}}|(1-z)^{-\alpha}|\frac{|\dif z|}{|z|^{n+1}}\\
&=\frac{1}{R^{n+1}}\int_{\mathcal{R}}|1-z|^{\re(-\alpha)}\exp^{-\im(-\alpha)\cdot\arg(1-z)}|\dif z|\\
&\leq \frac{1}{R^{n+1}}M R^{\re(-\alpha)}\int_{\mathcal{R}}|dz|\\
&\leq\frac{2\pi M}{R^{n-\re(-\alpha)}},
\end{align*}
where $M$ is some positive constant, we can finally let
$R\rightarrow \infty$ and are left with an integral representation
for $f_n$ where $\mathcal{C}$ has been replaced by a contour
$\mathcal{C}_2$ that starts from $+\infty$ in the lower half-plane,
winds clockwise around 1, and ends at $+\infty$ in the upper
half-plane. The latter is a typical case of a {\it Hankel contour}
(see \cite{Flajolet:07a} Theorem B.1, p.745). a judicious choice of its
distance to the half-line $\mathbb{R}_{\geq 1}$ yields the
expansion.

To specify precisely the integration path, we particularize
$\mathcal{C}_2$ to be the contour $\mathcal{H}(n)$ that passes at a
distance $\frac1{n}$ from the half line $\mathbb{R}_{\geq 1}$:
\begin{equation}
\mathcal{H}(n)=\mathcal{H}^-(n)\cup \mathcal{H}^+(n)\cup \mathcal{H}^\circ (n),
\end{equation}
where
\begin{equation}
\begin{cases}
\mathcal{H}^-(n)\ =\ \{z=w-\frac{\i}{n}, w\geq 1\}\\
\mathcal{H}^+(n)\ =\ \{z=w+\frac{\i}{n}, w\geq 1\}\\
\mathcal{H}^\circ(n)\ =\ \{z=1-\frac{\exp^{\phi\i}}{n}, \phi\in[-\frac{\pi}{2},\frac{\pi}{2}]\}.
\end{cases}
\end{equation}
Now, a change of variable
\[
z=1+\frac{t}{n}
\]
in the integral (\ref{coef}) gives the form
\begin{equation}
f_n=\frac{n^{\alpha-1}}{2\pi \i}\int_{\mathcal{H}}(-t)^{-\alpha}\left(1+{t\over n}\right)^{-(n+1)}\dif t,
\end{equation}
(The Hankel contour $\mathcal{H}$ winds about 0, being at distance 1
from the positive real axis; it is the same as the one in the proof
of Theorem B.1, p.745 \cite{Flajolet:07a}.)

Now we split the contour $\mathcal{H}$ according to $\re(t)\leq
\log^2n$ and $\re(t)> \log^2n$ (see Figure \ref{F:split}), and we set
\[
\mathcal{H}_1=\mathcal{H}\cap \{\re(t)\leq
\log^2n\}, \mathcal{H}_2=\mathcal{H}\cap \{\re(t)>
\log^2n\}.
\]
\restylefloat{figure}\begin{figure}[h!t!b!p]
\begin{center}
\includegraphics{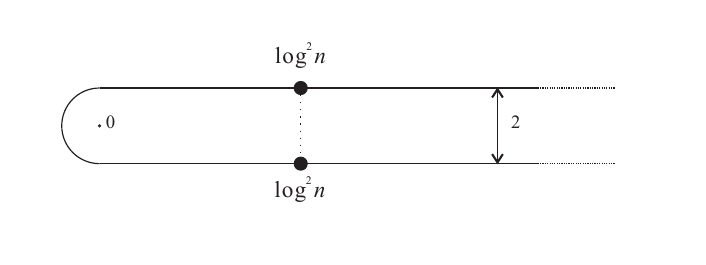}
\end{center}
\caption{The splitting of $\mathcal{H}$.\label{F:split}}
\end{figure}

For $\mathcal{H}_2$, let $k$ be a sufficient large positive integer
which make $\re(-\alpha)-k<-1$, then
\begin{align*}
T &\equiv\left|\frac{n^{\alpha-1}}{2\pi \i}\int_{\mathcal{H}_2}(-t)^{-\alpha}\left(1+{t\over n}\right)^{-(n+1)}\dif t\right|\\
&\leq \frac{|n^{\alpha-1}|}{2\pi \i}\int_{\mathcal{H}_2}|(-t)^{-\alpha-k}|\cdot\left|t^k\left(1+{t\over n}\right)^{-(n+1)}\right||\dif t|\\
&=\frac{|n^{\alpha-1}|}{2\pi \i}\int_{\mathcal{H}_2}\exp ^{-\im (-\alpha)\arg(-t)}|t|^{\re(-\alpha)-k}\cdot n^k\cdot\left|\frac{\left({t\over n}\right)^k}{\left(1+{t\over n}\right)^{\xi(n+1)}}\right|\cdot
\left|\left(1+{t\over n}\right)^{-(1-\xi)(n+1)}\right||\dif t|
\end{align*}
where $\xi$ is a real number which satisfies ${k\over{n+1}}<\xi<1$,
and we know it is possible when $n$ is big enough.

Explicitly, $\exp ^{-\im (-\alpha)\arg (-t)}$ is bounded. We also notice that whatever $\left(t\over n\right)$ is bounded or not, when
$n\rightarrow \infty$,
\begin{align*}
\frac{\left({t\over n}\right)^k}{\left(1+{t\over n}\right)^{\xi(n+1)}}=O(1)
\end{align*}
is always true. Moreover, we have
\[
\left(1+{t\over n}\right)^{-(1-\xi)(n+1)}=O(\exp^{-(1-\xi)\log^2 n}),\ n\rightarrow \infty,
\]
then
\begin{align*}
T\leq M'\frac{|n^{\alpha-1}|\cdot n^k}{2\pi}\exp^{-(1-\xi)\log^2 n}\int_{\mathcal{H}_2}|t|^{\re(-\alpha)-k}|\dif t|
\end{align*}

 where $M'$ is some positive constant. The integration of the
formula above is obviously convergence, so $T$ tends to 0 when $n$
tends to infinity.

For $\mathcal{H}_1$, when $\re(t)\leq \log ^2n$,
$t^{m_1}/n^m=o(n^{-(m-1)})$, $m, m_1\geq 1$. Therefore $\left(1+{t\over n}\right)^{-n-1}$ has asymptotic expansion of form
\begin{equation}
\exp^{-t}\left(1+\frac{t^2-2t}{2n}+\frac{3t^4-20t^3+24t^2}{24n^2}+o\left(\frac1{n^2}\right)\right)\label{eterm}.
\end{equation}
According the discussion about $\mathcal{H}_1$ and $\mathcal{H}_2$
above, when $n$ tends to infinity, we have
\begin{align}
[z^n](1-z)^{-\alpha}=f_n&=\frac{n^{\alpha-1}}{2\pi \i}\int_{\mathcal{H}_1}(-t)^{-\alpha}\left(1+{t\over n}\right)^{-(n+1)}\dif t
\nonumber\\
&+\frac{n^{\alpha-1}}{2\pi \i}\int_{\mathcal{H}_2}(-t)^{-\alpha}\left(1+{t\over n}\right)^{-(n+1)}\dif t\quad(\text{$\mathcal{H}_1$ tends to $\mathcal{H}$})\nonumber\\
 \sim\,\frac{n^{\alpha-1}}{2\pi \i}&\left[\int_{\mathcal{H}}(-t)^{-\alpha}\exp^{-t}\left(1+\frac{t^2-2t}{2n}+\frac{3t^4-20t^3+24t^2}{24n^2}\right)
 \dif t+o\left(\frac1{n^2}\right)\right]\label{E:termcomp}
\end{align}
Applying the equation
\begin{equation*}
\frac1{2\pi \i}\int_{\mathcal{H}}(-t)^{-\alpha}t^r\exp^{-t}\dif t=\frac{(-1)^r}{\Gamma(\alpha-r)}=\frac1{\Gamma(\alpha)}(1-\alpha)
(2-\alpha)\cdots(r-\alpha)
\end{equation*}
to eq.~(\ref{E:termcomp}), we can easily get eq.~(\ref{result1}). The proof is complete.
\end{proof}

If we add a logarithmic factor for $f(z)$ in eq.~(\ref{chapter2:standard-function-scale}), we will have the following asymptotic theorem.

\begin{theorem}[\cite{Flajolet:07a}]\label{chapter2:logarithm-asym}
$\alpha$ and $\beta$ are respectively negative and positive integer. The coefficient of $z^n$ in the function
\begin{equation}
f(z)=(1-z)^{-\alpha}\left(\log{1\over{1-z}}\right)^\beta
\end{equation}
admits for large $n$ a full asymptotic expansion of the form
\begin{equation}\label{chapter2:logasym}
f_n\equiv [z^n]f(z)\sim n^{\alpha-1}\left[F_0(\log n)+\frac{F_1(\log n)}{n}+\cdots\right],
\end{equation}
where $F_j$ is a polynomial with degree of $\beta-1$.
\end{theorem}

The proof of Theorem \ref{chapter2:logarithm-asym} has been present in \cite{Flajolet:07a}.

\subsection{Transfers}

In this subsection our general objective is to translate an approximation of a function near a singularity into an asymptotic approximation of its coefficients.

In previous subsection, we have noticed that the functions we want to approximate is analytic in the complex plane slit along the real half line $\mathbb{R}_{\geq 1}$. As a matter of fact, there exists a  weaker condition sufficing that any domain whose boundary makes an acute angle with the half line $\mathbb{R}_{\geq 1}$ appears to be suitable.

\begin{definition}[\cite{Flajolet:07a}]
Given two numbers $\phi, R$ with $R > 1$ and $0 < \phi <
{\pi\over2}$, the open domain $\Delta(\phi, R)$ is defined as
\[
\Delta(\phi, R) = \left\{z:
|z| < R, z \neq 1, | \arg(z - 1)| > \phi\right\}.
\]

A domain is a $\Delta$-domain at 1 if it is a $\Delta(\phi, R)$ for
some $R$ and $\phi$ (see Figure \ref{F:delta-contour}). For a complex
number $\zeta \neq 0$, a $\Delta$-domain at $\zeta$ is the image by
the mapping $z \mapsto \zeta z$ of a $\Delta$-domain at 1. A
function is $\Delta$-analytic if it is analytic in some
$\Delta$-domain.
\end{definition}

\restylefloat{figure}\begin{figure}[h!b!t!p]
\centering
\includegraphics{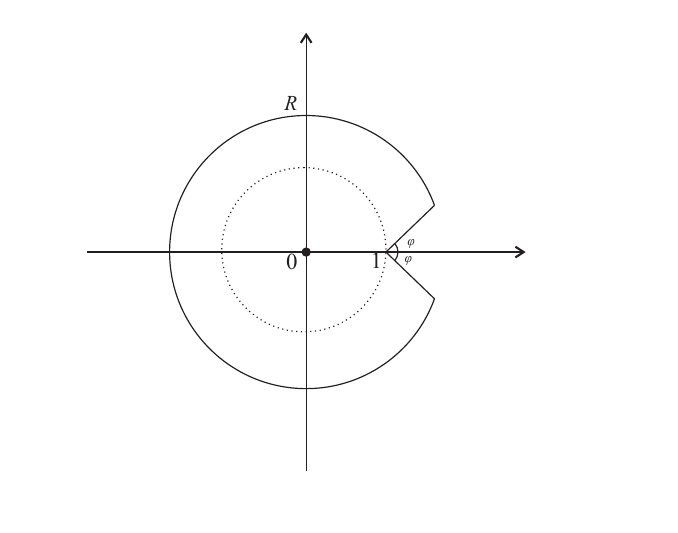}
\caption{$\Delta$-domain: The solid line is the edge of $\Delta$-domain and the dash line represents the edge of convergence domain.\label{F:delta-contour}}
\end{figure}

Based on the analyticity in a $\Delta$-domain, we have the following transfer theorem, which is also the key theorem in the asymptotic part of this thesis.

\begin{theorem}[Transfer\cite{Flajolet:07a}]\label{chapter2:transfer}
Let $\alpha, \beta$ be arbitrary real numbers, $\alpha, \beta \in R$
and let $f (z)$ be a function that is $\Delta$-analytic.
Assume that $f (z)$ satisfies in the intersection of a neighborhood of 1 with its
$\Delta$--domain the condition
\begin{equation}
f(z)=o\left((1-z)^{-\alpha}\left(\log{1\over {1-z}}\right)^\beta\right).
\end{equation}
Then one has: $[z^n]f(z)=o(n^{\alpha-1}(\log n)^\beta)$.
\end{theorem}

An immediate corollary of Theorem \ref{chapter2:logarithm-asym} and Theorem \ref{chapter2:transfer} is the possibility of transferring asymptotic equivalence from singular forms to coefficients:

\begin{corollary}[sim-transfer\cite{Fusy:10}]\label{chapter2:sim-transfer}
Assume that $f(z), g(z)$ are all $\Delta$-analytic at $1$ and
\begin{equation*}
f(z)\sim (1-z)^{-\alpha},\ g(z)=(z-1)^{-r}\log (z-1),\qquad \text{as $z\rightarrow 1, z\in \Delta$},
\end{equation*}
with $\alpha\not\in \mathbb{Z}_{\leq 0}$ and $r\in\mathbb{Z}_{\leq 0}$. Then the coefficients of $f$ and $g$ satisfy
\begin{align*}
[z^n]f(z)&\sim \frac{n^{\alpha-1}}{\Gamma(\alpha)},\\
[z^n]g(z)&\sim -(-r)!\ n^{r-1}.
\end{align*}
\end{corollary}

\subsection{The process of singularity analysis}

We let $\mathcal{S}$ denote the set of such singular functions:
\begin{equation}
\mathcal{S}=\{(1-z)^{-\alpha}\lambda(z)^\beta\mid \alpha, \beta\in\mathbb{C}\},\quad \lambda(z):=\log\frac{1}{1-z}.
\end{equation}

According to the previous subsection, every element in $\mathcal{S}$ is available to do the asymptotic estimate under suitable conditions. Generally, if we have an expansion of a function at its singularity (also called {\it singular expansion}) which is composed by the elements of $\mathcal{S}$, we can do the approximation for the coefficients of this function. We state the following theorem \cite{Flajolet:07a}.

\begin{theorem}[Singularity analysis, single singularity]\label{chapter2:singularity-analysis}
Let $f (z)$ be function analytic at $0$ with a singularity at $\zeta$
, such that $f (z)$ can be continued to a domain of the form
$\zeta\cdot\Delta_0$ , for a  $\Delta$-domain $\Delta_0$, where
$\zeta\cdot\Delta_0$ is the image of $\Delta_0$ by the mapping $z
\rightarrow \zeta z$. Assume that there exist two functions
$\sigma$, $\tau$ , where $\sigma$ is a (finite) linear combination
of functions in $\mathcal{S}$ and $\tau\in\mathcal{S}$, so that
\begin{equation}\label{chapter2:singular-expansion}
f(z)=\sigma(z/\zeta)+o(\tau(z/\zeta))\qquad \text{as $z\rightarrow\zeta$ in
$\zeta\cdot\Delta_0$}.
\end{equation}
Then, the coefficients of $f(z)$ satisfy the asymptotic estimate
\begin{equation}\label{chapter2:asym-coef}
[z^n]f(z)=\zeta^{-n}[z^n]\sigma(z)+o(\zeta^{-n}\tau_n^\star),
\end{equation}
where $\tau_n^\star=n^{a-1}(\log n)^b$, if
$\tau(z)=(1-z)^{-a}\lambda(z)^b$, $\lambda(z)=z^{-1}\log(1-z)^{-1}$.
\end{theorem}

Theorem \ref{chapter2:singularity-analysis} gives us a general process of singularity analysis for a generating function, $f(z)$, progressing in following three steps.
\begin{enumerate}
\item {\it Preparation.} Locate the singularity of $\zeta$ of $f(z)$, and check if $f(z)$ is analytic in some domain of the form $\zeta\Delta_0$.
\item {\it Singular expansion.} Analyze $f(z)$ as $z\rightarrow\eta$ in the domain $\zeta\Delta_0$ and determine in that domain an expansion of the form as eq. (\ref{chapter2:singular-expansion}).
\item {\it Transfer.} Using Theorem \ref{chapter2:singularity-analysis} to determine the asymptotic estimate of the $n$-th coefficient of the generating function $f(z)$, as eq. (\ref{chapter2:asym-coef}).
\end{enumerate}

\subsection{Supercritical paradigm}

In our derivations the following instance of the supercritical paradigm
\cite{Flajolet:07a} is of central importance: we
are given a function, $f(z)$, which admits a log-power expansion
, $\Delta$-analytic and an algebraic function
$g(u)$ satisfying $g(0)=0$. Furthermore we suppose that $f(g(u))$
has a unique real valued dominant singularity $\gamma$ and $g$ is
regular in a disc with radius slightly larger than $\gamma$. Then
the supercritical paradigm stipulates that the subexponential
factors of $f(g(u))$ at $u=0$, given that $g(u)$ satisfies certain
conditions, coincide with those of $f(z)$.

Theorem \ref{chapter2:standard-function-scale}, Corollary \ref{chapter2:sim-transfer} and
Theorem~\ref{chapter2:singularity-analysis} allow under certain conditions to obtain the
asymptotic estimate of the coefficients of supercritical compositions of the
``outer'' function $f(z)$ and ``inner'' function $\psi(z)$.

Suppose $f(z)$ has single singularity at $\rho$ and the singular expansion is given by
\begin{equation}\label{E:fx-sing-exp}
f(z)=\sum_{i=0}^{k} a_i (z-\rho)^{i}+a_k(z-\rho)^k\log(z-\rho)+o((z-\rho)^k),\quad a_i\in\mathbb{C},
\end{equation}
then we have the following proposition.
\begin{proposition}\label{chapter2:algeasym}
Let $\psi(z)$ be an analytic function in a
domain $\mathcal{D}=\{z||z|\leq r\}$ such that
$\psi(0)=0$. Suppose $\gamma$ is the unique
dominant singularity of $f(\psi(z))$ and minimum positive real
solution of $\psi(\gamma)=\rho$, $|\gamma|<r$. $|\psi(z)|\leq \rho$ for $|z|\leq \gamma$
and $\psi'(\gamma)>0$.
Then $f(\psi(z))$ has a singular expansion and
\begin{equation}
[z^n]f(\psi(z)) \sim A\,n^{-k-1}
\left(\frac{1}{\gamma}\right)^n,
\end{equation}
where $A$ is some constant.
\end{proposition}

\begin{proof}
Here we only present the proof of $\Delta$-analyticity of $f(\psi(z))$. The asymptotic part is totally the same as the proof of Theorem 3.21 in \cite{Reidys:10book} when $s=0$.

Suppose the outer function, $f(w)$ is analytic in a $\Delta$-domain defined as
\begin{equation*}
\Delta_f=\{w\mid |w|<\lambda, w\neq \rho, |\arg(w-\rho)|>\phi\},\quad \text{where $\lambda>\rho$}.
\end{equation*}

Step {\sf (a)}. {\it Prove that $f(\psi(z))$ is analytic in the domain of $\{z\mid |z|\leq\gamma\}-\{\gamma\}$.}
Since $\gamma$ is the unique dominant singularity, $\psi(z)=\rho$ if and only if $z=\gamma$ when $|z|\leq \gamma$. According to the condition of $|\psi(z)|\leq \rho$ for $|z|\leq \gamma$, we easily conclude that $\psi(z)$ lies in $\Delta_f$ for any $z$ satisfying $|z|\leq \gamma$ without $z=\gamma$, whence the claim {\sf (a)}.

Step {\sf (b)}. {\it Prove that there exists a domain around $z=\gamma$, defined as
\begin{equation}\label{E:gamma-delta}
\tilde{\Delta}_\gamma=\{z\mid |z-\gamma|<\delta, \arg|z-\gamma|>\beta, 0<\beta<\frac{\pi}{2}\},\quad \text{where $\delta<r-\gamma$}
\end{equation}
satisfying $\psi(\tilde{\Delta}_\gamma)\subset \Delta_f$.
}

Due to the supercritical paradigm of $f(\psi(z))$, $\psi(z)$ is analytic at $\gamma$. Therefore we have Taylor expansion,
\begin{equation}\label{E:psi-expansion}
\psi(z)=\rho+\psi'(\gamma)(z-\gamma)+e(z)
\end{equation}
where $e(z)$ is the error term satisfying $e(z)=o(z-\gamma)$.

Now we begin to construct a domain presented in eq. (\ref{E:gamma-delta}).

\restylefloat{figure}\begin{figure}[h!t!p!b]
\centering
\includegraphics{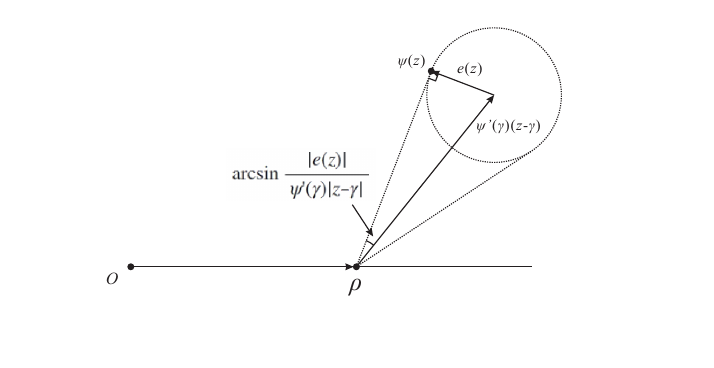}
\caption{The figure shows us the range of argument of $\psi(z)-\rho$ when $z$ is in a sufficient small neighborhood of $\gamma$.\label{F:delta-neigh}}
\end{figure}


 We give a positive real number $\varepsilon$ which is as small as $\varepsilon<\min\{10^{-2},\lambda-\rho\}$, then there exists $\delta$ satisfying
\begin{equation}\label{E:psi-ep}
|\psi(z)-\rho|<\varepsilon
\end{equation}
 and
\begin{equation}\label{E:gamma-epsilon}
|e(z)|\leq \varepsilon \psi'(\gamma)|z-\gamma|
\end{equation}
for all $|z-\gamma|<\delta$.

As a result of very small quantity of $\varepsilon$, we have the following inequality concluded from eq. (\ref{E:psi-expansion}),
\begin{equation}\label{E:gamma-ineq}
\arg(z-\gamma)-\theta\leq\arg(\psi(z)-\rho)\leq \arg(z-\gamma)+\theta,
\end{equation}
for all $|z-\gamma|<\delta$, where $\theta=\arcsin\left|\frac{e(z)}{\psi'(\gamma)(z-\gamma)}\right|\leq\arcsin\varepsilon$. The inequality is easily approved geometrically, see Fig.~\ref{F:delta-neigh}. Picking up $\beta$ sufficing $\phi+\theta<\beta<\pi/2$, then eq. (\ref{E:gamma-ineq}) implies
\begin{equation}\label{E:arg-psi}
|\arg(\psi(z)-\rho)|>\phi
\end{equation}
when $|\arg(z-\gamma)|>\beta$.
According to eq. (\ref{E:psi-ep}) and eq. (\ref{E:arg-psi}), $\psi(z)$ lies in $\Delta_f$ if $|z-\gamma|<\delta$ and $|\arg(z-\gamma)|>\beta$. $\tilde{\Delta}_\gamma$ in eq. (\ref{E:gamma-delta}) is settled. Hence the claim {\sf (b)} is true.

\restylefloat{figure}\begin{figure}[h!t!b!p]
\centering
\scalebox{0.8}[0.8]{\includegraphics{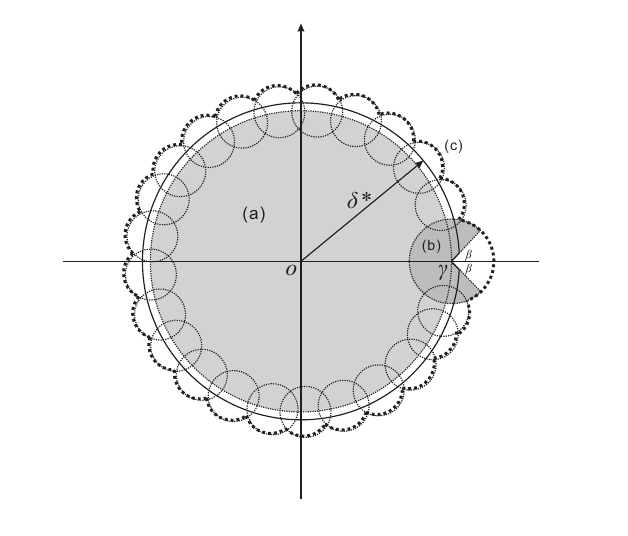}}
\caption{Three steps to generate a $\Delta$-domain for $f(\psi(z))$: Step (a) is proving $\psi(z)$ is constraint in the circle of radius $\rho$ when $|z|\leq \gamma$. Step (b) is creating a split domain $\tilde{\Delta}_\gamma$ (dark grey), while Step (c) is to expand the disc (light grey) to a larger domain $\Psi\cup U(\gamma, \delta)\cup U(0,\gamma)$ (with thick dotted edge), then choose a proper $\delta^*$ and construct a $\Delta$-domain (solid black curve) where $f(\psi(z))$ is analytic.\label{F:continuation}}
\end{figure}


Step {\sf (c)}. Suppose $z'$ is an arbitrary point lying on the circle of $|z|=\gamma$ except $z=\gamma$. We deduce from {\sf (a)} that $\psi(z)$ is analytic at $z'$. Based on this fact, giving a sufficient small neighborhood $U(\psi(z'),\varepsilon_{z'})\subset \Delta_f$, there exists $\delta_{z'}$ such that
\begin{equation}\label{E:psi-ep-0}
\psi(z)\in U(\psi(z'),\varepsilon_{z'})
\end{equation}
for all $z\in U(z',\delta_{z'})$, where $U(\clubsuit,\mu)$ denotes the neighborhood of $\clubsuit$ with radius $\mu$. According to the condition of, eq. (\ref{E:psi-ep-0}) can be deducted to $|\psi(z)|<\lambda$. {\sf (b)} has given us a domain $\tilde{\Delta}_\gamma$ which covers a part of circle $|z|=\gamma$. Our goal here is to find a domain made by a discs chain $\Psi$ covering the arc $\iota$ of $\{z\mid|z|=\gamma\}-\tilde{\Delta}_\gamma$ and every $z$ lies in $\Psi$ satisfies $\psi(z)\in\Delta_f$. Clearly $\iota$ is a closed line, therefore Among the set of $\{\delta_{z'}\}_{z'\in\iota}$, we have a limited sequence $\{\delta_{z_1},\delta_{z_2},\ldots,\delta_{z_n}\}$, satisfying
\begin{equation*}
\iota\subset\Psi=\bigcup_{i=1}^n U(z_i, \delta_{z_i})
\end{equation*}
by finite covering theorem. $\Psi$ is well defined.

Choose $\hat{\delta}=\min\{|z|: z\in\partial(\Psi\cup U(\gamma, \delta)\cup U(0,\gamma))\}$, see Fig. \ref{F:continuation}. Clearly $\hat{\delta}>\gamma$. We set a $\Delta$-domain,
\begin{equation}
\Delta=\{z\mid |z|<\delta^*, |\arg(z-\gamma)|>\beta\},
\end{equation}
where $\delta^*\in (\gamma, \hat{\delta})$.

For an arbitrary $z\in \Delta$, $z$ should be contained only in $\Psi$, $U(\gamma,\delta)$ or $U(0,\gamma)$ since the disc $U(0,\delta^*)$ is inside $\Psi\cup U(\gamma, \delta)\cup U(0,\gamma)$. $z$ would be in $\tilde{\Delta}_\gamma$ if $z\in U(\gamma,\delta)$ for the reason of $|\arg(z-\gamma)|>\beta$. These will imply that $\psi(z)$ is analytic in $\Delta$ and always lies on the analytic domain $\Delta_f$ of $f(w)$, which infers that $f(\psi(z))$ is also analytic in $\Delta$. The $\Delta$-analyticity of $f(\psi(z))$ is proved.

\end{proof}

\section{Singular expansion of the OGF ${\bf F}_k(z)$ of $k$-noncrossing matchings}

${\bf F}_k(z)$ is denoted as the ordinary generating function (OGF) of $k$-noncrossing matchings. In this section ,we will give the singular expansion of ${\bf F}_k(z)$ through the ordinary differential equation (ODE) that ${\bf F}_k(z)$ satisfies, for $2\leq k\leq 9$.

\subsection{The ordinary differential equations}

The $D$-finiteness has been proved in Corollary \ref{C:Bessel}. That is, there exists some $e\in \mathbb{N}$ for which ${\bf F}_k(z)$
satisfies an ODE of the form
\begin{equation}\label{E:JKH}
q_{0,k}(z)\frac {d^e}{dz^e}{\bf
F}_k(z)+q_{1,k}(z)\frac{d^{e-1}}{dz^{e-1}}{\bf F}_k(z)+\cdots+
q_{e,k}(z){\bf F}_k(z)=0,
\end{equation}
where $q_{j,k}(z)$ are polynomials. The fact that ${\bf F}_k(z)$ is
the solution of an ODE implies the existence of an analytic
continuation into any simply-connected domain containing original point avoiding zeros of $q_{0,k}(z)$ \cite{Stanley:80,Wasow:87},
i.e.~{$\Delta$}-analyticity at singularity of $\rho_k^2$.

The reasons for discussing about the ODE of ${\bf F}_k(z)$ are the following.

\begin{itemize}
\item Any dominant singularity of a solution is contained
      in the set of roots of $q_{0,k}(z)$ \cite{Stanley:80}. In other
      words the ODE ``controls'' the dominant singularities, that are
      crucial for asymptotic enumeration.
\item Under certain regularity conditions (discussed below) the singular
      expansion of ${\bf F}_k(z)$ follows from the ODE, see
      Proposition~\ref{P:fk}.
\end{itemize}

Accordingly, let us first compute for $2\le k\le 9$ the ODEs for
${\bf F}_k(z)$.
\begin{proposition}[\cite{Reidys:10}]\label{P:q}
For $2\le k\le 9$, ${\bf F}_k(z)$ satisfies the ODEs listed in
Table~\ref{Table:F_kdiff} and we have in particular
\begin{eqnarray}
\label{E:r2}
q_{0,2}(z) & = &  (4z-1)\, z,    \\
\label{E:r3}
q_{0,3}(z) & = &  (16z-1)\, z^2,    \\
\label{E:r4}
q_{0,4}(z) & = & (144\, z^2-40\, z +1)  \, z^3,\\
\label{E:r5}
q_{0,5}(z) & = & (1024\, z^2-80\, z  + 1)\, z^4, \\
\label{E:r6} q_{0,6}(z) & =& (14400\, z^3-4144\, z^2 + 140\, z-1)
\,z^{5}, \\
\label{E:r7} q_{0,7}(z) & =&(147456\,z^3-12544\, z^2+224\, z-1)\,z^{6}, \\
\label{E:r8}
q_{0,8}(z)  & = & (2822400z^4-826624z^3+31584z^2-336z+1)z^{7},\\
\label{E:r9}
q_{0,9}(z) & = & (37748736z^4-3358720z^3+69888z^2-480z+1)z^{8},\qquad
\end{eqnarray}
\end{proposition}
Proposition~\ref{P:q} immediately implies the following sets of roots:
{\begin{eqnarray*}
\mathfrak{r}_2 = \{\frac{1}{4}\}; \
\mathfrak{r}_4  =  \mathfrak{r}_2 \,\cup\,\{\frac{1}{36}\};\
\mathfrak{r}_6  =  \mathfrak{r}_4 \,\cup\,\{\frac{1}{100}\};\
\mathfrak{r}_8  =  \mathfrak{r}_6 \,\cup\,\{\frac{1}{196}\};\\
\mathfrak{r}_3 = \{\frac{1}{16}\}; \
\mathfrak{r}_5  =  \mathfrak{r}_3\,\cup\, \{\frac{1}{64}\}; \
\mathfrak{r}_7  =  \mathfrak{r}_5\,\cup\,\{\frac{1}{144}\}; \
\mathfrak{r}_9  =  \mathfrak{r}_7\,\cup\,\{\frac{1}{256}\}.
\end{eqnarray*}}
Eq.~(\ref{E:r2})-(\ref{E:r9}) and Theorem~\ref{T:B} show that
for $2\le k\le 9$ the unique dominant singularity of ${\bf
F}_k(z)$ is given by {$\rho_k^2$}, where $\rho_k=1/2(k-1)$.

\subsection{The singular expansion}

Before solving the singular expansion of ${\bf F}_k(z)$, let us learn some concepts \cite{Flajolet:07a}.

\begin{definition}\label{D:meromorphic}
A {\bf meromorphic ODE} is an ODE
of the form
\begin{equation}\label{m-ode}
f^{(r)}(z)+d_1(z)f^{(r-1)}(z)+\cdots+d_r(z)f(z)=0,
\end{equation}
where $f^{(m)}(z)=\frac{d^m}{d z^m}f(z)$, $0\leq m\leq r$ and the
$d_j(z)$ are meromorphic in some domain $\Omega$. Assuming that $\zeta$
is a pole of a meromorphic function $d (z)$, $\omega_\zeta ( d )$
denotes the {\bf order} of the pole $\zeta$. In case $d (z)$ is analytic at
$\zeta$ we write $\omega_\zeta(d)=0$.
\end{definition}

\begin{definition}\label{D:mero-sing}
Meromorphic differential equations have a {\bf singularity} at $\zeta$ if at
least one of the $\omega_\zeta (d_j)$ is positive. Such a $\zeta$ is
said to be a regular singularity\index{singularity! regular} if
\begin{equation*}
\forall\, 1\leq j\leq r;\quad \omega_\zeta (d_j )\le j
\end{equation*}
and an {\bf irregular} singularity\index{singularity! irregular} otherwise.
\end{definition}

\begin{definition}\label{D:indicial}
The indicial equation\index{indicial equation} $I(\alpha)=0$ of an differential equation of
the form eq.~(\ref{m-ode}) at its regular singularity $\zeta$ is
given by
\begin{equation*}
I(\alpha)=(\alpha)_r+\delta_1(\alpha)_{r-1}+\cdots
+\delta_r,\qquad (\alpha)_{\ell}:=\alpha(\alpha- 1)\cdots
(\alpha-\ell + 1),
\end{equation*}
where $\delta_ j := \lim_{z\rightarrow\zeta} (z -\zeta)^ jd_j (z)$.
\end{definition}

\begin{theorem}\cite{Henrici:74,Wasow:87}
\label{T:reg}
Suppose we are given a meromorphic differential equation (\ref{m-ode})
with regular singularity $\zeta$.
Then, in a slit neighborhood of $\zeta$, any solution of eq.~(\ref{m-ode})
is a linear combination of functions of the form
\begin{equation*}
(z-\zeta)^{\alpha_i}\left(\log(z-\zeta)\right)^{\ell_{ij}}
H_{ij}(z-\zeta),
\quad\text{for }1\leq i\leq r,\ 1\leq j\leq i,
\end{equation*}
where $\alpha_1, \ldots , \alpha_r$ are the roots of the indicial
equation at $\zeta$, $\ell_{ij}$ are non-negative integer and each
$H_{ij}$ is analytic at $0$.
\end{theorem}

\begin{theorem}\label{T:B}\cite{Reidys:08k}
For arbitrary $k\in\mathbb{N}$, $k\ge 2$ we have
\begin{equation}\label{E:theorem}
f_{k}(2n,0) \, \sim  \, c_k  \, n^{-((k-1)^2+(k-1)/2)}\,
(2(k-1))^{2n},\quad \text{\it where $c_k>0$}.
\end{equation}
Particularly, $c_3=24/\pi$.
\end{theorem}

Thanks to Proposition \ref{P:q}, Theorem \ref{T:reg} and Theorem \ref{T:B}, we have the singular expansion of ${\bf F}_k(z)$ for $2\leq k\leq 9$.

\begin{proposition}[\cite{Reidys:10book}]\label{P:fk}
For $2\le k\le 9$, the singular expansion of ${\bf F}_k(z)$
for $z\rightarrow \rho_k^2$ is given by
\begin{equation}\label{E:Fk_sing_ex}
\begin{split}
{\bf F}_k(z)=
\begin{cases}
P_k(z-\rho_k^2)+c_k'(z-\rho_k^2)^{m}
\log(z-\rho_k^2)+o((z-\rho)^m)  \\
P_k(z-\rho_k^2)+c_k'(z-\rho_k^2)^{((k-1)^2+(k-1)/2)-1}
\left(1+o(1)\right)
\end{cases}
\end{split}
\end{equation}
where $m=((k-1)^2+(k-1)/2)-1$, depending on $k$ being odd or even. Furthermore, the terms $P_k(z)$ are
polynomials of degree less than $m+1$, $c_k'$ is
some constant, and $\rho_k=1/2(k-1)$. Particularly, $c_3'=-65536/\pi$.
\end{proposition}

Since ${\bf F}_k(z)$ is $D$-finite, ${\bf F}_k(\psi(z))$ will be $D$-finite if $\psi(z)$ is algebraic, which can imply ${\bf F}_k(\psi(z))$ is $\Delta$-analytic. Meanwhile, the singular expansion of ${\bf F}_k(z)$ possesses the same form of $f(z)$ expressed in eq. (\ref{E:fx-sing-exp}). Therefore we have the following corollary implied from Proposition \ref{chapter2:algeasym}.

\begin{corollary}[\cite{Reidys:10book}]\label{C:algeasym}
Let $\psi(z)$ be an algebraic, analytic function in a
domain $\mathcal{D}=\{z||z|\leq r\}$ such that
$\psi(0)=0$. Suppose $\gamma$ is the unique
dominant singularity of ${\bf F}_k(\psi(z))$ and minimum positive real
solution of $\psi(\gamma)=\rho_k^2$, $|\gamma|<r$,
where $\psi'(\gamma)\neq 0$.
Then ${\bf F}_k(\psi(z))$ has a singular expansion and
\begin{equation}\label{E:LL}
[z^n]{\bf F}_k(\psi(z)) \sim A\,n^{-((k-1)^2+(k-1)/2)}
\left(\frac{1}{\gamma}\right)^n,
\end{equation}
where $A$ is some constant.
\end{corollary}

\section{Distribution analysis}

We will introduce the central limit theorem\index{central limit theorem} for
distributions given in terms of bivariate generating functions. The result is going to be applied to do the statistical computation for the arc of canonical $3$-noncrossing skeleton diagrams in Section \ref{S:statistics}. The central limit theorem is due to Bender \cite{Bender:73} and based on
the following classic result on limit distributions \cite{Feller:91}:

\begin{theorem}\label{T:K}{\bf (L\'{e}vy-Cram\'{e}r)
}
Let $\{\xi_{n}\}$ be a sequence of random variables and let
$\{\varphi_{n}(x)\}$ and $\{F_{n}(x)\}$ be the corresponding
sequences of characteristic and distribution functions. If there
exists a function $\varphi(t)$, such that $\lim_{n \rightarrow
\infty}\varphi_{n}(t)=\varphi(t)$ uniformly over an arbitrary finite
interval enclosing the origin, then there exists a random variable
$\xi$ with distribution function $F(x)$ such that
$$F_{n}(x)\Longrightarrow F(x)$$
uniformly over any finite or infinite interval of continuity of
$F(x)$.
\end{theorem}

We come now to the central limit theorem. It analyzes the characteristic
function via the above L\'{e}vy-Cram\'{e}r Theorem.
\begin{theorem}[\cite{Reidys:10book}]\label{T:normal}
Suppose we are given the bivariate generating function
\begin{equation*}
f(z,u)=\sum_{n,m\geq 0}f(n,m)\,z^n\,u^m,
\end{equation*}
where $f(n,m)\geq 0$
and $f(n)=\sum_tf(n,t)$. Let $\mathbb{X}_n$ be a r.v.~such that
$\mathbb{P}(\mathbb{X}_n=t)=f(n,t)/f(n)$. Suppose
\begin{equation}\label{E:knackpunkt}
[z^n]f(z,e^s)\sim c(s) \, n^{\alpha}\,  \gamma(s)^{-n}
\end{equation}
uniformly in $s$ in a neighborhood of $0$, where $c(s)$ is continuous
and nonzero near $0$, $\alpha$ is a constant, and $\gamma(s)$ is $C\infty$
near $0$, which means $\gamma(s)$ is analytic in infinite times.
Then there exists a pair $(\mu,\sigma)$ such that the normalized
random variable
\begin{equation*}\label{E:normalized}
\mathbb{X}^*_{n}=\frac{\mathbb{X}_{n}- \mu \, n}{\sqrt{{n\,\sigma}^2 }}
\end{equation*}
has asymptotic normal distribution with parameter $(0,1)$.
That is, we have
\begin{equation}\label{E:converge-2}
\lim_{n\to\infty}
\mathbb{P}\left(
\mathbb{X}^*_{n} < x \right)  =
\frac{1}{\sqrt{2\pi}} \int_{-\infty}^{x}\,e^{-\frac{1}{2}c^2} dc \,
\end{equation}
where $\mu$ and $\sigma^2$ are given by
\begin{equation}\label{E:was}
\mu= -\frac{\gamma'(0)}{\gamma(0)}
\quad \text{\it and} \quad \sigma^2=
\left(\frac{\gamma'(0)}{\gamma(0)}
\right)^2-\frac{\gamma''(0)}{\gamma(0)}.
\end{equation}
\end{theorem}

The crucial points for applying Theorem~\ref{T:normal} are
\begin{itemize}
\item eq.~(\ref{E:knackpunkt})
\begin{equation*}
[z^n]f(z,e^s)\sim c(s) \, n^{\alpha}\,  \gamma(s)^{-n},
\end{equation*}
uniformly in $s$ in a neighborhood of $0$, where $c(s)$ is continuous
and nonzero near $0$ and $\alpha$ is a constant,
\item $\gamma(s)$ is analytic in $s$,
\end{itemize}

\begin{table}
\tabcolsep 0pt
\tiny
\begin{center}
\def\temptablewidth{1\textwidth}
{\rule{\temptablewidth}{1pt}}
\begin{tabular*}{\temptablewidth}{@{\extracolsep{\fill}}l|l}
$k$  & $ $ \\
\hline
$2\ $& $\left( 4\,x-1 \right) x f'' \left( x \right)
+ \left( 10\,x-2 \right) f' \left( x \right) +2\,f\left( x \right)=0$\vspace{3pt}\\
\hline
$3\ $ &$\left( 16\,{x}^{3}-{x}^{2} \right) f^{(3)}  \left( x \right)
 + \left( 96\,{x}^{2}-8\,x \right)f''
 \left( x \right) + \left( 108\,x -12\right)f' \left(
x \right) +12\,f \left( x \right)=0
$\vspace{3pt}\\
\hline
$4\ $ & $ \left( 144\,{x}^{5}-40\,{x}^{4}+{x}^{3} \right) f^{(4)}
\left( x \right) + \left( 1584\,{x}^{4}-556\,{x}^{3}+20\,{x}^{2}
 \right) f^{(3)} \left( x \right)$\\
$ $&  $+ \left(4428\,{x}^{3} -1968\,{x}^{2}+112 \,x \right) f''
\left( x \right)+ \left( 3024\,{x}^{2}-1728\,x+168 \right)f' \left(
x
 \right)$\\
$ $&$ + \left( 216\,x -168\right) f \left( x \right)=0
$\vspace{3pt} \\
\hline $5\ $ & $\left( 1024\,{x}^{6}-80\,{x}^{5}+{x}^{4} \right)
f^{(5)} \left( x \right) + \left(
20480\,{x}^{5}-2256\,{x}^{4}+40\,{x}^{3}
 \right) f^{(4)} \left( x \right)$\\
$$&$ + \left(
121600\,{x}^{4}-19380\,{x}^{3}+532\,{x}^{2} \right) f^{(3)} \left( x
\right) + \left( 241920\,{x}^{3}-56692\,{x}^{2}+
2728\,x \right) f'' \left( x \right)$\\
$$&$ + \left(
130560\,{x}^{2}-46048\,x+4400 \right)f' \left( x
 \right) + \left( 7680\,x-4400 \right) f \left( x \right)=0
$\vspace{3pt} \\
\hline
$6\ $ & $\left( 14400\,{x}^{8}-4144\,{x}^{7}+140\,{x}^{6}- {x}^{5}
\right) f^{(6)} \left( x \right)$\\
$$&$ + \left( 367200\,{x}^{7}-
148368\,{x}^{6}+7126\,{x}^{5}-70\,{x}^{4} \right) f^{(5)}
\left( x \right)$\\
$$&$ + \left(3078000\,{x}^{6}-
1728900\,{x}^{5}+123850\,{x}^{4}-1792\,{x}^{3} \right)f^{(4)}
\left( x \right)$\\
$$&$ +
 \left( 10179000\,{x}^{5}-7880640\,{x}^{4}+880152\,x^3-20704\,{x}^{2} \right)
 f^{(3)} \left( x \right)$\\
 $$&$ + \left(12555000\,{x}^{4} -13367880\,{x}^{3}+
2399184\,{x}^{2}-106016\,x \right) f'' \left( x \right) $\\
$$&$+ \left(
4374000\,{x}^{3}-6475680\,{x}^{2}+1922736\,x-187200 \right) f'
 \left( x \right)$\\
$$&$ + \left( 162000\,{x}^{2}-350640\,x+
187200 \right) f \left( x \right)=0
$ \vspace{3pt}\\
\hline
$7\ $ & $\left( 147456\,{x}^{9}-12544\,{x}^{8} +224\,{x}^{7}-{x}^{6}
\right) f^{(7)} \left( x \right)$\\
$$&$ + \left( 6193152\,{x}^{8}-
757760\,{x}^{7}+18816\,{x}^{6}-112\,{x}^{5} \right) f^{(6)}
\left( x \right)$\\
$$&$ + \left( 89800704\,{x}^{7}-16035456\,{x}^
{6}+582280\,{x}^{5}-4872\,{x}^{4} \right) f^{(5)}
 \left( x \right)$\\
  $$&$+ \left( 561254400\,{x}^{6}-146691840\,{x}^{5}+8254664\,{x}^{4}
  -104480
\,{x}^{3} \right) f^{(4)}
 \left( x \right)$\\
  $$&$+ \left( 1535708160\,{x}^{5}-
585419280\,{x}^{4}+54069792\,{x}^{3}-1151984\,{x}^{2} \right) f^{(3)}
 \left( x \right)$\\
  $$&$+ \left( 1651829760\,{x}^{4}-916833600\,{
x}^{3} +144777216\,{x}^{2}-6094528\,x\right) f'' \left( x
 \right)$\\
  $$&$+ \left( 516741120\,x^3-421901280\,{x}^{2}+117590208\,x-11797632 \right)
  f' \left( x \right)$\\
  $$&$ + \left(17418240\,{x}^{2} -22034880\,x+11797632
 \right) f \left( x \right)=0
$\vspace{3pt}\\
\hline
$8\ $ & $\left( 2822400\,{x}^{11}-826624\,{x}^{10}+31584\,{x}^{9}
-336\,{x}^{8}+{x}^{7} \right) f^{(8)} \left( x \right)$\\
$$&$ +
 \left( 129830400\,{x}^{10}-55968384\,{x}^
{9}+3026208\,{x}^{8}-43512\,{x}^{7}+168\,{x}^{6} \right) f^{(7)}
 \left( x
 \right)$\\
 $$&$ + \left( 2202883200\,{x}^{9}-1363532352\,{x}^{8}
 +107691912\,{x}^{7
}-2188752\,{x}^{6}+11424\,{x}^{5} \right) f^{(6)}
\left( x \right)$\\
$$&$ + ( 17455132800\,{x}^{8}-15140260128\,{x}^{7}
+1789953376\,{x}^{6}$\\
$$&$\quad -54349728\,{x}^{5}+
405200\,{x}^{4} ) f^{(5)} \left( x \right)$\\
$$&$ + ( 67586778000\,{x}^{
7}-80551356480\,{x}^{6}+14421855200\,{x}^{5}$\\
$$&$\quad
-698609104\,{x}^{4}+8035104\,{x}^{3} ) f^{(4)} \left( x
 \right)$\\
 $$&$ + ( 122393376000\,
{x}^{6}-197784236160\,{x}^{5}+
53661386080\,{x}^{4}$\\
$$&$\quad -4437573920\,{x}^{3}+88180864\,{x}^{2} ) f^{(3)} \left( x \right)$\\
$$&$ + ( 90239184000\,{x}^{5}-196676000640\,{x}^{4}+80758975680\,{x
}^{3}$\\
$$&$\quad-11973419104\,x^2+ 488846272\,x) f'' \left( x \right)$\\
$$&$ + (19559232000\,x^4-57892907520\,x^3+
36309203520\,{x}^{2}$\\
$$&$\quad-9969500032\,x+1033305728) f'\left( x \right)$\\
$$&$ +
 \left(444528000\,x^3-1852865280\,{x}^{2}+1869893760\,x-1033305728 \right) f \left(
x \right)=0
$\vspace{3pt}\\
\hline
$9\ $ &$\left(37748736\,{x}^{12}-3358720\,{x}^{11}
+69888\,{x}^{10}-480\,{x}^{9}+{x}^{8} \right)f^{(9)}  \left( x \right)$\\
$$&$ +
 \left( 2717908992\,{x}^{11}-
351387648\,{x}^{10}+10065408\,{x}^{9}-90912\,{x}^{8}+240\,{x}^{7} \right) f^{(8)}
 \left( x \right) $\\
 $$&$+ ( 72873934848\,{x}^{10}-
13784408064\,{x}^{9}$\\
$$&$\quad+563449728\,{x}^{8}
-6950616\,{x}^{7}+24024\,{x}^{6} )
f^{(7)}  \left( x \right)$\\
$$&$ + (
940566380544\,{x}^{9}-258478202880\,{x}^{8}
+15638941312\,{x}^{7}$\\
$$&$\quad-276275160\,{x}^{6}+1304336\,{x}^{5} ) f^{(6)}  \left( x
 \right)$\\
  $$&$+ ( 6273464795136\,{x}^{8}-2467959432192\,x^7+227994061392\,{x}^{6}$\\
  $$&$\quad-6121052128\,{x}^{5}+41782224
\,{x}^{4} ) f^{(5)}
 \left( x \right)$\\
  $$&$+ ( 21523928186880\,{x}^{7}-11931746135040\,{x}^{6}+1713129509184\,{x}^{5}$\\
  $$&$\quad-75115763872\,{x}^{4
}+802970368\,{x}^{3}
 ) f^{(4)}  \left( x \right)$\\
 $$&$ + (35583374131200\,{x}^{6}-27454499665920\,{x}^{5} +
6147724228704\,{x}^{4}$\\
$$&$\quad-
475182777504\,{x}^{3}+8956331968\,{x}^{2} ) f^{(3)}  \left( x \right)$\\
$$&$ + ( 24400027975680\,{x}^{5}-26056335882240\,{x}^{4}
+9086553292608\,{x}^{3}$\\
$$&$\quad-
1308864283488\,{x}^{2}+ 52313960192\,x ) f''  \left( x
 \right)$\\
  $$&$+ ( 4976321495040\,{x}^{4}-7402528051200\,{x}^{3}+4051342551744\,{x}^{2
}$\\
$$&$\quad-1122348764928\,x+120086385408 ) f' \left( x \right)$\\
 $$&$+ \left(107017666560\,{x}^{3}-
230051819520\,{x}^{2}+208033076736\,x-120086385408 \right) f \left( x \right)=0
$ \\
\hline
\end{tabular*}
\end{center}
\caption{The differential equations for ${\bf F}_k(z)(2\leq k\leq 9)$,
obtained by Maple package {\tt gfun}.} \vspace*{-12pt}
\label{Table:F_kdiff}
\end{table}


\newchap{Modular, $k$-noncrossing diagrams}
\thispagestyle{fancyplain}

In this chapter we will introduce modular, $k$-noncrossing diagrams, which is a kind of the complicated RNA structures. We first use the method of inflation from ${\sf V}_k$-shape to obtain the generating function of modular, noncrossing diagrams. Then we improve the ${\sf V}_k$-shape to {\it colored shape} so as to compute the modular, $k$-noncrossing diagrams where $k>2$. Finally, we do the singularity analysis for the generating function and conclude the asymptotic enumeration of modular, $k$-noncrossing diagrams with length $n$.

\section{Shape}


\begin{definition}
A ${\sf V}_k$-shape is a $k$-noncrossing matching having stacks of length
exactly one.
\end{definition}
In the following we refer to ${\sf V}_k$-shape simply as shapes.
That is, given a modular, $k$-noncrossing diagram, $\delta$, its shape
is obtained by first replacing each stem by an arc and then removing all
isolated vertices, see Fig.~\ref{F:Ikmapchain}.
\restylefloat{figure}\begin{figure}[ht]
\centerline{\includegraphics[width=0.65\textwidth]{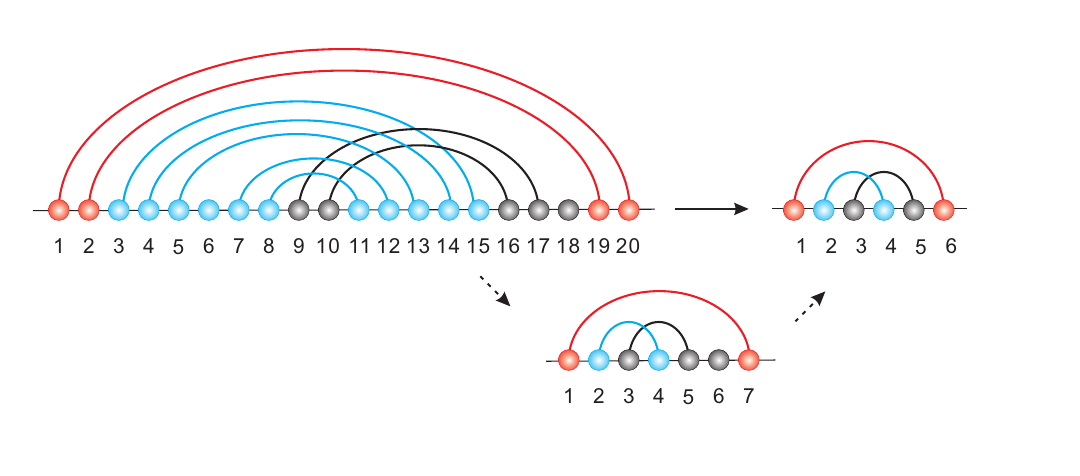}}
\caption{\small From diagrams to shapes: A modular, $3$-noncrossing
diagram (top-left) is mapped in two steps into its ${\sf V}_3$-shape
(top-right). A stem (blue) is replaced by an single shape-arc
(blue).} \label{F:Ikmapchain}
\end{figure}

Let ${\mathcal I}_k(s,m)$ ($i_k(s,m)$) denote the set (number) of
the ${\sf V}_k$-shapes with $s$ arcs and $m$ $1$-arcs having the bivariate
generating function
\begin{equation}
{\bf I}_k(z,u)=\sum_{s\geq0}\sum_{m=0}^{s} i_k(s,m)z^su^m.
\end{equation}
The bivariate generating function of $i_k(s,m)$ and the
generating function of ${\bf F}_k(z)$ are related as follows:
\begin{lemma}\label{T:gfIk}\cite{Reidys:09shape}
Let $k$ be a natural number where $k\geq 2$, then the generating
function ${\bf I}_k(z,u)$ satisfies
\begin{eqnarray}\label{E:gfIk}
{\bf I}_k(z,u) & = & \frac{1+z}{1+2z-zu}
{\bf F}_k\left(\frac{z(1+z)}{(1+2z-zu)^2}\right).
\end{eqnarray}
\end{lemma}

\section{Modular noncrossing diagrams}\label{S:k=2}


Let us begin by studying first the case $k=2$ \cite{Schuster:98},
where the asymptotic formula
\begin{equation*}
{\sf Q}_2(n)\sim 1.4848\cdot n^{-3/2}\cdot 1.8489^n
\end{equation*}
has been derived. In the following we extend the result in
\cite{Schuster:98} by computing the generating function explicitly.
The above asymptotic formula follows then easily by means of
singularity analysis.

\begin{proposition}\label{P:k=2}
The generating function of modular, noncrossing diagrams is given by
\begin{equation}\label{E:UU}
{\bf Q}_2(z)=\frac{1-z^2+z^4}{1 - z - z^2 + z^3 + 2 z^4 + z^6}\cdot {\bf F}_2
\left(\frac{z^4 - z^6 + z^8}{(1 - z - z^2 + z^3 + 2 z^4 + z^6)^2}\right)
\end{equation}
and the coefficients of  ${\sf Q}_2(n)$ satisfy
\begin{equation*}
{\sf Q}_2(n)\sim c_2 n^{-3/2} \gamma_{2}^{-n},
\end{equation*}
where $\gamma_{2}$ is the minimal, positive real solution
of $\vartheta(z)=1/4$, and
\begin{equation}
\vartheta(z)=
\frac{z^4 - z^6 + z^8}{(1 - z - z^2 + z^3 + 2 z^4 + z^6)^2}.
\end{equation}
Here we have $\gamma_{2}\approx 1.8489$ and
$c_2\approx 1.4848$.
\end{proposition}
\begin{proof}
Let ${\mathcal Q}_{2}$ denote the set of modular noncrossing diagrams,
${\mathcal I}_2$ the set of all ${\sf V}_2$-shapes and ${\mathcal I}_2(m)$
those having exactly $m$ $1$-arcs. Then we have the surjective map
\begin{equation*}
\varphi\colon {\mathcal Q}_{2}\rightarrow {\mathcal I}_2.
\end{equation*}
The map $\varphi$ is obviously surjective, inducing the partition
${\mathcal Q}_{2}=\dot\cup_\gamma\varphi^{-1}(\gamma)$, where
$\varphi^{-1}(\gamma)$ is the preimage set of shape $\gamma$ under
the map $\varphi$. Accordingly, we arrive at
\begin{equation}\label{E:Hgf}
{\bf Q}_{2}(z) =
\sum_{m\geq 0}\sum_{\gamma\in\,{\mathcal I}_2(m)}
\mathbf{Q}_{\gamma}(z).
\end{equation}
We proceed by computing the generating function
$\mathbf{Q}_{\gamma}(z)$. We shall construct
$\mathbf{Q}_{\gamma}(z)$ from certain combinatorial classes as
``building blocks''. The latter are: $\mathcal{M}$ (stems),
$\mathcal{K}$ (stacks), $\mathcal{N}$ (induced stacks),
$\mathcal{L}$ (isolated vertices), $\mathcal{R}$ (arcs) and
$\mathcal{Z}$ (vertices), where $\mathbf{Z}(z)=z$ and
$\mathbf{R}(z)=z^2$. We inflate $\gamma\in {\mathcal I}_2(m)$ having
$s$ arcs, where $s\geq \max\{1,m\}$, to a modular noncrossing
diagram in two steps: \\
{\it Claim.} For any shape $\gamma\in\mathcal{I}_2(s,m)$ we have
\begin{eqnarray*}
{\bf Q}_{\gamma}(z)&=& \left(\frac{\frac{z^{4}}{1-z^2}}
{1-\frac{z^{4}}{1-z^2}\left(2\frac{z}{1-z}
+\left(\frac{z}{1-z}\right)^2\right)}\right)^s
\left(\frac{1}{1-z}\right)^{2s+1-m}
\left(\frac{z^3}{1-z}\right)^{m}\nonumber\\
&=&(1-z)^{-1}\left(\frac{z^{4}}{(1-z^2)(1-z)^2-(2z-z^2)
z^{4}}\right)^s \, (z^3)^m.
\end{eqnarray*}
{\bf Step I:} we inflate any shape-arc to a stack of length at least
$2$ and subsequently add additional stacks. The latter are called
induced stacks \index{induced stack} and have to be separated by
means of inserting isolated vertices, see Fig.~\ref{F:addstack}.
\restylefloat{figure}\begin{figure}[ht]
\centerline{\includegraphics[width=0.9\textwidth]{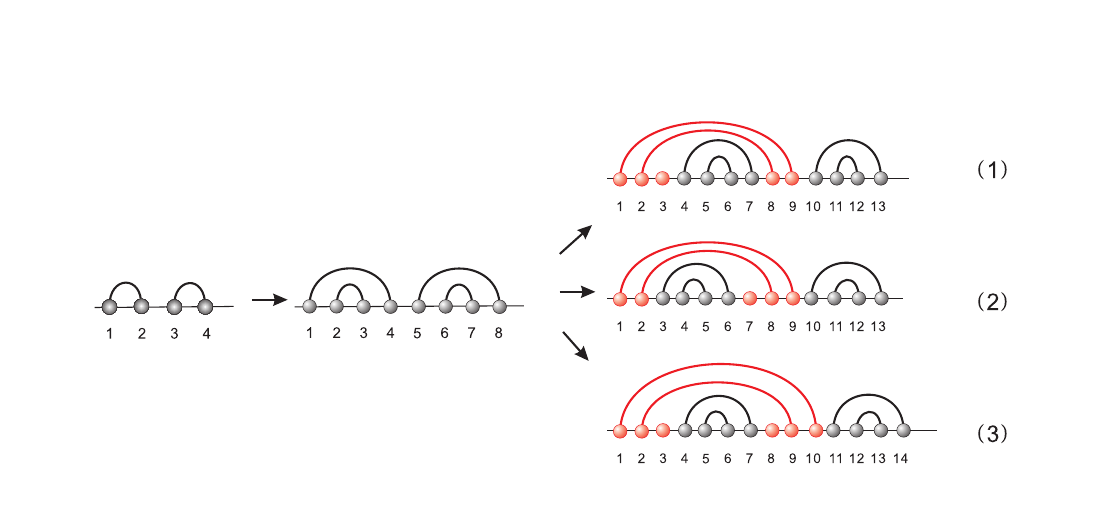}}
\caption{\small Illustration of Step I.} \label{F:addstack}
\end{figure}
Note that during this first inflation step no intervals of isolated
vertices, other than those necessary for separating the nested
stacks are inserted. We generate
\begin{itemize}
\item sequences of isolated vertices
$\mathcal{L}= \textsc{Seq}(\mathcal{Z})$, where
\begin{eqnarray*}
 {\bf L}(z) & = &  \frac{1}{1-z}
\end{eqnarray*}
\item stacks, i.e.
\begin{equation*}
\mathcal{K}=
\mathcal{R}^{2}\times\textsc{Seq}\left(\mathcal{R}\right)
\end{equation*}
with the generating function
\begin{eqnarray*}
\mathbf{K}(z) & = & z^{4}\cdot \frac{1}{1-z^2},
\end{eqnarray*}
\item induced stacks, i.e.~stacks together with at least one nonempty
interval of isolated vertices on either or both its sides.
\begin{equation*}
\mathcal{N}=\mathcal{K}\times \left(\mathcal{Z}\times\mathcal{L}
+\mathcal{Z}\times\mathcal{L}+\left(\mathcal{Z}\times
\mathcal{L}\right)^2\right)
\end{equation*}
with generating function
\begin{equation*}
\mathbf{N}(z)=\frac{z^{4}}{1-z^2}\left(2\frac{z}{1-z}
+\left(\frac{z}{1-z}\right)^2\right),
\end{equation*}
\item stems\index{stem}, that is pairs consisting of a stack $\mathcal{K}$
and an arbitrarily long sequence of induced stacks
\begin{equation*}
\mathcal{M}=\mathcal{K} \times \textsc{Seq}\left(\mathcal{N}\right)
\end{equation*}
with generating function
\begin{eqnarray*}
\mathbf{M}(z)=\frac{\mathbf{K}(z)}{1-\mathbf{N}(z)}=
\frac{\frac{z^{4}}{1-z^2}}
{1-\frac{z^{4}}{1-z^2}\left(2\frac{z}{1-z}
+\left(\frac{z}{1-z}\right)^2\right)}.
\end{eqnarray*}
\end{itemize}
{\bf Step II:} we insert additional isolated vertices at the
remaining $(2s+1)$ positions. For each $1$-arc at least three such
isolated vertices are necessarily inserted, see
Fig.~\ref{F:addvertex}.
\restylefloat{figure}\begin{figure}[ht]
\centerline{\includegraphics[width=1\textwidth]{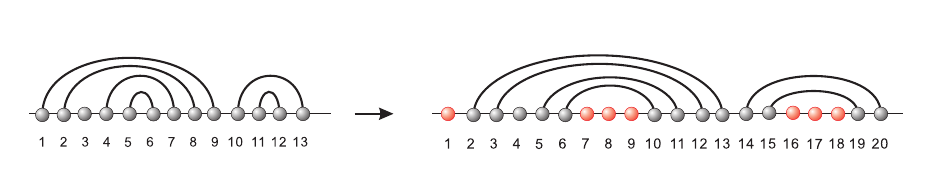}}
\caption{\small Step II: the noncrossing diagram (left) obtained in
{\sf (1)} in Fig.~\ref{F:addstack} is inflated to a modular
noncrossing diagram (right) by adding isolated vertices (red).}
\label{F:addvertex}
\end{figure}
We arrive at
\begin{equation}
\mathcal{Q}_{\gamma}=\left(\mathcal{M}\right)^s
\times\mathcal{L}^{2s+1-m}\times\left(\mathcal{Z}^3\times
\mathcal{L}\right)^{m},
\end{equation}
where $\mathcal{Q}_{\gamma}$ is the combinatorial class of modular
noncrossing diagrams having shape $\gamma\in\mathcal{I}_2(s,m)$.
Combining these generating functions the Claim follows.\\
Since for any $\gamma,\gamma_1\in \mathcal{I}_2(s,m)$
\index{$\mathcal{I}_k(s,m)$}
we have $\mathbf{Q}_{\gamma}(z)=\mathbf{Q}_{\gamma_1}(z)$, we derive
\begin{equation*}
{\bf Q}_{2}(z) = \sum_{m\geq 0}\sum_{\gamma\in\,{\mathcal
I}_2(m)} \mathbf{Q}_\gamma(z) =
\sum_{s\geq 0}\sum_{m=0}^s i_2(s,m)\mathbf{Q}_\gamma(z).
\end{equation*}
We set
\begin{equation*}
\eta(z)=\frac{z^{4}}{(1-z^2)(1-z)^2-(2z-z^2)z^{4}}
\end{equation*}
and note that Lemma~\ref{T:gfIk} guarantees
\begin{eqnarray*}
\sum_{s\geq0}\,\sum_{m=0}^{s} \,i_2(s,m)\,x^s\,y^m & = &
\frac{1+x}{1+2x-xy}\sum_{s\geq
0}f_2(2s)\left(\frac{x(1+x)}{(1+2x-xy)^2}\right)^s.
\end{eqnarray*}
Therefore, setting $x=\eta(z)$ and $y=z^3$ we arrive at
\begin{eqnarray*}
{\bf Q}_2(z)=\frac{1-z^2+z^4}{1 - z - z^2 + z^3 + 2 z^4 + z^6}\cdot {\bf F}_2
\left(\frac{z^4 - z^6 + z^8}{(1 - z - z^2 + z^3 + 2 z^4 + z^6)^2}\right)
\end{eqnarray*}
By Theorem~\ref{T:d-finite_property}, ${\bf Q}_2(z)$ is $D$-finite. Pringsheim's
Theorem \cite{Tichmarsh:39} guarantees that ${\bf Q}_2(z)$ has a
dominant real positive singularity $\gamma_{2}$. We verify that
$\gamma_{2}$ which is the unique solution of minimum modulus of the
equation $\vartheta(z)=\rho_2^2$, where  $\rho_2^2$ is the unique
dominant singularity of ${\bf F}_2(z)$ and $\rho_2=1/2$. Furthermore
we observe that $\gamma_{2}$ is the unique dominant singularity of
${\bf Q}_2(z)$. It
is straightforward to verify that $\vartheta'(\gamma_{2})\neq 0$.
According to Corollary~\ref{C:algeasym}, we therefore have
\begin{equation*}
{\sf Q}_2(n)\sim c_2 n^{-3/2} \gamma_{2}^{-n},
\end{equation*}
and the proof of Proposition~\ref{P:k=2} is complete.
\end{proof}

\section{Colored shapes}\label{S:color}

In the following part we shall assume that $k>2$, unless stated
otherwise. The key steps to compute the generating function of modular
$k$-noncrossing diagrams are certain refinements of their ${\sf
V}_k$-shapes. These refined shapes are called colored shapes and
obtained by distinguishing a variety of crossings of $2$-arcs,
i.e.~arcs of the form $(i,i+2)$. Each such class requires its
specific inflation-procedure in Theorem~\ref{T:queen}.

Let us next have a closer look at these combinatorial classes (colors):
\begin{itemize}
\item $\mathbf{C}_1$\index{$\mathbf{C}_1$} the class of $1$-arcs,
\item $\mathbf{C}_2$\index{$\mathbf{C}_2$}
      the class of arc-pairs consisting of mutually crossing
      $2$-arcs,
\item $\mathbf{C}_3$\index{$\mathbf{C}_3$} the class of arc-pairs
      $(\alpha,\beta)$ where $\alpha$ is the unique $2$-arc crossing
      $\beta$ and $\beta$ has length at least three.
\item $\mathbf{C}_4$\index{$\mathbf{C}_4$}
      the class of arc-triples $(\alpha_1,\beta,\alpha_2)$, where
      $\alpha_1$ and $\alpha_2$ are $2$-arcs that cross $\beta$.
\end{itemize}
In Fig.\ \ref{F:map2} we illustrate how these classes are induced by
modular $k$-noncrossing diagrams.
\restylefloat{figure}\begin{figure}[h!t!b!p]
\centering
\includegraphics{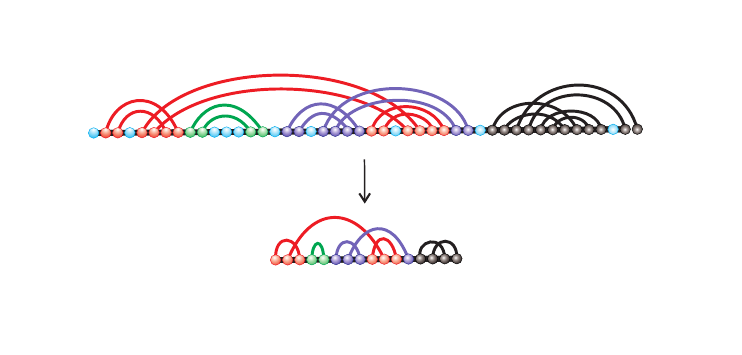}
\caption{\small
Colored ${\sf V}_k$-shapes: a modular $3$-noncrossing diagram (top)
and its colored ${\sf V}_3$-shape (bottom). In the resulting ${\sf
V}_3$-shape we color the four classes as follows:
$\mathbf{C}_1$(green), $\mathbf{C}_2$(black), $\mathbf{C}_3$(blue)
and $\mathbf{C}_4$(red).} \label{F:map2}
\end{figure}

Let us refine ${\sf V}_k$-shapes in two stages. For this
purpose let ${\mathcal I}_k(s,u_1,u_2)$\index{${\mathcal
I}_k(s,u_1,u_2)$} and $i_k(s,u_1,u_2)$\index{${\mathcal
I}_k(s,u_1,u_2)$} denote the set and cardinality of ${\sf V}_k$-shapes
having $s$ arcs, $u_1$ $1$-arcs and $u_2$ pairs of
mutually crossing $2$-arcs. Our first objective consists in
computing the generating function
\begin{equation*}
{\bf W}_k(x,y,w)=\sum_{s\geq 0}
\sum_{u_1=0}^{s}\sum_{u_2=0}^{\lfloor\frac{s-u_1}{2}\rfloor}
i_k(s,u_1,u_2)\, x^sy^{u_1}w^{u_2}.
\end{equation*}
That is, we first take the classes $\mathbf{C}_1$ and
$\mathbf{C}_2$ into account.

\begin{lemma}\label{L:recursion-0}
For $k> 2$, the coefficients $i_k(s,u_1,u_2)$ satisfy
\begin{eqnarray}\label{E:00}
i_k(s,u_1,u_2)& = & 0 \quad \text{for } u_1+2u_2>s\\
\label{E:wq}
\sum_{u_2=0}^{\lfloor\frac{s-u_1}{2}\rfloor}i_k(s,u_1,u_2) & = & i_k(s,u_1),
\end{eqnarray}
where $i_k(s,u_1)$ denotes the number of ${\sf V}_k$-shapes
having $s$ arcs, $u_1$ $1$-arcs. Furthermore we have the recursion:
\begin{eqnarray}\label{E:3rec}
(u_2+1)i_k(s+1,u_1,u_2+1)&=&(u_1+1)i_k(s,u_1+1,u_2)\nonumber\\&&
+(u_1+1)i_k(s-1,u_1+1,u_2).\label{E:2arcp}
\end{eqnarray}
The solution of eq.~(\ref{E:00})--(\ref{E:2arcp}) is unique.
\end{lemma}
The proof of Lemma~\ref{L:recursion-0} is given in Section~\ref{S:appendix}.
We next proceed by computing ${\bf W}_k(x,y,w)$.
\begin{proposition}\label{P:trivariate}
For $k>2$, we have
\begin{equation}\label{E:3gf}
{\bf W}_k(x,y,w)=(1+x)v \, {\bf F}_k\left(x(1+x)v^2\right),
\end{equation}
where $v=\left((1-w)x^3+(1-w)x^2+(2-y)x+1\right)^{-1}$.
\end{proposition}
\begin{proof}
According to Lemma~\ref{T:gfIk}, we have
\begin{equation*}
{\bf I}_k(z,u)  =  \frac{1+z}{1+2z-zu} {\bf
F}_k\left(\frac{z(1+z)}{(1+2z-zu)^2}\right).
\end{equation*}
This generating function is connected to ${\bf W}_k(x,y,z)$ via
eq.~(\ref{E:wq}) of Lemma~\ref{L:recursion-0} as follows: setting
$w=1$, we have ${\bf W}_k(x,y,1)={\bf I}_k(x,y)$.
The recursion of eq.~(\ref{E:2arcp}) gives rise to the partial differential
equation
\begin{eqnarray}\label{E:wdiff}
\frac{\partial {\bf W}_k(x,y,w)}{\partial w}&=&x\frac{\partial {\bf
W}_k(x,y,w)}{\partial y}+x^2\frac{\partial {\bf
W}_k(x,y,w)}{\partial y}.
\end{eqnarray}
We next show
\begin{itemize}
\item the function
\begin{eqnarray}
{\bf W}_k^*(x,y,w)&=&\frac{(1+x)}{(1-w)x^3+(1-w)x^2+(2-y)x+1}
\times \nonumber\\
&&{\bf
F}_k\left(\frac{(1+x)x}{((1-w)x^3+(1-w)x^2+(2-y)x+1)^2}\right)\qquad
\end{eqnarray}
is a solution of eq.~(\ref{E:wdiff}),
\item the coefficients
$i_k^*(s,u_1,u_2)=[x^sy^{u_1}w^{u_2}]{\bf
W}_k^*(x,y,w)$ satisfy $$i_k^*(s,u_1,u_2)=
0\quad\text{for}\quad u_1+2u_2>s,$$
\item ${\bf W}_k^*(x,y,1)={\bf I}_k(x,y)$.
\end{itemize}
Firstly,
\begin{eqnarray}\label{E:diffw}
\frac{\partial {\bf W}^*_k(x,y,w)}{\partial y} & = & u\,
{\bf F}_k\left(u\right)+ 2u \,{\bf F}_k'\left(u\right)\\
\frac{\partial {\bf W}^*_k(x,y,w)}{\partial w} & = & x(1+x)u \, {\bf
F}_k\left(u\right)+ 2x(1+x)u {\bf F}_k'\left(u\right),
\end{eqnarray}
where $$u= \frac{x(1+x)}{\left( (1-w)x^3+(1-w)x^2+(2-y)x+1\right)
^{2}}$$ and ${\bf F}_k'\left(u\right)=\sum_{n\geq
0}nf_k(2n)(u)^n$. Consequently, we derive that
\begin{equation}
\frac{\partial {\bf W}_k^*(x,y,w)}{\partial w}=x\frac{\partial {\bf
W}_k^*(x,y,w)}{\partial y}+x^2\frac{\partial {\bf
W}_k^*(x,y,w)}{\partial y}.
\end{equation}
Secondly we prove $i_k^*(s,u_1,u_2)=0$ for $u_1+2u_2>s$.
To this end we observe that ${\bf W}^*_k(x,y,w)$ is a power series,
since it is analytic in $(0,0,0)$. It now suffices to note that the
indeterminants $y$ and $w$ only appear in form of products $xy$ and
$x^2w$ or $x^3w$.
Thirdly, the equality ${\bf W}_k^*(x,y,1)={\bf I}_k(x,y)$ is obvious.\\
{\it Claim.}
\begin{equation}\label{E:eequal2}
{\bf W}^*_k(x,y,w)={\bf W}_k(x,y,w).
\end{equation}
By construction the coefficients $i^*_k(s,u_1,u_2)$ satisfy
eq.~(\ref{E:3rec}) and we have just proved $i_k^*(s,u_1,u_2)=0$ for $u_1+2u_2>s$.
In view of ${\bf W}_k^*(x,y,1)={\bf I}_k(x,y)$ we have
\begin{equation*}
\forall\, s,u_1;\qquad
 \sum_{u_2=0}^{\lfloor\frac{s-u_1}{2}\rfloor}i_k^*(s,u_1,u_2)=i_k(s,u_1).
\end{equation*}
Using these three properties, Lemma~\ref{L:recursion-0} implies
\begin{equation*}
\forall\, s,u_1,u_2\ge 0;\qquad i_k^*(s,u_1,u_2)=i_k(s,u_1,u_2),
\end{equation*}
whence the Claim and the proposition is proved.
\end{proof}
In addition to $\mathbf{C}_1$ and $\mathbf{C}_2$, we consider next
the classes $\mathbf{C}_3$ and $\mathbf{C}_4$. For this purpose we
have to identify two new recursions, see Lemma~\ref{L:recursions-0}.
Setting $\vec{u}=(u_1,\dots,u_4)$, we denote by
$\mathcal{I}_k(s,\vec{u})$\index{$\mathcal{I}_k(s,\vec{u})$} and
$i_k(s,\vec{u})$\index{$i_k(s,\vec{u})$} the set and the number of
colored ${\sf V}_k$-shapes over $s$ arcs, containing $u_i$
elements of class $\mathbf{C}_i$, where $1\le i\le 4$. The key
result is
\begin{lemma}\label{L:recursions-0}
For $k>2$, the coefficients $i_k(s,\vec{u})$ satisfy
\begin{eqnarray}\label{E:erni2}
i_k(s,u_1,u_2,u_3,u_4) & = & 0  \quad\text{for }u_1+2u_2+2u_3+3u_4>s\\
\label{E:5ini}
\sum_{u_3,u_4\geq 0}i_k(s,u_1,u_2,u_3,u_4) & = & i_k(s,u_1,u_2).
\end{eqnarray}
Furthermore we have the recursions
\begin{align}
(u_3+1)&i_k(s+1,u_1,u_2,u_3+1,u_4)=\nonumber \\
&\quad \,  2u_1i_k(s-1,u_1,u_2,u_3,u_4)\nonumber\\
&+4(u_2+1)i_k(s-1,u_1,u_2+1,u_3,u_4)\nonumber\\
&+4(u_2+1)i_k(s-1,u_1,u_2+1,u_3-1,u_4)\nonumber\\
&+4(u_2+1)i_k(s-2,u_1,u_2+1,u_3-1,u_4)\nonumber\\
&+2(u_3+1)i_k(s,u_1,u_2,u_3+1,u_4)\nonumber\\
&+2u_3i_k(s-1,u_1,u_2,u_3,u_4)\nonumber\\
&+{6(u_3+1)i_k(s-1,u_1,u_2,u_3+1,u_4)}\nonumber\\
&+2(u_3+1)i_k(s-2,u_1,u_2,u_3+1,u_4)\nonumber\\
&+2u_3i_k(s-2,u_1,u_2,u_3,u_4)\nonumber\\
&+4(u_4+1)i_k(s,u_1,u_2,u_3-1,u_4+1)\nonumber\\
&+4(u_4+1)i_k(s-1,u_1,u_2,u_3-1,u_4+1)\nonumber\\
&+4u_4i_k(s-1,u_1,u_2,u_3,u_4)\nonumber\\
&+4(u_4+1)i_k(s-1,u_1,u_2,u_3,u_4+1)\nonumber\\
&+4u_4i_k(s-2,u_1,u_2,u_3,u_4)\nonumber\\
&+2(u_4+1)i_k(s-2,u_1,u_2,u_3,u_4+1)\nonumber\\
&+({2s}-2u_1-4u_2-4u_3-6u_4)i_k(s,u_1,u_2,u_3,u_4)\nonumber\\
&+2(2(s-1)-2u_1-4u_2-4u_3-6u_4)i_k(s-1,u_1,u_2,u_3,u_4)\nonumber\\
&+({2(s-2)}-4u_2-4u_3-6u_4)i_k(s-2,u_1,u_2,u_3,u_4)\label{u3recursion}
\end{align}
and
\begin{align}
2(u_4+1)i_k(s+1,u_1,u_2,u_3,u_4+1)&=(u_3+1)i_k(s,u_1,u_2,u_3+1,u_4)\nonumber\\
&\quad +2(u_2+1)i_k(s,u_1,u_2+1,u_3,u_4).\quad\label{u4recursion}
\end{align}
The sequence satisfying eq.~(\ref{E:erni2})--(\ref{u4recursion}) is unique.
\end{lemma}
The proof of Lemma~\ref{L:recursions-0} is given in
Section~\ref{S:appendix}. It is obtained by removing a specific arc in a
labeled ${\bf C}_3$-element or a labeled ${\bf C}_4$-element and accounting
of the resulting arc-configurations.

Proposition~\ref{P:trivariate} and Lemma~\ref{L:recursions-0} put us
in position to compute the generating function of colored ${\sf
V}_k$-shapes
\begin{equation}
{\bf I}_k(x,y,z,w,t)=
\sum_{s,u_1,u_2,u_3,u_4}i_k(s,\vec{u})\,
x^s y^{u_1} z^{u_2}w^{u_3} t^{u_4}.
\end{equation}
\begin{proposition}\label{P:5tri}
For $k>2$, the generating function of colored ${\sf V}_k$-shapes is given by
\begin{equation}
{\bf I}_k(x,y,z,w,t)=\frac{1+x}{\theta}{\bf F}_k
\left(\frac{x(1+(2w-1)x+(t-1)x^2)}{\theta^2}\right),
\end{equation}
where $\theta=1-(y-2)x+(2w-z-1)x^2+(2w-z-1)x^3$.
\end{proposition}
\begin{proof}
The first recursion of Lemma~\ref{L:recursions-0} implies the partial
differential equation
\begin{align}
\frac{\partial {\bf I}_k}{\partial w}&=
\frac{\partial {\bf I}_k}{\partial x}(2x^2+4x^3+2x^4)-
\frac{\partial {\bf I}_k}{\partial y}(2xy+2x^2 y)\nonumber\\
&+\frac{\partial {\bf I}_k}{\partial z}
(-4xz+4x^2w+4x^2-4x^3z-8x^2z+4x^3w)\nonumber\\
&+\frac{\partial {\bf I}_k}{\partial w}(-4xw+2x-6x^2w+6x^2-2x^3w+2x^3)
\nonumber\\
&+\frac{\partial {\bf I}_k}{\partial t}(-6xt+4xw-8x^2t+4x^2w+4x^2-2x^3t+2x^3).
\label{E:KA}
\end{align}
Analogously, the second recursion of Lemma~\ref{L:recursions-0} gives rise to the
partial differential equation
\begin{align}
 2\frac{\partial {\bf I}_k}{\partial t}
&=\frac{\partial {\bf I}_k}{\partial w}x+
\frac{\partial {\bf I}_k}{\partial z}\, 2x.\label{E:KB}
\end{align}
Aside from being a solution of eq.~(\ref{E:KA}) and
eq.~(\ref{E:KB}), we take note of the fact that eq.~(\ref{E:5ini})
of Lemma~\ref{L:recursions-0} is equivalent to
\begin{equation}
{\bf I}_k(x,y,z,1,1)={\bf W}_k(x,y,z).\label{bound}
\end{equation}
We next show that
\begin{itemize}
\item the function
\begin{align*}
{\bf I}_k^*(x,y,z,w,t)&=\frac{1+x}{1-(y-2)x+(2w-z-1)x^2+(2w-z-1)x^3}
\times \\
&{\bf F}_k
\left(\frac{x(1+(2w-1)x+(t-1)x^2)}
{(1-(y-2)x+(2w-z-1)x^2+(2w-z-1)x^3)^2}\right)
\end{align*}
is a solution of eq.~(\ref{E:KA}) and eq.~(\ref{E:KB}),
\item its coefficients,
$i_k^*(s,u_1,u_2,u_3,u_4)=[x^sy^{u_1}z^{u_2}w^{u_3}t^{u_4}]{\bf
I}_k^*(x,y,z,w,t)$, satisfy
$$i_k^*(s,u_1,u_2,u_3,u_4)=0\quad\text{for}\quad u_1+2u_2+2u_3+3u_4>s,$$
\item ${\bf I}_k^*(x,y,z,1,1)={\bf W}_k(x,y,z)$.
\end{itemize}
We verify by direct computation that ${\bf I}_k^*(x,y,z,w,t)$
satisfies eq.~(\ref{E:KA}) as well as eq.~(\ref{E:KB}).
Next we prove $i_k^*(s,u_1,u_2,u_3,u_4)=0$ for $u_1+2u_2+2u_3+3u_4>s$.
Since ${\bf I}_k^*(x,y,z,w,t)$ is analytic in $(0,0,0,0,0)$, it is
a power series. As the indeterminants $y$, $z$, $w$ and $t$
appear only in form of products $xy$, $x^2z$ or $x^3z$, $x^2w$ or
$x^3w$, and $x^3t$, respectively, the assertion follows.\\
{\it Claim.}
\begin{equation*}
{\bf I}_k^*(x,y,z,w,t)={\bf I}_k(x,y,z,w,t).
\end{equation*}
By construction,
$i_k^*(s,\vec{u})$ satisfies the recursions eq.~(\ref{u3recursion})
and eq.~(\ref{u4recursion}) as well as
$i_k^*(s,u_1,u_2,u_3,u_4)=0$ for $u_1+2u_2+2u_3+3u_4>s$.
Eq.~(\ref{bound}) implies
\begin{equation*}
\sum_{u_3,u_4\geq 0}i_k^*(s,u_1,u_2,u_3,u_4)=i_k(s,u_1,u_2).
\end{equation*}
Using these properties we can show via Lemma~\ref{L:recursions-0},
\begin{equation*}
\forall\, s,u_1,u_2,u_3,u_4\geq 0;\qquad
i_k^*(s,u_1,u_2,u_3,u_4)=i_k(s,u_1,u_2,u_3,u_4)
\end{equation*}
and the proposition is proved.
\end{proof}


\section{The main theorem}\label{S:main}

We are now in position to compute ${\bf Q}_{k}(z)$. All technicalities aside,
we already introduced the main the strategy in the proof of
Proposition~\ref{P:k=2}:
as in the case $k=2$ we shall take care of all ``critical'' arcs by specific
inflations.

\begin{theorem}\label{T:queen}
Suppose $k>2$, then
\begin{eqnarray}\label{E:gut0}
{\bf Q}_{k}(z) &=&
\frac{1-z^2+z^4}{q(z)}
{\bf F}_k\left(\vartheta(z)\right),
\end{eqnarray}
where
\begin{eqnarray}
q(z)& = & 1-z-z^2+z^3+2z^4+z^6-z^8+z^{10}-z^{12}
\nonumber \\
\label{sing_eq}
\vartheta(z) & = &
\frac{z^4(1-z^2-z^4+2z^6-z^8)}{q(z)^2}.
\end{eqnarray}
Furthermore, for $3\leq k\leq 9$, ${\sf Q}_k(n)$ satisfies
\begin{equation}\label{E:gut1}
{\sf Q}_k(n)\sim  c_{k}\, n^{-((k-1)^2+(k-1)/2)}\,
\gamma_{k}^{-n} ,\quad \text{for some
$c_{k}>0$,}\qquad
\end{equation}
where $\gamma_{k}$ is the minimal, positive real solution
of $\vartheta(z)=\rho_k^2$, see Table~\ref{T:tab20}.
\end{theorem}
\begin{table}[!h]
\tabcolsep 0pt
\begin{center}
\def\temptablewidth{1\textwidth}
{\rule{\temptablewidth}{1pt}}
\begin{tabular*}{\temptablewidth}{@{\extracolsep{\fill}}llllllllll}
$k$   & $ 3 $ & $ 4 $ & $ 5 $ & $ 6 $ &  $7$ & $8$ & $9$\\
\hline
$\theta(n)$ &$n^{-5}$ & $n^{-\frac{21}{2}}$ &
$n^{-18}$ & $n^{-\frac{55}{2}}$ &
$n^{-39}$ & $n^{-\frac{105}{2}}$ & $n^{-68}$
\vspace{3pt}\\
\hline $\gamma_{k}^{-1}$& $2.5410$ & $3.0132$ & $3.3974$ &
$3.7319$ & $4.0327$ & $4.3087$ & $4.5654$\\
\hline
\end{tabular*}
\end{center}
\caption{\small Exponential growth rates $\gamma_{k}^{-1}$ and
subexponential factors $\theta(n)$, for modular, $k$-noncrossing
diagrams.} \vspace*{-12pt} \label{T:tab20}
\end{table}
\begin{proof}
Let $\mathcal{Q}_k$ denote the set of modular, $k$-noncrossing diagrams
and let $\mathcal{I}_k$ and $\mathcal{I}_k(s,\vec{u})$ denote
the set of all ${\sf V}_k$-shapes and those having $s$ arcs and
$u_i$ elements belonging to class $\mathbf{C}_i$, where $1\le i\le 4$.
Then we have the surjective map,
\begin{equation*}
\varphi_k:\ \mathcal{Q}_k\rightarrow \mathcal{I}_k,
\end{equation*}
inducing the partition
$\mathcal{Q}_k=\dot{\cup}_\gamma\varphi_k^{-1}(\gamma)$, where
$\varphi_k^{-1}(\gamma)$ is the preimage set of shape $\gamma$ under
the map $\varphi_k$. This partition allows us to organize ${\bf
Q}_k(z)$ with respect to colored ${\sf V}_k$-shapes,
$\gamma$, as follows:
\begin{equation}\label{E:tk2}
{\bf Q}_k(z)= \sum_{s,\vec{u}}
\sum_{\gamma\in\mathcal{I}_k(s,\vec{u})} {\bf Q}_\gamma(z).
\end{equation}
We proceed by computing the generating function ${\bf Q}_\gamma(z)$
following the strategy of Proposition~\ref{P:k=2}, also using the
notation therein. The key point is that the inflation-procedures are
specific to the $\mathbf{C}_i$-classes. We next inflate all ``critical''
arcs, i.e.~arcs that require the insertion of additional isolated vertices
in order to satisfy the minimum arc length condition.\\
{\it Claim $1$.} For a shape $\gamma\in \mathcal{I}_k(s,\vec{u})$ we have
\begin{align*}
{\bf Q}_\gamma(z) &=
{\bf C}_1(z)^{u_1}\cdot{\bf C}_2(z)^{u_2}\cdot
{\bf C}_3(z)^{u_3}\cdot{\bf C}_4(z)^{u_4}\cdot
{\bf S}(z)\\
&=\frac{1}{1-z}\varsigma_0(z)^s \varsigma_1(z)^{u_1} \varsigma_2(z)^{u_2}
\varsigma_3(z)^{u_3} \varsigma_4(z)^{u_4},
\end{align*}
where
\begin{align*}
\varsigma_0(z)&=\frac{z^4}{1 - 2 z + 2 z^3 - z^4 - 2 z^5 + z^6}
,\quad\varsigma_1(z)=z^3\\
\varsigma_2(z)&=
   \frac{z (1 - 4 z^3 + 2 z^4 + 8 z^5 - 6 z^6 - 7 z^7 + 8 z^8 + 2 z^9 -
   4 z^{10} + z^{11})}{1-z}\\
\varsigma_3(z)&=z (2 - 2 z^2 + z^3 + 2 z^4 - z^5)\\
\varsigma_4(z)&=z^2 (5 - 4 z - 3 z^2 + 6 z^3 + 2 z^4 - 4 z^5 + z^6).
\end{align*}
We show how to inflate a shape into a modular $k$-noncrossing diagram,
distinguishing specific classes of shape-arcs. For this purpose we refer
to a stem different from a $2$-stack as a $\dagger$-stem.
Accordingly, the combinatorial class of $\dagger$-stems is given by
$(\mathcal{M}-\mathcal{R}^2)$.
\begin{itemize}
\item {\bf $\mathbf{C}_1$-class:} here we insert isolated vertices, see
Fig. \ref{F:step1},
\restylefloat{figure}\begin{figure}[h!t!b!p]
\centering
\includegraphics[width=1\textwidth]{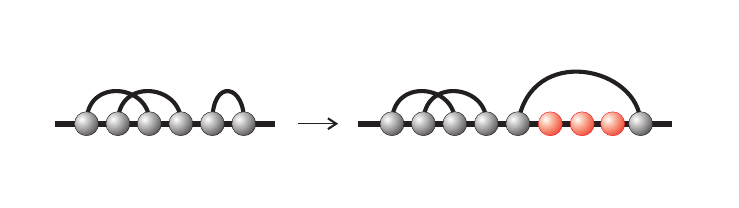}
\caption{\small $\mathbf{C}_1$-class: insertion of at least three
vertices (red) \label{F:step1}}
\end{figure}
and obtain immediately
\begin{equation}\label{E:K1}
{\bf C}_1(z)=\frac{z^3}{1-z}.
\end{equation}

\item {\bf $\mathbf{C}_2$-class:} any such element is a pair
$((i,i+2),(i+1,i+3))$ and we shall distinguish the following
scenarios:
\begin{itemize}
\item both arcs are inflated to stacks of length two, see Fig. \ref{F:step2a}.
      Ruling out the cases where no isolated vertex is inserted and the
      two scenarios, where there is no insertion into the
      interval\index{interval}
      $[i+1,i+2]$ and only in either $[i,i+1]$ or $[i+2,i+3]$, see
      Fig.\ \ref{F:step2a}, we arrive at
\begin{equation*}
      \mathcal{C}_2^{(\text{a})}
      = \mathcal{R}^4 \times [(\textsc{Seq}(\mathcal{Z}))^3-\mathcal{E}
      -2(\mathcal{Z}\times\textsc{Seq}(\mathcal{Z}))].
\end{equation*}
      This combinatorial class has the generating function
\begin{equation*}
{\bf C}_{2}^{(\text{a})}(z)=
z^8\left(\left(\frac{1}{1-z}\right)^3-1-\frac{2z}{1-z}\right).
\end{equation*}
\restylefloat{figure}\begin{figure}[h!t!b!p]
\centering
\includegraphics[width=1\textwidth]{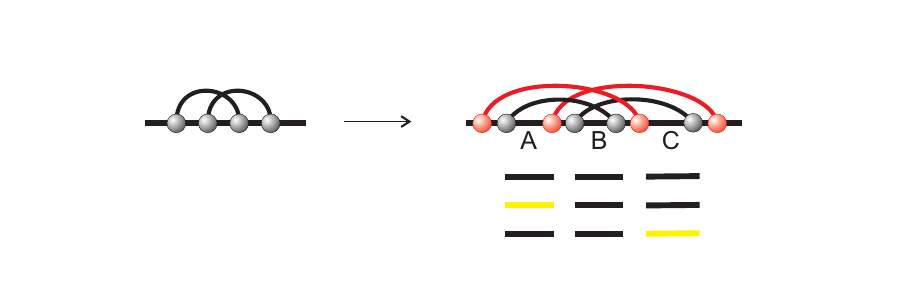}
\caption{\small {\bf $\mathbf{C}_2$-class:} inflation of both arcs
to 2-stacks. Inflated arcs are colored red while the original arcs
of the shape are colored black. We set $A=[i,i+1]$, $B=[i+1,i+2]$
and $C=[i+2,i+3]$ and illustrate the ``bad'' insertion scenarios as
follows: an insertion of some isolated vertices is represented by a
yellow segment and no insertion by a black segment. See the text for
details.} \label{F:step2a}
\end{figure}
\item one arc, $(i+1,i+3)$ or $(i,i+2)$ is inflated to a $2$-stack,
      while its counterpart is inflated to an arbitrary  $\dagger$-stem,
      see Fig. \ref{F:step2b}.
      Ruling out the cases where no vertex is inserted
      in $[i+1,i+2]$ and $[i+2,i+3]$ or $[i,i+1]$ and $[i+1,i+2]$, we obtain
\begin{equation*}
      \mathcal{C}_{2}^{(\text{b})}=2\cdot\left[
      \mathcal{R}^2\times(\mathcal{M}-\mathcal{R}^2)
      \times((\textsc{Seq}(\mathcal{Z}))^2-\mathcal{E})
      \times\textsc{Seq}(\mathcal{Z})\right],
\end{equation*}
      having the generating function
    \begin{equation*}
      {\bf C}_{2}^{(\text{b})}(z)=2
      z^4\left(\frac{\frac{z^4}{1-z^2}}{1-\frac{z^4}{1-z^2}
      \left(\frac{2z}{1-z}+\left(\frac{z}{1-z}\right)^2\right)}-z^4\right)
      \left(\left(\frac{1}{1-z}\right)^2-1\right)
      {\frac{1}{1-z}}.
\end{equation*}
\restylefloat{figure}\begin{figure}[h!t!b!p]
\centering
\includegraphics[width=0.55\textwidth]{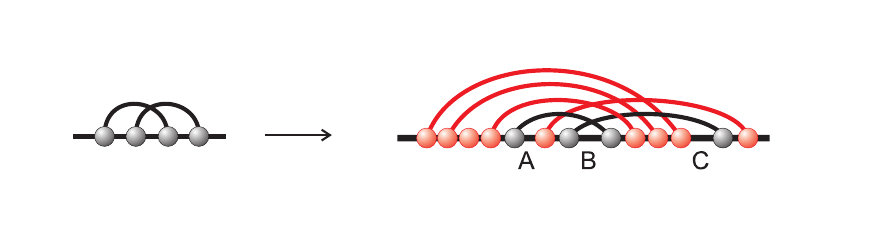}
\caption{\small {\bf $\mathbf{C}_2$-class:} inflation of only one
arc to a $2$-stack. Arc-coloring and labels as in Fig.\
\ref{F:step2a}} \label{F:step2b}
\end{figure}
\item both arcs are inflated to an arbitrary $\dagger$-stem,
      respectively, see Fig.\ \ref{F:step2c}. In this case
      the insertion of isolated
      vertices is arbitrary, whence
\begin{equation*}
      \mathcal{C}_{2}^{(\text{c})}=
      (\mathcal{M}-\mathcal{R}^2)^2\times(\textsc{Seq}(\mathcal{Z}))^3,
\end{equation*}
      with generating function
    \begin{equation*}
      {\bf C}_{2}^{(\text{c})}(z)=
      \left(\frac{\frac{z^4}{1-z^2}}{1-\frac{z^4}{1-z^2}
      \left(\frac{2z}{1-z}+\left(\frac{z}{1-z}\right)^2\right)}-z^4\right)^2
      \left(\frac{1}{1-z}\right)^3.
\end{equation*}
\end{itemize}
\restylefloat{figure}\begin{figure}[h!t!b!p]
\centering
\includegraphics[width=1\textwidth]{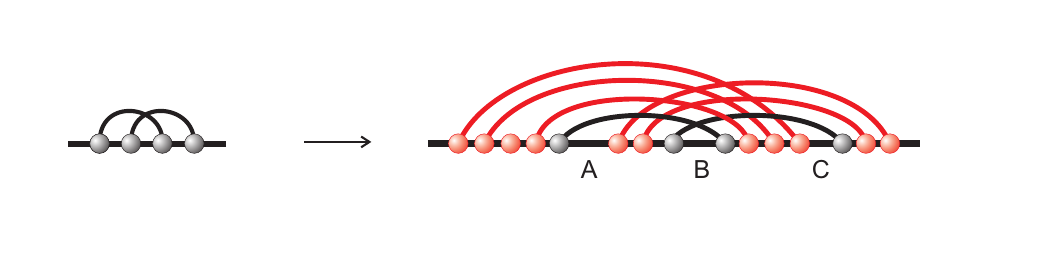}
\caption{\small {\bf $\mathbf{C}_2$-class:} inflation of both arcs
to an arbitrary $\dagger$-stem. Arc-coloring and labels as in Fig.\
\ref{F:step2a}} \label{F:step2c}
\end{figure}
As the above scenarios are mutually exclusive, the generating
function of the $\mathcal{C}_2$-class is given by
\begin{equation}\label{E:K2}
{\bf C}_2(z)= {\bf C}_2^{(\text{a})}(z)+{\bf
C}_2^{(\text{b})}(z)+{\bf C}_2^{(\text{c})}(z).
\end{equation}
Furthermore note that {\it both} arcs of an $\mathcal{C}_2$-element
are inflated in the cases (a), (b) and (c).

\item {\bf $\mathbf{C}_3$-class:} this class consists of arc-pairs $(\alpha,\beta)$
      where $\alpha$ is the unique $2$-arc crossing $\beta$ and $\beta$ has
      length at least three. Without loss of generality we can restrict our
      analysis to the case {$((i,i+2),(i+1,j))$}, $(j>i+3)$.
\begin{itemize}
\item the arc $(i+1,j)$ is inflated to a $2$-stack. Then
      we have to insert at least one isolated vertex in either
      $[i,i+1]$ or $[i+1,i+2]$, see Fig. \ref{F:step3}. Therefore we have
\begin{equation*}
    \mathcal{C}_{3}^{(\text{a})}=
    \mathcal{R}^2 \times (\textsc{Seq}(\mathcal{Z})^2-\mathcal{E}),
\end{equation*}
    with generating function
    \begin{equation*}
    {\bf C}_{3}^{(\text{a})}(z)=
      z^4\left(\left(\frac{1}{1-z}\right)^2-1\right).
    \end{equation*}
     Note that the arc $(i,i+2)$ is not considered here, it can be inflated
     without any restrictions.
\item the arc $(i+1,j)$ is inflated to an arbitrary  $\dagger$-stem,
      see Fig.\ \ref{F:step3}). Then
\begin{equation*}
    \mathcal{C}_{3}^{(\text{b})}=
    (\mathcal{M}-\mathcal{R}^2)\times
    {\textsc{Seq}(\mathcal{Z})^2,}
\end{equation*}
    with generating function
    \begin{equation*}
    {\bf C}_{3}^{(\text{b})}(z)=
    \left(\frac{\frac{z^4}{1-z^2}}{1-\frac{z^4}{1-z^2}
    \left(\frac{2z}{1-z}+\left(\frac{z}{1-z}\right)^2\right)}-z^4\right)
    \cdot \left(\frac{1}{1-z}\right)^2.
\end{equation*}
\end{itemize}
\restylefloat{figure}\begin{figure}[h!t!b!p]
\centering
\includegraphics[width=1\textwidth]{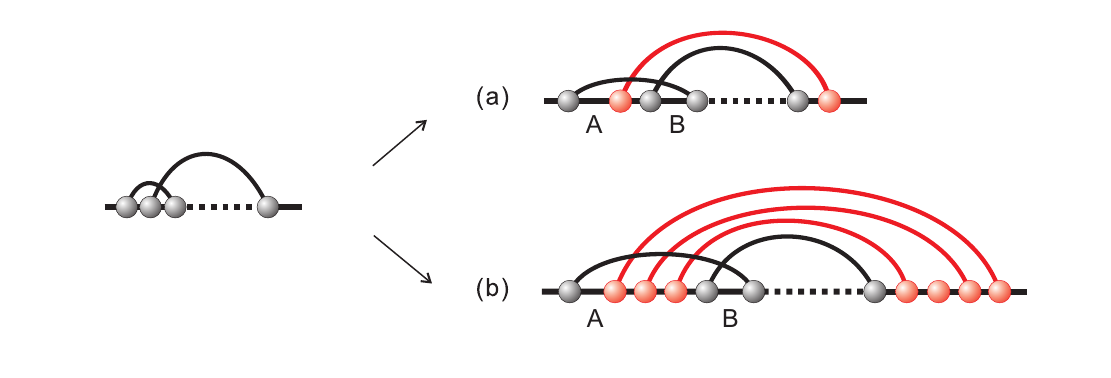}
\caption{\small $\mathbf{C}_3$-class: only one arc is inflated here
and its inflation distinguishes two subcases. Arc-coloring as in
Fig.\ \ref{F:step2a}}\label{F:step3}
\end{figure}
Consequently, this inflation process leads to a generating function
\begin{equation}\label{E:K3}
{\bf C}_3(z)={\bf C}_{3}^{(\text{a})}(z)+{\bf
C}_{3}^{(\text{b})}(z).
\end{equation}
Note that during inflation (a) and (b) only {\it one} of the two
arcs of an $\mathbf{C}_3$-class element is being inflated.

\item {\bf $\mathbf{C}_4$-class:} this class consists of arc-triples
     $(\alpha_1,\beta,\alpha_2)$, where $\alpha_1$ and $\alpha_2$
     are $2$-arcs, respectively, that cross $\beta$.
\begin{itemize}
\item $\beta$ is inflated to a $2$-stack,
      see Fig.\ \ref{F:step4}. Using similar arguments as in the case of
      $\mathbf{C}_3$-class, we arrive at
\begin{equation*}
    \mathcal{C}_{4}^{(\text{a})}=
     \mathcal{R}^2 \times (\textsc{Seq}(\mathcal{Z})^2-\mathcal{E})
     \times(\textsc{Seq}(\mathcal{Z})^2-\mathcal{E}),
\end{equation*}
    with generating function
\begin{equation*}
    {\bf C}_{4}^{(\text{a})}(z)=
    z^4\left(\left(\frac{1}{1-z}\right)^2-1\right)^2.
\end{equation*}
\item the arc $\beta$ is inflated to an arbitrary
     $\dagger$-stem\index{$\dagger$-stem},
     see Fig.\ \ref{F:step4},
\begin{equation*}
     \mathcal{C}_{4}^{(\text{b})}=
     (\mathcal{M}-\mathcal{R}^2)\times
     {\textsc{Seq}(\mathcal{Z})^4,}
\end{equation*}
    with generating function
\begin{equation*}
    {\bf C}_{4}^{(\text{b})}(z)=
    \left(\frac{\frac{z^4}{1-z^2}}{1-\frac{z^4}{1-z^2}
    \left(\frac{2z}{1-z}+\left(\frac{z}{1-z}\right)^2\right)}-z^4\right)
    \cdot \left(\frac{1}{1-z}\right)^4.
\end{equation*}
\end{itemize}
\restylefloat{figure}\begin{figure}[h!t!b!p]
\centering
\includegraphics[width=1\textwidth]{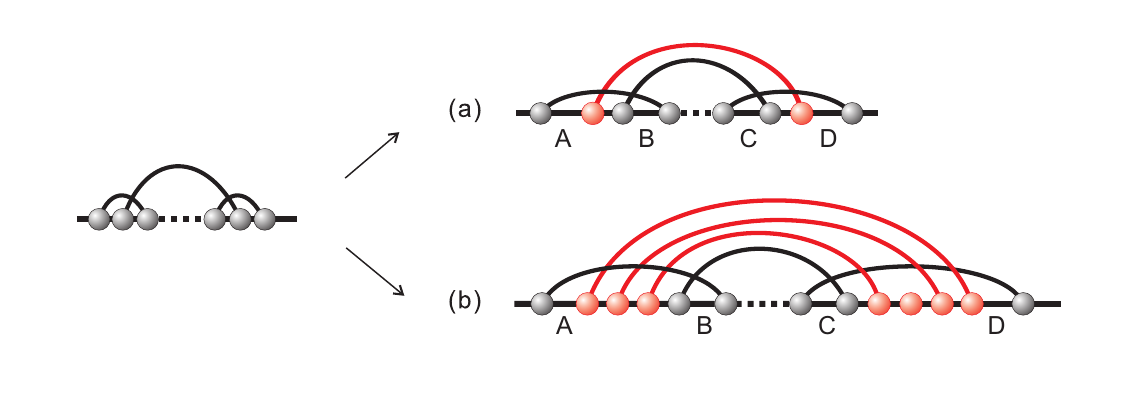}
\caption{\small $\mathbf{C}_4$-class: as for the inflation of
$\mathbf{C}_3$ only the non $2$-arc is inflated, distinguishing two
subcases. Arc-coloring as in Fig.\ \ref{F:step2a}} \label{F:step4}
\end{figure}
Accordingly we arrive at
\begin{equation}\label{E:K4}
{\bf C}_4(z)={\bf C}_{4}^{(\text{a})}(z)+{\bf
C}_{4}^{(\text{b})}(z).
\end{equation}
\end{itemize}
The inflation of any arc of $\gamma$ not considered in the previous
steps follows the logic of Proposition~\ref{P:k=2}. We observe that
$(s-2u_2-u_3-u_4)$ arcs of the shape $\gamma$ have not been
considered. Furthermore, $(2s+1-u_1-3u_2-2u_3-4u_4)$ intervals were
not considered for the insertion of isolated vertices. The inflation
of these along the lines of Proposition~\ref{P:k=2} gives rise to the class
\begin{equation*}
\mathcal{S}= \mathcal{M}^{s-2u_2-u_3-u_4}\times
(\textsc{Seq}(\mathcal{Z}))^{2s+1-u_1-3u_2-2u_3-4u_4},
\end{equation*}
having the generating function
\begin{eqnarray*}
{\bf S}(z) & = & \left(\frac{\frac{z^4}{1-z^2}}{1-\frac{z^4}{1-z^2}
           \left(\frac{2z}{1-z}+\left(\frac{z}{1-z}\right)^2\right)}\right)
            ^{s-2u_2-u_3-u_4}\times \\
 &&  \qquad\qquad\qquad\qquad   \quad
      \left(\frac{1}{1-z}\right)^{2s+1-u_1-3u_2-2u_3-4u_4}.
\end{eqnarray*}
Combining these observations Claim $1$ follows.\\
Observing that ${\bf Q}_{\gamma_1}(z)={\bf Q}_{\gamma_2}(z)$ for any
$\gamma_1,\gamma_2\in \mathcal{I}_k(s,\vec{u})$, we have, according
to eq.~(\ref{E:tk2}),
\begin{equation*}
{\bf Q}_k(z)=\sum_{s,\vec{u}\ge 0}\, i_k(s,\vec{u})\;{\bf Q}_\gamma(z),
\end{equation*}
where $\vec{u}\ge 0$ denotes $u_i\geq0$ for $1\le i\le 4$.
Proposition \ref{P:5tri} guarantees
\begin{align*}
&\ \sum_{s,\vec{u}\geq0}i_k(s,\vec{u})\; x^s y^{u_1}
    r^{u_2} w^{u_3} t^{u_4}\\
&=\frac{1+x}{1-(y-2)x+(2w-r-1)x^2+(2w-r-1)x^3}\ \times \\
& \quad\  {\bf F}_k\left(\frac{x(1+(2w-1)x+(t-1)x^2)}
{(1-(y-2)x+(2w-r-1)x^2+(2w-r-1)x^3)^2}\right).
\end{align*}
Setting $x=\varsigma_0(z)$, $y=\varsigma_1(z)$,
$r=\varsigma_2(z)$, $w=\varsigma_3(z)$,
$t=\varsigma_4(z)$, we arrive at
\begin{eqnarray*}
{\bf Q}_k(z)&=&
\frac{1-z^2+z^4}{1-z-z^2+z^3+2z^4+z^6-z^8+z^{10}-z^{12}}\ \times \\
& &{\bf F}_k\left(\frac{z^4(1-z^2-z^4+2z^6-z^8)}
{(1-z-z^2+z^3+2z^4+z^6-z^8+z^{10}-z^{12})^2}\right).
\end{eqnarray*}
By Theorem~\ref{T:d-finite_property}, ${\bf Q}_k(z)$ is $D$-finite. Pringsheim's
Theorem \cite{Tichmarsh:39} guarantees that ${\bf Q}_k(z)$ has a
dominant real positive singularity $\gamma_{k}$. We verify that for
$3\leq k\leq 9$, $\gamma_{k}$ which is the unique solution with
minimum modulus of the equation $\vartheta(z)=\rho_k^2$ is the
unique dominant singularity of ${\bf Q}_k(z)$,  and
$\vartheta'(z)\neq 0$. According to
Corollary~\ref{C:algeasym} we therefore have
\begin{equation*}
{\sf Q}_k(n)\sim  c_{k}\, n^{-((k-1)^2+(k-1)/2)}\,
(\gamma_{k}^{-1})^n ,\quad \text{for some {$c_{k}^{}>0$}},
\end{equation*}
and the proof of Theorem~\ref{T:queen} is complete.
\end{proof}

\begin{remark}\label{R:obacht}{\rm
We remark that Theorem~\ref{T:queen} does not hold for $k=2$,
i.e.~we cannot compute the generating function ${\bf Q}_{2}(z)$ via
eq.~(\ref{E:gut0}). The reason is that Lemma~\ref{L:recursions-0}
only holds for $k>2$ and indeed we find
\begin{equation}\label{E:obacht}
{\bf Q}_2(z)\neq
\frac{1-z^2+z^4}{q(z)}
{\bf F}_2\left(\frac{z^4(1-z^2-z^4+2z^6-z^8)}{q(z)^2}\right).
\end{equation}
However, the computation of the generating function ${\bf Q}_2(z)$
in Proposition~\ref{P:k=2} is based on Lemma~\ref{T:gfIk}, which
does hold for $k=2$.}
\end{remark}
\section{Proofs of Lemma~\ref{L:recursion-0} and
Lemma~\ref{L:recursions-0}}\label{S:appendix}

\subsection{Proof of Lemma~\ref{L:recursion-0}}
\begin{proof}
By construction, eq.~(\ref{E:00}) and eq.~(\ref{E:wq}) hold. We next
prove eq.~(\ref{E:2arcp}). Choose a shape $\delta\in\mathcal
I_k(s+1, u_1,u_2+1)$ and label exactly one of the $(u_2+1)$ ${\bf
C}_2$-elements. We denote the leftmost ${\bf C}_2$-arc (being a
$2$-arc) by $\alpha$. Let $\mathcal{L}$ be the set of these labeled
shapes, $\lambda$, then
\begin{equation*}
\vert \mathcal{L}\vert\,=\,(u_2+1)\,i_k(s+1,u_1,u_2+1).
\end{equation*}
We next observe that the removal of $\alpha$ results in either a
shape or a matching. Let the elements of the former set be
$\mathcal{L}_1$ and those of the latter $\mathcal{L}_2$. By
construction,
\begin{equation*}
\mathcal{L}=\mathcal{L}_1\dot\cup\mathcal{L}_2.
\end{equation*}
{\it Claim 1.}
\begin{equation*}
\vert \mathcal{L}_1 \vert\,=\,(u_1+1)\,i_k(s,u_1+1,u_2).
\end{equation*}
To prove Claim 1, we consider the labeled ${\bf C}_2$-element
$(\alpha,\beta)$. Let $\mathcal{L}_1^\alpha$ be the set of shapes
induced by removing $\alpha$. It is straightforward to verify that
the removal of $\alpha$ can lead to only one additional ${\bf
C}_1$-element, $\beta$. Therefore $\mathcal{L}_1$-shapes induce
unique $\mathcal{I}_k(s,u_1+1,u_2)$-shapes, having a labeled
$1$-arc, $\beta$, see
Fig.~\ref{F:insertu2_1}. This proves Claim 1.\\
\restylefloat{figure}\begin{figure}[h!t!b!p]
\centering
\includegraphics[width=1\textwidth]{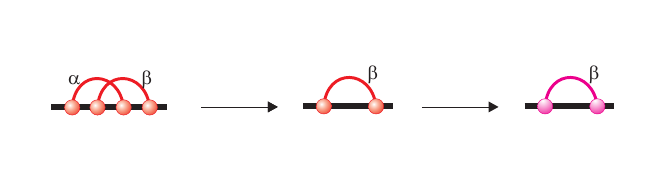}
\caption{The term $(u_1+1)\,i_k(s,u_1+1,u_2)$. \label{F:insertu2_1}}
\end{figure}

\noindent{\it Claim 2.}
\begin{equation*}
\vert \mathcal{L}_2\vert\,=\,(u_1+1)\,i_k(s-1,u_1+1,u_2).
\end{equation*}
To prove Claim 2, we consider $\mathcal{M}_2^\alpha$, the set of
matchings, $\mu_2^\alpha$, obtained by removing $\alpha$. Such a
matching contains exactly one stack of length two, $(\beta_1,
\beta_2)$, where $\beta_2$ is nested in $\beta_1$. Let
$\mathcal{L}_2^\alpha$ be the set of shapes induced by collapsing
$(\beta_1, \beta_2)$ into $\beta_2$. We observe that $\alpha$
crosses $\beta_2$ and that $\beta_2$ becomes a $1$-arc. Therefore,
$\mathcal{L}_2$ is the set of labeled shapes, that induce unique
$\mathcal{I}_k(s-1,u_1+1,u_2)$-shapes having a labeled $1$-arc,
$\beta_2$, see Fig.~\ref{F:insertu2_2}. This proves Claim 2.\\
\restylefloat{figure}\begin{figure}[h!t!b!p]
\centering
\includegraphics[width=1\textwidth]{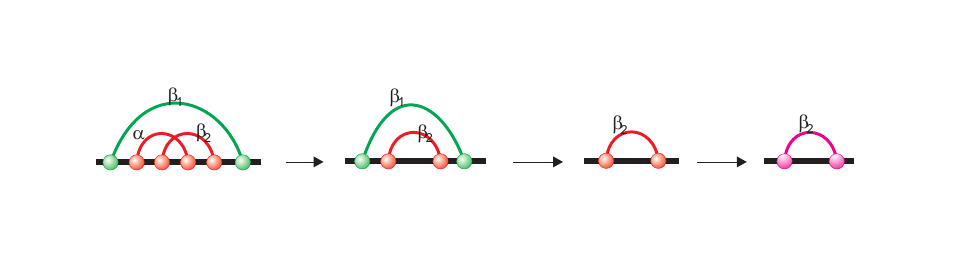}
\caption{The term $(u_1+1)\,i_k(s-1,u_1+1,u_2)$.}
\label{F:insertu2_2}
\end{figure}
Combining Claim 1 and Claim 2 we derive eq.~(\ref{E:2arcp}).\\
It remains to show (by induction on $s$) that the numbers
$i_k(s,u_1,u_2)$ can be uniquely derived from eq.~(\ref{E:00}),
eq.~(\ref{E:wq}) and eq.~(\ref{E:2arcp}), whence the lemma.
\end{proof}
\subsection{Proof of Lemma~\ref{L:recursions-0}}
\begin{proof}
By construction, eq.~(\ref{E:erni2}) and eq.~(\ref{E:5ini}) hold. We
next prove eq.~(\ref{u3recursion}). \\
Choose a shape $\delta\in \mathcal{I}_k(s+1,u_1,u_2,u_3+1,u_4)$ and
label exactly one of the $(u_3+1)$ $\mathbf{C}_3$-elements
containing a unique $2$-arc, $\alpha$. We denote the set of these
labeled shapes, $\lambda$, by $\mathcal{L}$. Clearly
\begin{equation*}
\vert \mathcal{L}\vert =(u_3+1)i_k(s+1,u_1,u_2,u_3+1,u_4).
\end{equation*}
We observe that the removal of $\alpha$ results in either a shape
($\mathcal{L}_1$) or a matching ($\mathcal{L}_2$), i.e.~we have
\begin{equation*}
\mathcal{L} = \mathcal{L}_1\, \dot\cup \, \mathcal{L}_2.
\end{equation*}
\restylefloat{figure}\begin{figure}
\centering
\includegraphics[width=1\textwidth]{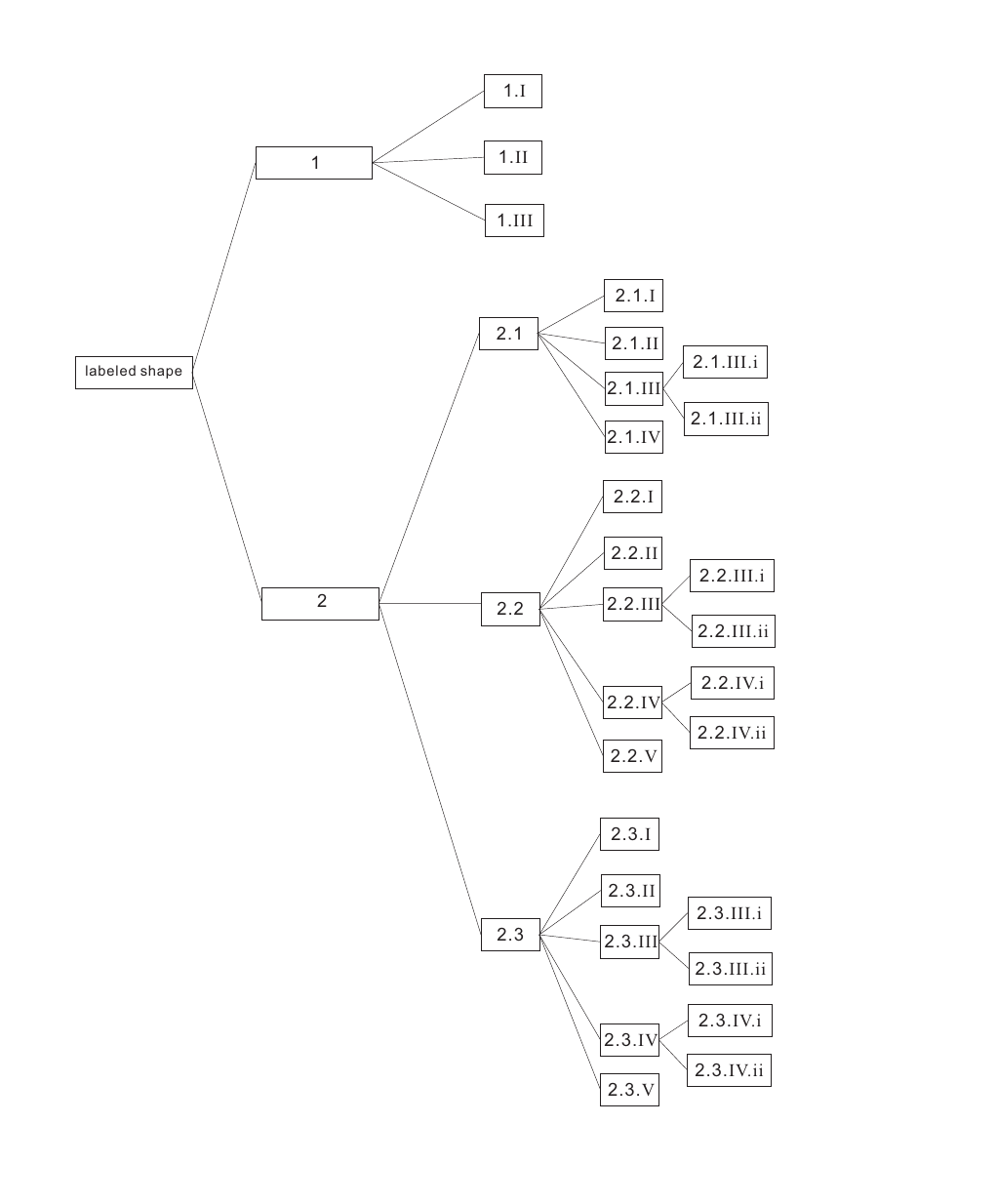}
\caption{Accounting: the scenarios arising from the removal of
$\alpha$ from a labeled ${\bf C}_3$-element of a
$\mathcal{I}_k(s+1,u_1,u_2,u_3+1,u_4)$-shape.} \label{F:tree}
\end{figure}

{\it Claim 1.}
\begin{eqnarray}
\vert \mathcal{L}_1\vert & = &  2(u_3+1)\, i_k(s,u_1,u_2,u_3+1,u_4) + \\
                         &&     4(u_4+1)\, i_k(s,u_1,u_2,u_3-1,u_4+1) + \\
                &&   (2(s-u_1-2u_2-2u_3-3u_4))\, i_k(s,u_1,u_2,u_3,u_4).
\nonumber
\end{eqnarray}
To prove Claim 1, we consider the labeled ${\bf C}_3$-element of a
$\mathcal{L}_1$-shape, $(\alpha,\beta)$. We set
$\mathcal{L}_1^\alpha$ to be the set of shapes induced by removing
$\alpha$ and denote the resulting shapes by $\lambda_1^\alpha$.
\restylefloat{figure}\begin{figure}[h!t!b!p]
\centering
\includegraphics[width=1\textwidth]{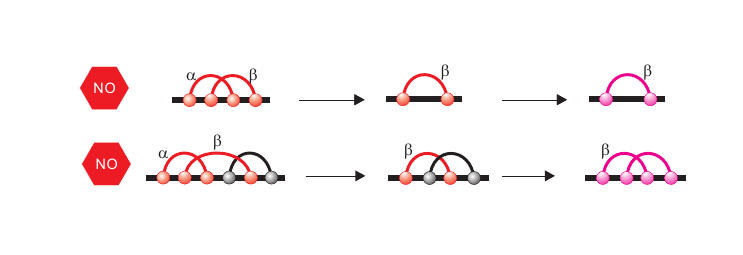}
\caption{\small The removal of $\alpha$ cannot give rise to
additional $\mathbf{C}_1$- or
$\mathbf{C}_2$-elements.}\label{F:s_10}
\end{figure}
By construction a $\lambda_1^\alpha$-shape cannot contain any
additional $\mathbf{C}_1$- or $\mathbf{C}_2$-elements, see
Fig.~\ref{F:s_10}. Clearly, the removal of $\alpha$ can lead to at
most one additional ${\bf C}_i$-element, whence
\begin{equation*}
\mathcal{L}_{1}=\mathcal{L}_1^{\mathbf{C}_3}\dot\cup
\mathcal{L}_1^{\mathbf{C}_4}\dot\cup \mathcal{L}_1^{0},
\end{equation*}
where $\mathcal{L}_1^{\mathbf{C}_i}, i=3,4$ denotes the set of
labeled shapes, $\lambda\in\mathcal{L}_1$, that induce a unique
shape having a labeled ${\bf C}_i$-element containing $\beta$ and
$\mathcal{L}_1^{0}$ the set of those shapes, in which there exists
no such ${\bf C}_i$-element. \\

We first prove
\begin{equation*}
\vert\mathcal{L}_1^{\mathbf{C}_3}\vert=2(u_3+1)\,
i_k(s,u_1,u_2,u_3+1,u_4).\leqno (1.\text{I})
\end{equation*}
\restylefloat{figure}\begin{figure}[h!t!b!p]
\centering
\includegraphics[width=1\textwidth]{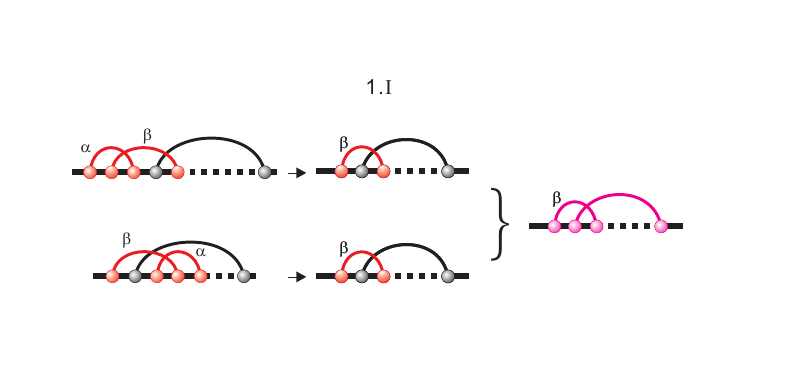}
\caption{\small We illustrate the effect of the removal of $\alpha$
when inducing a labeled ${\bf C}_3$-element.}\label{F:s_13}
\end{figure}
Indeed, in order to generate a labeled ${\bf C}_3$-element by
$\alpha$-removal from $\mathcal{L}_1$-shape, $\beta$ has to become a
$2$-arc in a labeled ${\bf C}_3$-element of a
$\mathcal{I}_k(s,u_1,u_2,u_3+1,u_4)$-shape, see Fig.~\ref{F:s_13}.\\

Next we prove
\begin{equation*}
\vert\mathcal{L}_1^{\mathbf{C}_4}\vert=4(u_4+1)\,
i_k(s,u_1,u_2,u_3-1,u_4+1). \leqno (1.\text{II})
\end{equation*}
\restylefloat{figure}\begin{figure}[h!t!b!p]
\centering
\includegraphics[width=1\textwidth]{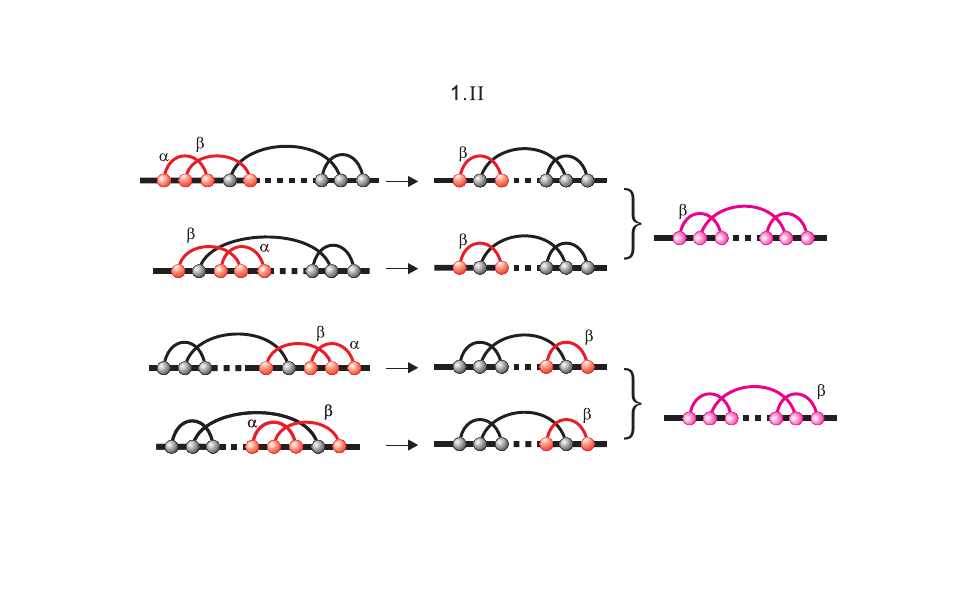}
\caption{\small We illustrate the effect of the removal of $\alpha$
with when inducing a labeled ${\bf C}_4$-element.}\label{F:s_14}
\end{figure}
Indeed in order to generate  a labeled ${\bf C}_4$-element by
$\alpha$-removal from $\mathcal{L}_1$-shape, $\beta$ has to become a
$2$-arc in a labeled ${\bf C}_4$-element of a
$\mathcal{I}_k(s,u_1,u_2,u_3-1,u_4+1)$-shape. We display all
possible scenarios in Fig.~\ref{F:s_14}. Otherwise $\beta$ becomes simply a labeled arc in a
$\mathcal{I}_k(s,u_1,u_2,u_3,u_4)$-shape, which is not contained in
any ${\bf C}_i$-element, whence
\begin{equation*}
\vert\mathcal{L}_1^{0}\vert= 2(s-u_1-2u_2-2u_3-3u_4)\,
i_k(s,u_1,u_2,u_3,u_4)\leqno (1.\text{III})
\end{equation*}
and Claim 1 follows.\\

We next consider $\mathcal{L}_2$. Let $\mathcal{M}_2^\alpha$ be the
set of
matchings, $\mu_2^\alpha$, obtained by removing $\alpha$.\\
{\it Claim 2.} Let $(\beta_1,\dots,\beta_\ell)$ denote a
$\mu_2^\alpha$-stack ($(\beta_1,\dots,\beta_\ell)\prec
\mu_2^\alpha$). Then we have
\begin{equation}\label{E:ree}
\mathcal{L}_2=\mathcal{L}_{2,1}\dot\cup\mathcal{L}_{2,2}
\dot\cup\mathcal{L}_{2,3},
\end{equation}
where
\begin{align*}
    \mathcal{L}_{2,1}=&
 \{\lambda\in\mathcal{L}_2\mid
\text{$\alpha, \beta_i\in\lambda, i=1,2$;
$(\beta_1,\beta_2)\prec\mu_2^\alpha$; $\alpha$ crosses $\beta_2$}\},\\
\mathcal{L}_{2,2}=& \{\lambda\in\mathcal{L}_2\mid\text{$\alpha,
\beta_i\in\lambda, i=1,2$;
$(\beta_1,\beta_2)\prec\mu_2^\alpha$; $\alpha$ crosses $\beta_1$}\},\\
    \mathcal{L}_{2,3}=&
\{\lambda\in\mathcal{L}_2\mid\text{$\alpha, \beta_i\in\lambda,
i=1,2,3$; $(\beta_1,\beta_2,\beta_3)\prec\mu_2^\alpha$; $\alpha$
crosses $\beta_2$}\}.
\end{align*}

To prove Claim 2, it suffices to observe that a
$\mathcal{M}_2^\alpha$-matching contains exactly one stack of length
either two or three. Now, eq.~(\ref{E:ree}) immediately follows by
inspection of
Figure~\ref{F:mnest}.\\
\restylefloat{figure}\begin{figure}[h!t!b!p]
\centering
\includegraphics[width=1\textwidth]{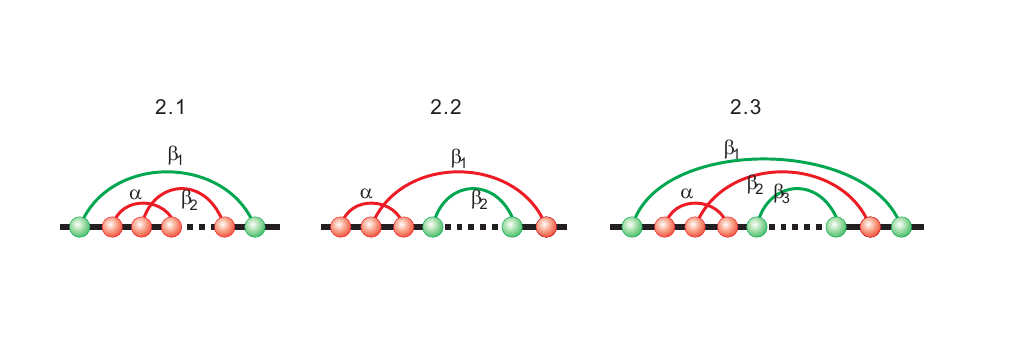}
\caption{\small $\mathcal{L}_2$ and $\mathcal{M}_2^\alpha$: how
stacks arise by the removal of $\alpha$.}\label{F:mnest}
\end{figure}
\noindent{\it Claim 2.1}
\begin{eqnarray*}
|\mathcal{L}_{2,1}|&=&4(u_2+1)\,i_k(s-1,u_1,u_2+1,u_3,u_4)\\
&+&4(u_3+1)\,i_k(s-1,u_1,u_2,u_3+1,u_4)\\
&+&[4(u_4+1)\,i_k(s-1,u_1,u_2,u_3-1,u_4+1)\\
&+&2(u_4+1)\,i_k(s-1,u_1,u_2,u_3,u_4+1)]\\
&+&2((s-1)-u_1-2u_2-2u_3-3u_4))\,i_k(s-1,u_1,u_2,u_3,u_4).
\end{eqnarray*}

To prove Claim 2.1, let $\mathcal{M}_{2,1}^\alpha$ be the set of
matchings induced by removing $\alpha$ from a
$\mathcal{L}_{2,1}$-shape. We set $\mathcal{L}_{2,1}^\alpha$ to be
the set of shapes induced by collapsing the unique
$\mathcal{M}_{2,1}^\alpha$-stack  of length two into the arc
$\beta_2$. Clearly, such a shape cannot exhibit any
additional $\mathbf{C}_1$-elements, see Fig.\ \ref{F:s_211}.\\
\restylefloat{figure}\begin{figure}[h!t!b!p]
\centering
\includegraphics[width=1\textwidth]{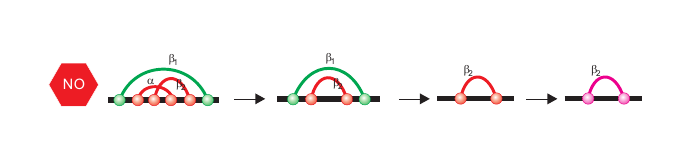}
\caption{\small $\lambda_{2,1}^\alpha$ cannot exhibit any additional
$\mathbf{C}_1$-elements.}\label{F:s_211}
\end{figure}

Since the removal of $\alpha$ and subsequent stack-collapse can lead
to at most one new $\mathbf{C}_i$ element, we have
\begin{equation*}
\mathcal{L}_{2,1} = \mathcal{L}_{2,1}^{\mathbf{C}_2} \dot\cup
\mathcal{L}_{2,1}^{\mathbf{C}_3} \dot\cup
\mathcal{L}_{2,1}^{\mathbf{C}_4} \dot\cup \mathcal{L}_{2,1}^0,
\end{equation*}
where $\mathcal{L}_{2,1}^{\mathbf{C}_i}$, $i=2,3,4$ denotes the set
of labeled shapes, $\lambda\in \mathcal{L}_{2,1}$, that induce
unique shapes having a labeled ${\bf C}_i$-element containing
$\beta_2$ and $\mathcal{L}_{2,1}^{0}$ denotes those in which there
exists no
${\bf C}_i$-element containing $\beta_2$.\\

We first prove
\begin{equation*}
\vert\mathcal{L}_{2,1}^{\mathbf{C}_2}\vert
=4(u_2+1)\,i_k(s-1,u_1,u_2+1,u_3,u_4).\leqno (2.1.\text{I}).
\end{equation*}
\restylefloat{figure}\begin{figure}[h!t!b!p]
\centering
\includegraphics[width=1\textwidth]{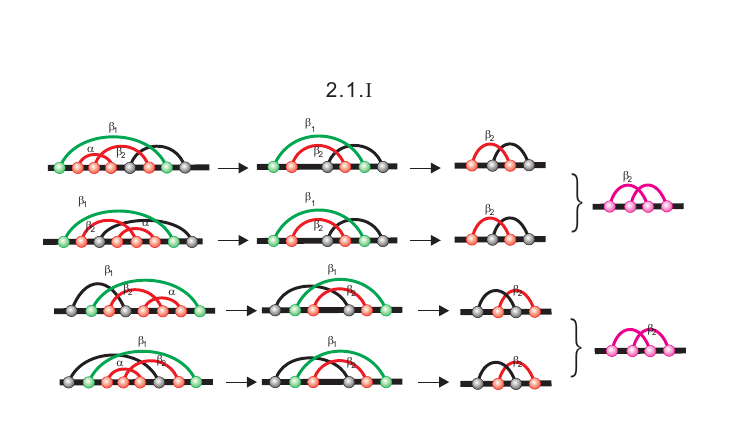}
\caption{\small In $\mathcal{L}_{2,1}$: the removal of $\alpha$ and
subsequent collapse of the unique stack of length two in
$\mathcal{M}_{2,1}^\alpha$, generating a labeled ${\bf C}_2$-element
(pink).}\label{F:s_212}
\end{figure}
\clearpage
Indeed, in order to generate a labeled $\mathbf{C}_2$-element via
$\alpha$-removal and subsequent stack-collapse from a
$\mathcal{L}_{2,1}$-shape, $\beta_2$ has to become a $2$-arc in a
${\bf C}_2$-element of a $\mathcal{I}_k(s-1,u_1,u_2+1,
u_3,u_4)$-shape. We display all possible scenarios in Fig.\
\ref{F:s_212}.\\

Next we prove
\begin{equation*}
\vert\mathcal{L}_{2,1}^{\mathbf{C}_3}\vert
=4(u_3+1)\,i_k(s-1,u_1,u_2,u_3+1,u_4).\leqno (2.1.\text{II})
\end{equation*}
\restylefloat{figure}\begin{figure}[h!t!b!p]
\centering
\includegraphics[width=1\textwidth]{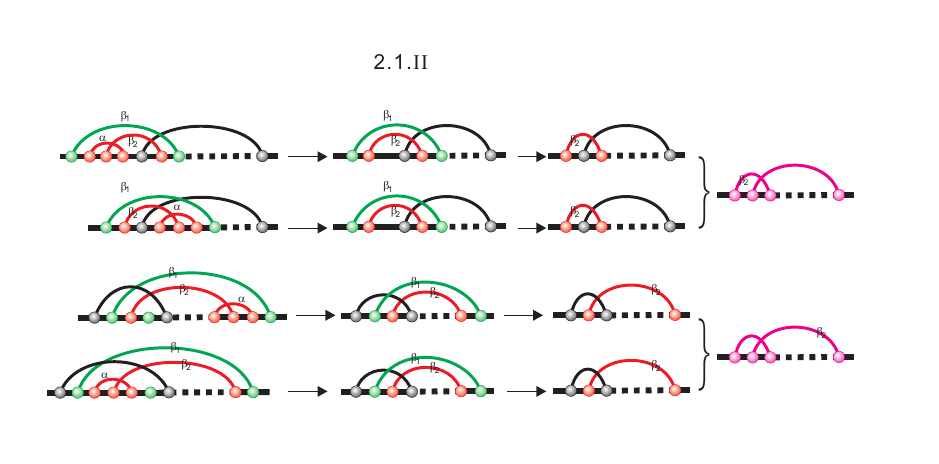}
\caption{\small In $\mathcal{L}_{2,1}$: the removal of $\alpha$ and
collapsing the unique stack of length two in
$\mathcal{M}_{2,1}^\alpha$, generating a labeled ${\bf C}_3$-element
(pink).}\label{F:s_213}
\end{figure}
\noindent In order to generate a labeled ${\bf C}_3$-element by
$\alpha$-removal and collapsing the unique stack of length two from
a $\mathcal{L}_{2,1}$-shape, $\beta_2$ has to become a $2$-arc or an
arc uniquely crossing a $2$-arc in a ${\bf C}_3$-element of a
$\mathcal{I}_k(s-1,u_1,u_2, u_3+1,u_4)$-shape. We display all
possible scenarios in Fig.~\ref{F:s_213}.

Third we prove
\begin{equation*}
\vert\mathcal{L}_{2,1}^{\mathbf{C}_4}\vert
=4(u_4+1)\,i_k(s-1,u_1,u_2,u_3-1,u_4+1)
+2(u_4+1)\,i_k(s-1,u_1,u_2,u_3,u_4+1).\leqno (2.1.\text{III})
\end{equation*}
\restylefloat{figure}\begin{figure}[h!t!b!p]
\centering
\includegraphics[width=1\textwidth]{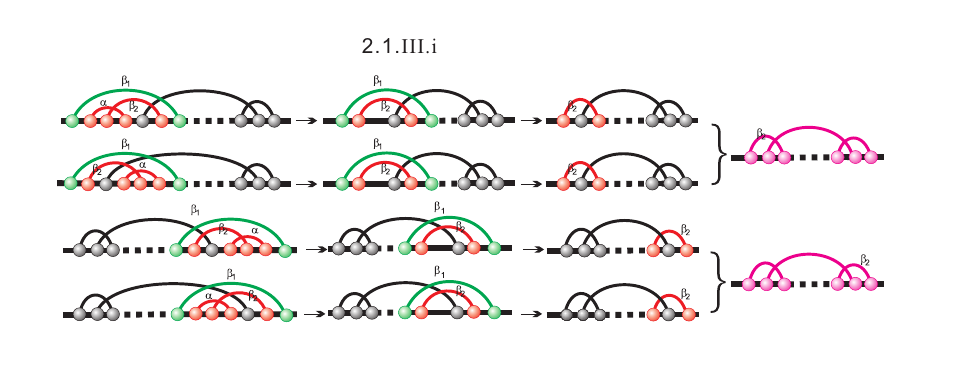}
\caption{\small How to derive a labeled ${\bf C}_4$-element: first
scenario. }\label{F:s_2141}
\end{figure}
\restylefloat{figure}\begin{figure}[h!t!b!p]
\centering
\includegraphics[width=1\textwidth]{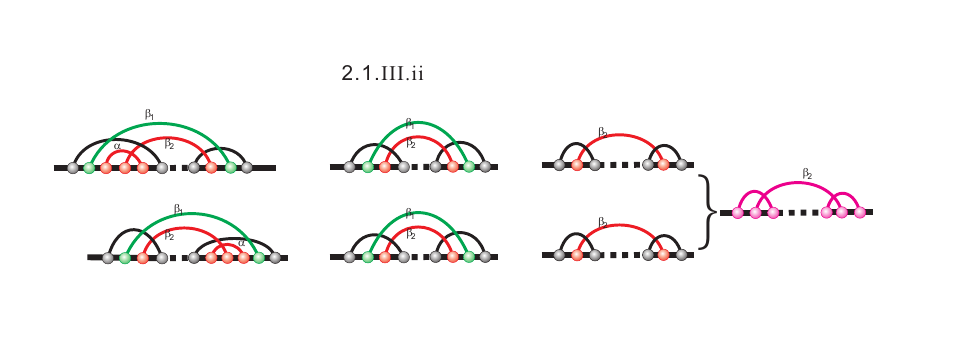}
\caption{\small How to derive a labeled ${\bf C}_4$-element:second scenario.}\label{F:s_2142}
\end{figure}
In order to generate a labeled ${\bf C}_4$-element by
$\alpha$-removal and collapsing the unique stack of length two from
a $\mathcal{L}_{2,1}$-shape, $\beta_2$ has to become either a
$2$-arc in a labeled ${\bf C}_4$-element of a
$\mathcal{I}_k(s-1,u_1,u_2,u_3-1,u_4+1)$-shape or an arc that
crosses two $2$-arcs in a labeled ${\bf C}_4$-element of a
$\mathcal{I}_k(s-1,u_1,u_2,u_3,u_4+1)$-shape. We display all
possible scenarios in Fig.~\ref{F:s_2141} and Fig.~\ref{F:s_2142}. Otherwise, $\beta_2$ becomes a labeled arc in a
$\mathcal{I}_k(s-1,u_1,u_2,u_3,u_4)$-shape, which is not contained
in any ${\bf C}_i$-element. Thus
\begin{equation*}
\vert\mathcal{L}_{2,1}^{0}\vert
=2((s-1)-u_1-2u_2-2u_3-3u_4))\,i_k(s-1,u_1,u_2,u_3,u_4), \leqno
(2.1.\text{IV})
\end{equation*}
from which Claim 2.1 follows.\\
{\it Claim 2.2.
\begin{eqnarray*}
|\mathcal{L}_{2,2}| &=& 2u_1\,i_k(s-1,u_1,u_2,u_3,u_4)\\
&+&4(u_2+1)\,i_k(s-1,u_1,u_2+1,u_3-1,u_4)\\
&+&[2u_3\,i_k(s-1,u_1,u_2,u_3,u_4)\\
&+&2(u_3+1)\,i_k(s-1,u_1,u_2,u_3+1,u_4)]\\
&+&[4u_4\,i_k(s-1,u_1,u_2,u_3,u_4)\\
&+&2(u_4+1)\,i_k(s-1,u_1,u_2,u_3,u_4+1)]\\
&+&2((s-1)-u_1-2u_2-2u_3-3u_4))i_k(s-1,u_1,u_2,u_3,u_4).
\end{eqnarray*}
}
In order to prove Claim 2.2, we consider $\mathcal{M}_{2,2}^\alpha$,
the set of matchings induced by removing $\alpha$ from a
$\mathcal{L}_{2,2}$-shape. We set $\mathcal{L}_{2,2}^\alpha$ to be
the set of shapes induced by collapsing the unique
$\mathcal{M}_{2,2}^\alpha$-stack of length two into $\beta_2$. The
removal of $\alpha$ and subsequent collapse can only lead to at most
one additional $\mathbf{C}_i$-element, whence
\begin{equation*}
\mathcal{L}_{2,2} = \mathcal{L}_{2,2}^{\mathbf{C}_1} \, \dot\cup\,
\mathcal{L}_{2,2}^{\mathbf{C}_2} \, \dot\cup \,
\mathcal{L}_{2,2}^{\mathbf{C}_3} \, \dot\cup \,
\mathcal{L}_{2,2}^{\mathbf{C}_4} \, \dot\cup \, \mathcal{L}_{2,2}^0,
\end{equation*}
using analogous notation and reasoning as in the proof of Claim 2.1.

We first prove
\begin{equation*}
\vert\mathcal{L}_{2,2}^{\mathbf{C}_1}\vert
=2u_1\,i_k(s-1,u_1,u_2,u_3,u_4).\leqno (2.2.\text{I})
\end{equation*}
\restylefloat{figure}\begin{figure}[h!t!b!p]
\centering
\includegraphics[width=1\textwidth]{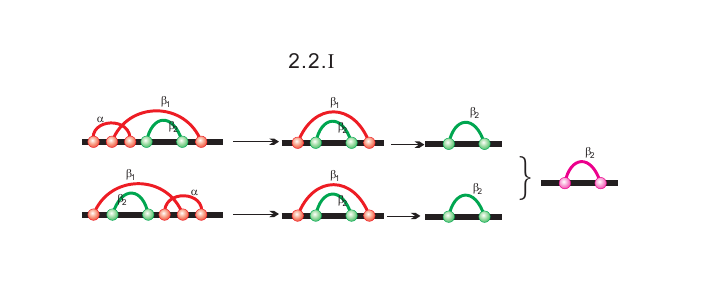}
\caption{\small Removal of $\alpha$ in $\mathcal{L}_{2,2}^\alpha$
and subsequent collapsing of the unique stack in
$\mathcal{M}_{2,2}^\alpha$, generating a labeled ${\bf
C}_1$-element.}\label{F:s_221}
\end{figure}
In order to generate a labeled $\mathbf{C}_1$-element by
$\alpha$-removal from a $\mathcal{L}_{2,2}$-shape and collapsing the
unique stack of length two, we need $\beta_2$ to be a $1$-arc in a
$\mathcal{I}_k(s-1,u_1,u_2,u_3,u_4)$-shape, see
Figure~\ref{F:s_221}. Note that this operation only transfers labels
but generates no new $1$-arcs.\\
Next we prove
\begin{equation*}
\vert\mathcal{L}_{2,2}^{\mathbf{C}_2}\vert
=4(u_2+1)\,i_k(s-1,u_1,u_2+1,u_3-1,u_4).\leqno (2.2.\text{II})
\end{equation*}
\restylefloat{figure}\begin{figure}[h!t!b!p]
\centering
\includegraphics[width=1\textwidth]{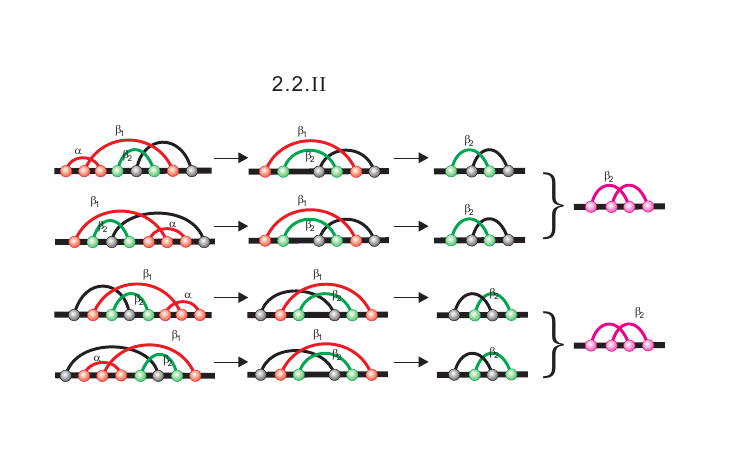}
\caption{\small Removal of $\alpha$ (red arc) in
$\mathcal{L}_{2,2}^\alpha$ and collapsing the unique stack of length
two in $\mathcal{M}_{2,2}^\alpha$, generating a labeled ${\bf
C}_2$-element.}\label{F:s_222}
\end{figure}
In order to generate a labeled ${\bf C}_2$-element by
$\alpha$-removal from a $\mathcal{L}_{2,1}$-shape and collapsing the
unique stack of length two, $\beta_2$ has to become a $2$-arc in a
labeled ${\bf C}_2$-element of
$\mathcal{I}_k(s-1,u_1,u_2+1,u_3-1,u_4)$. We display all possible
scenarios in Fig.~\ref{F:s_222}.\\

Third we prove
\begin{equation*}
\vert\mathcal{L}_{2,2}^{\mathbf{C}_3}\vert
=2u_3\,i_k(s-1,u_1,u_2,u_3,u_4)+
2(u_3+1)\,i_k(s-1,u_1,u_2,u_3+1,u_4).\leqno (2.2.\text{III})
\end{equation*}
\restylefloat{figure}\begin{figure}[h!t!b!p]
\centering
\includegraphics[width=1\textwidth]{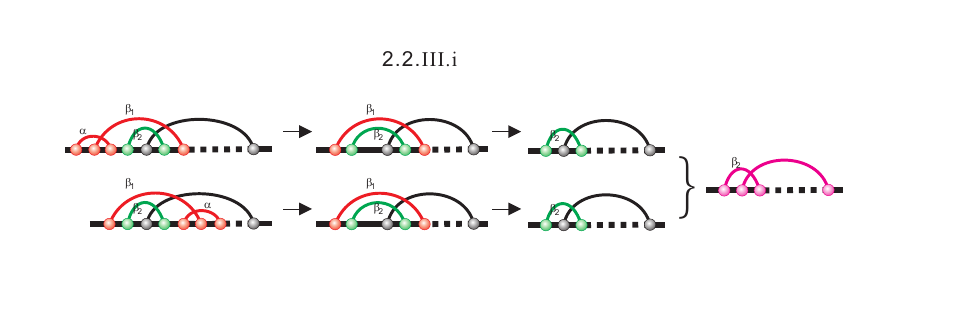}
\caption{\small The term $2u_3\,i_k(s-1,u_1,u_2,u_3,u_4)$: removal
of $\alpha$ in $\mathcal{L}_{2,2}^\alpha$ and collapsing the unique
stack of length two in $\mathcal{M}_{2,2}^\alpha$, generating a
labeled ${\bf C}_3$-element.}\label{F:s_2231}
\end{figure}
\restylefloat{figure}\begin{figure}[h!t!b!p]
\centering
\includegraphics[width=1\textwidth]{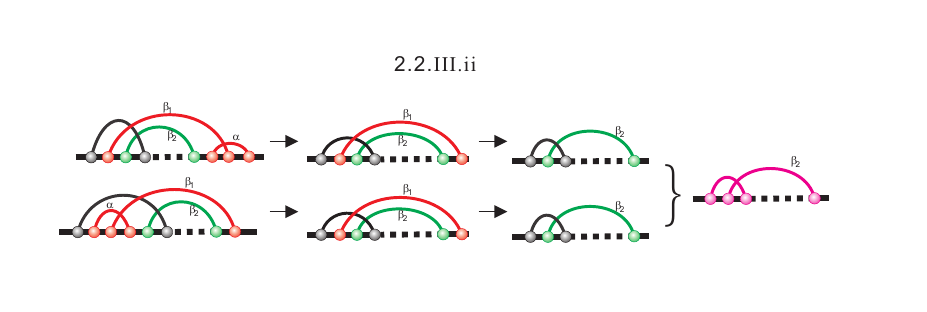}
\caption{\small The term
$2(u_3+1)\,i_k(s-1,u_1,u_2,u_3+1,u_4)$.}\label{F:s_2232}
\end{figure}
In order to generate a labeled ${\bf C}_3$-element by
$\alpha$-removal from a $\mathcal{L}_{2,2}$-shape and collapse of
the resulting unique stack, $\beta_2$ has to become either a $2$-arc
in a labeled ${\bf C}_3$-element of a
$\mathcal{I}_k(s-1,u_1,u_2,u_3,u_4)$-shape or an arc uniquely
crossing the $2$-arc in a labeled ${\bf C}_3$-element of a
$\mathcal{I}_k(s-1,u_1,u_2,u_3+1,u_4)$-shape. We display all
possible scenarios in Fig.~\ref{F:s_2231} and
Fig.\ref{F:s_2232}.\\

Fourth we prove
\begin{equation*}
\vert\mathcal{L}_{2,2}^{\mathbf{C}_4}\vert
=4u_4\,i_k(s-1,u_1,u_2,u_3,u_4)
+2(u_4+1)\,i_k(s-1,u_1,u_2,u_3,u_4+1).\leqno (2.2.\text{IV})
\end{equation*}
\restylefloat{figure}\begin{figure}[h!t!b!p]
\centering
\includegraphics[width=1\textwidth]{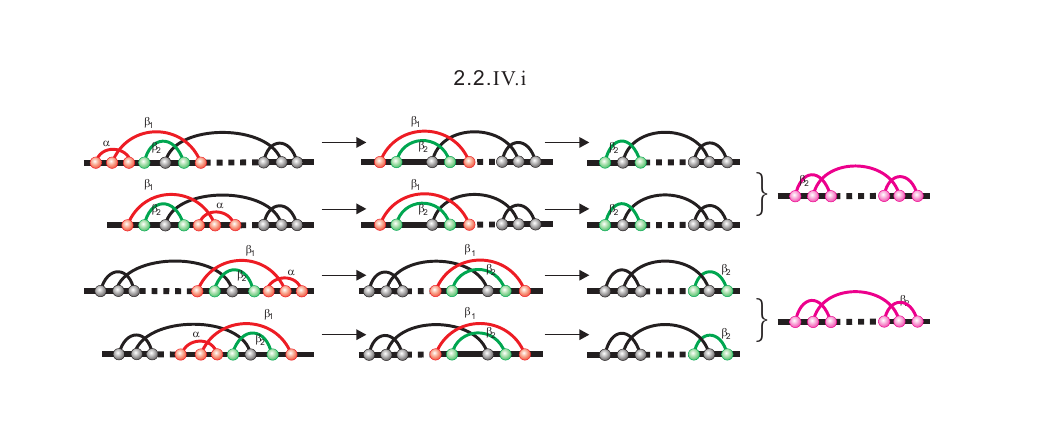}
\caption{\small The term
$4u_4\,i_k(s-1,u_1,u_2,u_3,u_4)$.}\label{F:s_2241}
\end{figure}
\restylefloat{figure}\begin{figure}[h!t!b!p]
\centering
\includegraphics[width=1\textwidth]{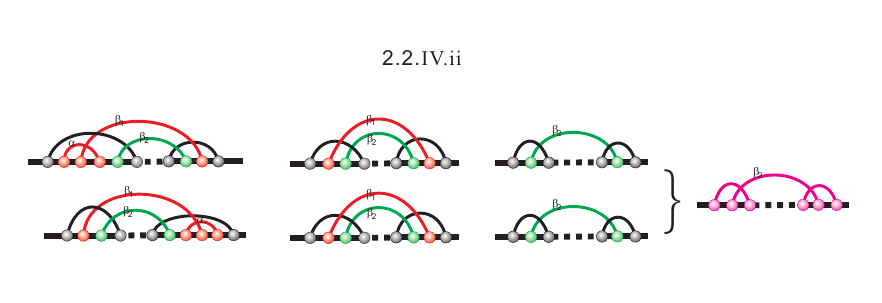}
\caption{\small The term
$2(u_4+1)\,i_k(s-1,u_1,u_2,u_3,u_4+1)$.}\label{F:s_2242}
\end{figure}
In order to generate a labeled ${\bf C}_4$-element, $\beta_2$ has to
become either a $2$-arc in a labeled ${\bf C}_4$-element of a
$\mathcal{I}_k(s-1,u_1,u_2,u_3,u_4)$-shape or an arc uniquely
crossing two $2$-arcs in a labeled ${\bf C}_4$-element of a
$\mathcal{I}_k(s-1,u_1,u_2,u_3,u_4+1)$-shape. We display all
possible scenarios in Fig.~\ref{F:s_2241} and Fig.~\ref{F:s_2242}.\\
It remains to observe that $\beta_2$ otherwise becomes a labeled arc
in a $\mathcal{I}_k(s-1,u_1,u_2,u_3,u_4)$-shape, which is not
contained in any ${\bf C}_i$-element. Thus
\begin{equation*}
\vert\mathcal{L}_{2,2}^{0}\vert=
2((s-1)-u_1-2u_2-2u_3-3u_4))\,i_k(s-1,u_1,u_2,u_3,u_4) \leqno
(2.2.\text{V})
\end{equation*}
and Claim 2.2 follows.\\

\noindent{\it Claim 2.3
\begin{eqnarray*}
|\mathcal{L}_{2,3}|&=&2u_1\,i_k(s-2,u_1,u_2,u_3,u_4)\\
&+&4(u_2+1)\,i_k(s-2,u_1,u_2+1,u_3-1,u_4)\\
&+&[2u_3\,i_k(s-2,u_1,u_2,u_3,u_4)\\
&+&2(u_3+1)i_k(s-2,u_1,u_2,u_3+1,u_4)]\\
&+&[4u_4i_k(s-2,u_1,u_2,u_3,u_4)\\
&+&2(u_4+1)i_k(s-2,u_1,u_2,u_3,u_4+1)]\\
&+&2((s-2)-u_1-2u_2-2u_3-3u_4))i_k(s-2,u_1,u_2,u_3,u_4).
\end{eqnarray*}
} Let $\mathcal{M}_{2,3}^\alpha$ be the set of matchings induced by
removing $\alpha$ from a $\mathcal{L}_{2,3}$-shape. Let
$\mathcal{L}_{2,3}^\alpha$ denote the set of shapes induced by
collapsing the unique $\mathcal{M}_{2,3}^\alpha$-stack of length
three into the arc $\beta_3$. The removal of $\alpha$ and subsequent
stack-collapse can only lead to at most one additional
$\mathbf{C}_i$ ($i=1,2,3,4$) element, whence
\begin{equation*}
\mathcal{L}_{2,3} = \mathcal{L}_{2,3}^{\mathbf{C}_1} \, \dot\cup\,
\mathcal{L}_{2,3}^{\mathbf{C}_2} \, \dot\cup \,
\mathcal{L}_{2,3}^{\mathbf{C}_3} \, \dot\cup\,
\mathcal{L}_{2,3}^{\mathbf{C}_4}\, \dot\cup\, \mathcal{L}_{2,3}^0,
\end{equation*}
where $\mathcal{L}_{2,3}^{\mathbf{C}_i}$ denotes the set of labeled
shapes, $\lambda \in \mathcal{L}_{2,3}$, that induce unique shapes
having a labeled ${\bf C}_i$-element containing $\beta_3$ and
$\mathcal{L}_{2,3}^{0}$ denotes those shapes in which there exists
no such  ${\bf C}_i$-element.\\

We first note
\begin{equation*}
\vert\mathcal{L}_{2,3}^{\mathbf{C}_1}\vert
=2u_1\,i_k(s-2,u_1,u_2,u_3,u_4),\leqno (2.3.\text{I})
\end{equation*}
see Fig.\ \ref{F:s_231}.\\
\restylefloat{figure}\begin{figure}[h!t!b!p]
\centering
\includegraphics[width=1\textwidth]{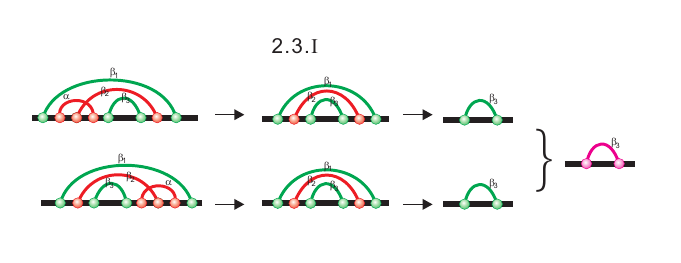}
\caption{\small The term $2u_1\,i_k(s-2,u_1,u_2,u_3,u_4)$.}
\label{F:s_231}
\end{figure}

Next we observe
\begin{equation*}
\vert\mathcal{L}_{2,3}^{\mathbf{C}_2}\vert
=4(u_2+1)\,i_k(s-2,u_1,u_2+1,u_3-1,u_4),\leqno (2.3.\text{II})
\end{equation*}
see Fig.~\ref{F:s_232}.\\
\restylefloat{figure}\begin{figure}[h!t!b!p]
\centering
\includegraphics[width=1\textwidth]{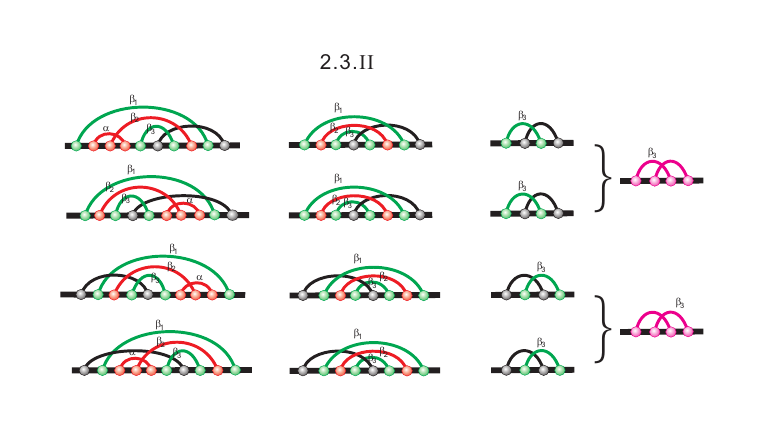}
\caption{\small The term $4(u_2+1)\,i_k(s-2,u_1,u_2+1,u_3-1,u_4)$.}
\label{F:s_232}
\end{figure}

Third we verify
\begin{equation*}
\vert\mathcal{L}_{2,3}^{\mathbf{C}_3}\vert
=2u_3\,i_k(s-2,u_1,u_2,u_3,u_4)+2(u_3+1)\,
i_k(s-2,u_1,u_2,u_3+1,u_4),\leqno (2.3.\text{III})
\end{equation*}
see Fig.~\ref{F:s_2331} and Fig.~\ref{F:s_2332}.\\
\restylefloat{figure}\begin{figure}[h!t!b!p]
\centering
\includegraphics[width=1\textwidth]{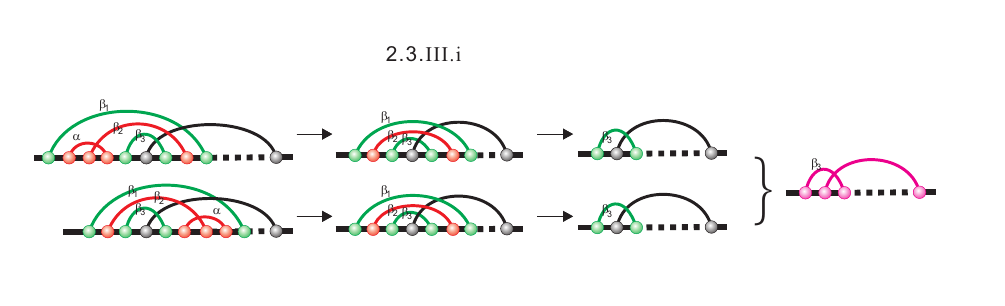}
\caption{\small The term
$2u_3\,i_k(s-2,u_1,u_2,u_3,u_4)$.}\label{F:s_2331}
\end{figure}
\restylefloat{figure}\begin{figure}[h!t!b!p]
\centering
\includegraphics[width=1\textwidth]{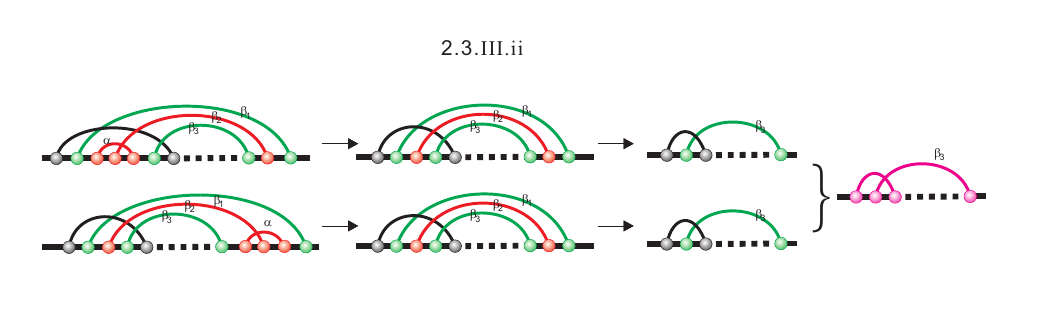}
\caption{\small The term
$2(u_3+1)\,i_k(s-2,u_1,u_2,u_3+1,u_4)$.}\label{F:s_2332}
\end{figure}

Fourth we note
\begin{equation*}
\vert\mathcal{L}_{2,3}^{\mathbf{C}_4}\vert =
4u_4\,i_k(s-2,u_1,u_2,u_3,u_4)
+2(u_4+1)\,i_k(s-2,u_1,u_2,u_3,u_4+1),\leqno (2.3.\text{IV})
\end{equation*}
see Fig.~\ref{F:s_2341} and Fig.~\ref{F:s_2342}.\\
\restylefloat{figure}\begin{figure}[h!t!b!p]
\centering
\includegraphics[width=1\textwidth]{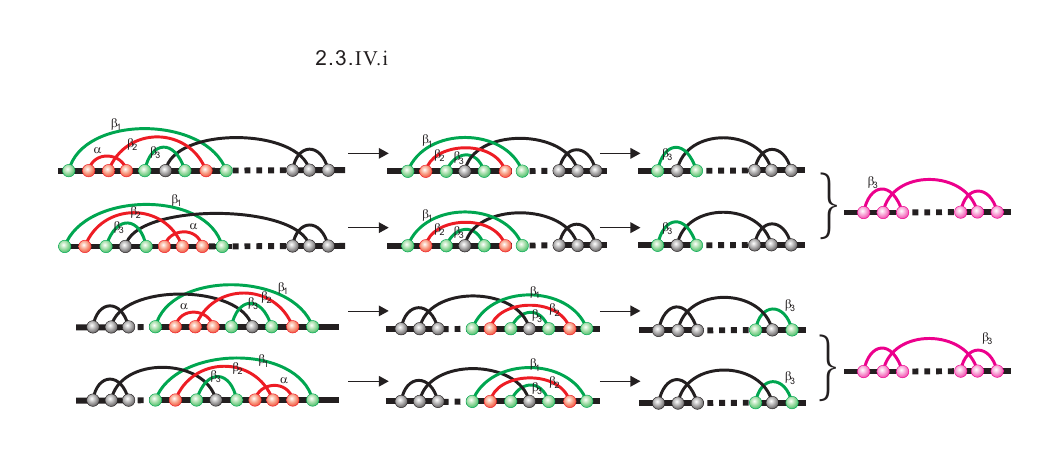}
\caption{\small The term
$4u_4\,i_k(s-2,u_1,u_2,u_3,u_4)$.}\label{F:s_2341}
\end{figure}
\restylefloat{figure}\begin{figure}[h!t!b!p]
\centering
\includegraphics[width=1\textwidth]{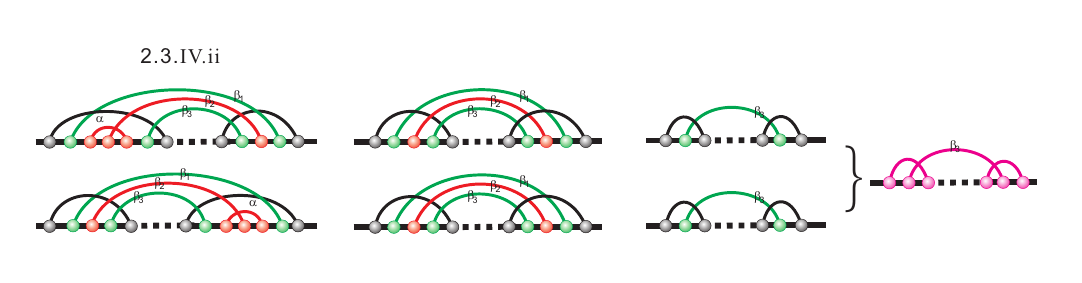}
\caption{\small The term
$2(u_4+1)\,i_k(s-2,u_1,u_2,u_3,u_4+1)$.}\label{F:s_2342}
\end{figure}

It remains to observe that $\beta_3$ becomes otherwise a labeled arc
in a $\mathcal{I}_k(s-1,u_1,u_2,u_3,u_4)$-shape, which is not
contained in any ${\bf C}_i$-element. Thus
\begin{equation*}
\vert\mathcal{L}_{2,3}^{0}\vert =2((s-2)-u_1-2u_2-2u_3-3u_4))\,
i_k(s-2,u_1,u_2,u_3,u_4) \leqno (2.3.\text{V})
\end{equation*}
and Claim 2.3 follows.\\

Eq. (\ref{u3recursion}) now follows from Claim 1, 2.1, 2.2, and
Claim 2.3.\\

Next we prove eq. (\ref{u4recursion}). We choose some $\eta\in
\mathcal{I}_k(s+1,u_1,u_2,u_3,u_4+1)$ and label one ${\bf
C}_4$-element denoting one of its two $2$-arcs by $\alpha$. We
denote the set of these labeled shapes, $\lambda$, by
$\mathcal{L}_*$. Clearly,
\begin{equation*}
\vert \mathcal{L}_*\vert =2(u_4+1)\,i_k(s+1,u_1,u_2,u_3,u_4+1).
\end{equation*}
Let $\gamma$ be the arc crossing $\alpha$. The removal of $\alpha$
can lead to either an additional ${\bf C}_2$- or an additional ${\bf
C}_3$-element in a shape, whence
\begin{equation}
\mathcal{L_*}= \mathcal{L}_*^{{\bf C}_2}\, \dot\cup\,
\mathcal{L}_*^{{\bf C}_3},
\end{equation}
where $\mathcal{L}_*^{{\bf C}_i}$ denotes the set of labeled shapes,
$\lambda\in\mathcal{L}_*$, that induce shapes having
a labeled ${\bf C}_i$-element containing $\gamma$.\\

First,
\begin{equation*}
\vert\mathcal{L}_*^{{\bf
C}_2}\vert=2(u_2+1)\,i_k(s,u_1,u_2+1,u_3,u_4),
\end{equation*}
follows by inspection of Fig.\ \ref{F:u4_2}.\\
\restylefloat{figure}\begin{figure}[h!b!p!t]
\centering
\includegraphics[width=1\textwidth]{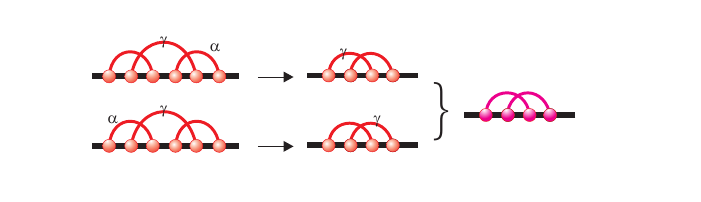}
\caption{\small The term
$2(u_2+1)\,i_k(s,u_1,u_2+1,u_3,u_4)$.}\label{F:u4_2}
\end{figure}

We next observe
\begin{equation*}
\vert\mathcal{L}_*^{{\bf
C}_3}\vert=(u_3+1)\,i_k(s,u_1,u_2,u_3+1,u_4).
\end{equation*}
see Fig.\ \ref{F:u4_1}, from which eq. \ref{u4recursion} immediately
follows.\\
\restylefloat{figure}\begin{figure}[h!b!p!t]
\centering
\includegraphics[width=1\textwidth]{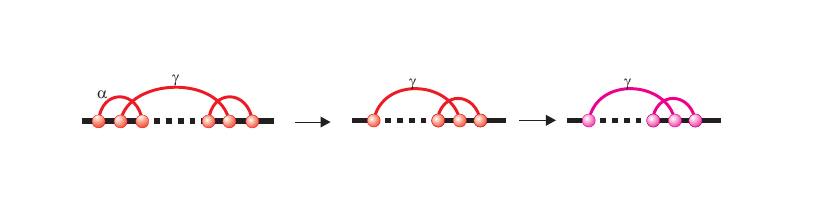}
\caption{\small The term
$(u_3+1)\,i_k(s,u_1,u_2,u_3+1,u_4)$.}\label{F:u4_1}
\end{figure}

It remains to show that the numbers $i_k(s,u_1,u_2,u_3,u_4)$ can be
uniquely derived from eq.~(\ref{E:erni2}), eq.~(\ref{E:5ini}),
eq.~(\ref{u3recursion}) and eq.~(\ref{u4recursion}). This follows by
induction on $s$.
\end{proof}

\newchap{Asymptotic enumeration of canonical $3$-noncrossing skeleton diagrams}
\thispagestyle{fancyplain}

A {\it skeleton diagram}, $S$, is a labeled graph whose core has no noncrossing arcs and dependency graph is connected. Here {\it dependency graph} is defined as a graph whose each vertex corresponds to each arc of diagram and two vertices are adjacent if and only if two corresponding arcs cross, see Fig.~\ref{F:dependency}. Recall that an {\it interval} is a sequence of consecutive, unpaired
bases $(i, i + 1, \ldots, j)$, where $i - 1$ and $j + 1$ are paired. The skeleton diagram and the interval have been present in Fig.~\ref{F:skeleton-intro} vividly.

In this chapter, we will discuss about the generating function equation of 3-noncrossing skeleton matching and conclude the explicit generating function of 3-noncrossing, 3-canonical skeleton diagrams in Section \ref{S:ca-sk-gf}. At last we will study the asymptotic enumeration of numbers of skeleton matchings and skeleton structures, and discuss the asymptotic distribution of arcs in a canonical $3$-noncrossing skeleton diagrams with certain length.


\restylefloat{figure}\begin{figure}[h!b!p!t]
\centering
\includegraphics{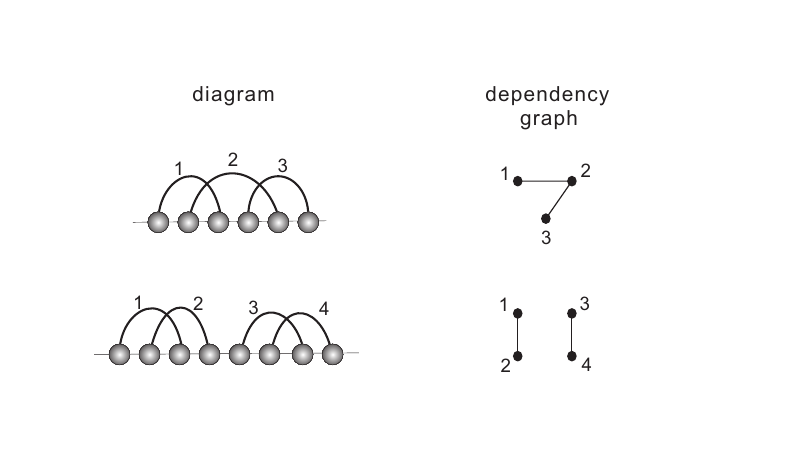}
\caption{Diagram and dependency graph: we can clearly see that the top diagram's dependency graph is connected while the bottom one's is not, which implies that skeleton diagram is the top diagram, not the bottom one.\label{F:dependency}}
\end{figure}

\section{Shapes and irreducible structures}

\begin{definition}
A ${\sf V}_k$-shape is a $k$-noncrossing matching having stacks of length
exactly one.
\end{definition}
In the following part we refer to ${\sf V}_k$-shape satisfying  skeleton's property simply as {\it skeleton shapes}.
That is, given a $3$-noncrossing skeleton, $\delta$, its shape
is obtained by first replacing each stem by an arc and then removing all
isolated vertices, see Fig.~\ref{F:skeleton-shape}.

\restylefloat{figure}\begin{figure}[h!b!t!p]
\centerline{\includegraphics[width=0.65\textwidth]{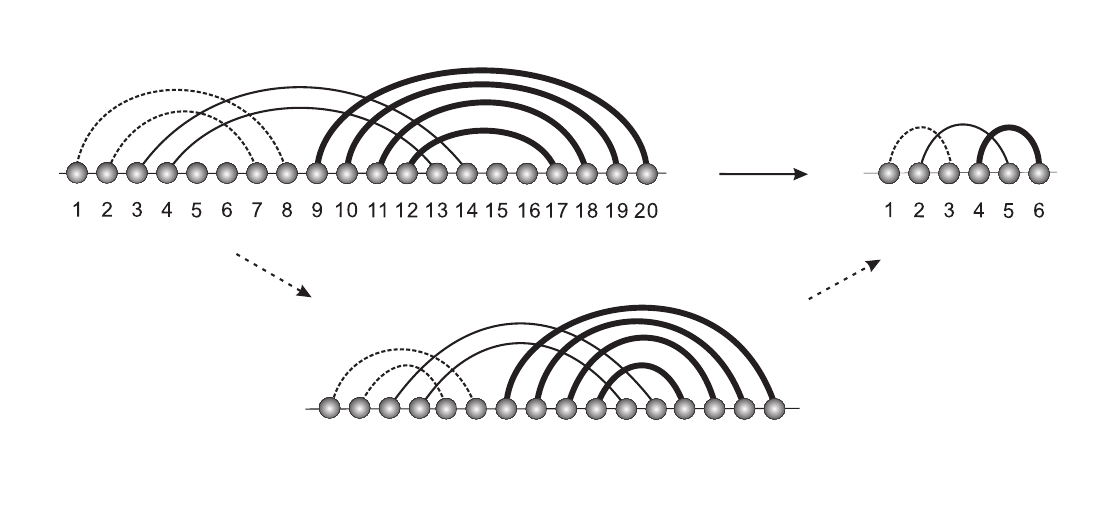}}
\caption{\small From diagrams to shapes: A $3$-noncrossing skeleton
diagram (top-left) is mapped in two steps into its skeleton shape
(top-right).} \label{F:skeleton-shape}
\end{figure}

\begin{definition}\label{D:irreducible}
A $k$-noncrossing diagram of length $n$ is {\it irreducible} (see Fig. \ref{F:irreducible}), if there does not exist such a vertex $\ell\in\{1,2,\ldots, n\}$, which separates the set of arcs into two disjoint subsets, ${\bf A}_1$ and ${\bf A}_2$,
\begin{eqnarray*}
{\bf A}_1&=&\{(i,j)\mid j<\ell\},\\
{\bf A}_2&=& \{(i,j)\mid i\geq\ell\}.
\end{eqnarray*}
\end{definition}

\restylefloat{figure}\begin{figure}[h!b!t!p]
\centerline{\includegraphics{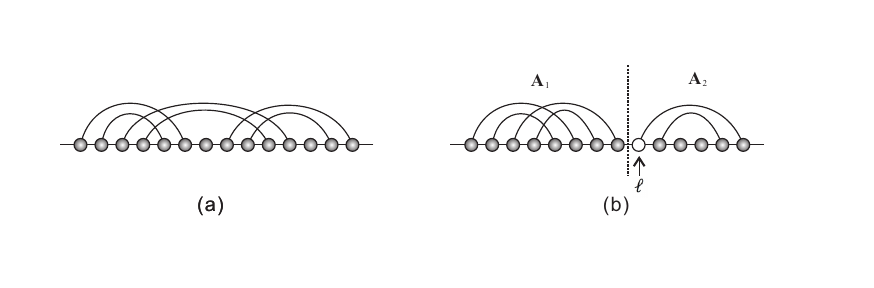}}
\caption{\small Irreducible and un-irreducible: ({\sf a}) shows the irreducible $3$-noncrossing diagram, in which there is no such vertex $\ell$ defined in Definition \ref{D:irreducible}. However, in ({\sf b}), the vertex pointed by an arrow is the $\ell$ which separates the set of arcs into ${\bf A}_1$ and ${\bf A}_2$.\label{F:irreducible}}
\end{figure}



\section{$3$-noncrossing skeleton matching}\label{S:skeleton-shape}


In this section we enumerate $3$-noncrossing skeleton matching and conclude a functional equation.

\begin{definition}\label{D:skeleta_matching}
A diagram is named by {\bf matching}, if and only if the diagram has no isolated vertices.
\end{definition}
Now we begin to compute the generating function of skeleton matching. We denote ${\sf Irr}(h)$ and ${\sf S}(h)$ as the number of irreducible matchings and skeleton matchings respectively. These two matchings both satisfy $3$-noncrossing property and contain $h$ arcs. In term of the definition of skeleton diagram, the arc number of a skeleton diagram is two at least. For showing the generating function of skeleton matching from the constant term, we suppose ${\sf S}(0)={\sf S}(1)=1$.

\begin{lemma}\label{irreducible1}
Let ${\bf Irr}(z)=\sum_{h\geq0}{\sf Irr}(h)z^h$ be the OGF of $3$-noncrossing irreducible diagram, then we have
the following functional equation,
\begin{equation}
1+\sum_{h\geq 1}{\sf S}(h){\bf F}_3(z)^{2h-1}z^h={\bf Irr}(z).
\end{equation}
\end{lemma}

\begin{proof}
For an arbitrary irreducible structure $\mathcal
{R}$ (See Figure \ref{IR}), we denote $\mathcal{B}$ be a set of arcs
of $\mathcal{R}$. Let $(i,j)$ be an arbitrary arc in $\mathcal{B}$,
then there will be no arcs such as $(i',j')$, which satisfies $i'<i,
j'>j$. We call such $\mathcal{B}$ the boundary of
$\mathcal{R}$. Then we set
\[
\mathcal{S}=\{\text{arc}\ \alpha\in\mathcal{R}: \alpha\
\text{crosses some arc in}\ \mathcal{B}\}
\]
\restylefloat{figure}\begin{figure}[!htbp]
\centering
\includegraphics[width=1\textwidth]{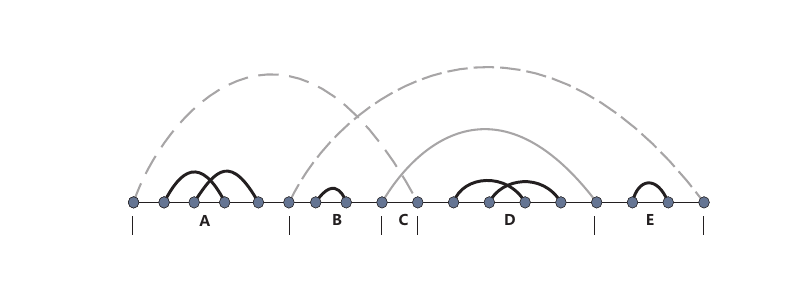}
\caption{The irreducible structure $\mathcal {R}$: The broken lines
are the {\it boundary} of $\mathcal {R}$. The grey lines are the
{\it skeleton} of $\mathcal {R}$, denoted by $\mathcal {S}$. And
letters A, B, C, D, E respectively represent the {\it intervals} of
$\mathcal {S}$.\label{IR}}
\end{figure}
If $\mathcal{S}\neq\emptyset$, we will prove $\mathcal{B}$ is a skeleton.

Assume $\mathcal{B}$ is not a skeleton, without losing generality, by definition, the dependency graph of $\mathcal{B}$ can be separated into several connected components.
Without losing generality, we set that the dependency graph of $\mathcal{B}$ contains two connected components, $G_1$ and $G_2$.

Therefore, we have
\begin{align}
\mathcal{R}& = \bigcup_{(i,j)\in \mathcal{B}}\bigcup_{i\leq i'<j'\leq j}(i',j')\nonumber\\
&= \left(\bigcup_{(i,j)\in V(G_1)}\bigcup_{i\leq i'<j'\leq j}(i',j')\right)\cup \left(\bigcup_{(i,j)\in V(G_2)}\bigcup_{i\leq i'<j'\leq j}(i',j')\right)\label{irreducible_part}.
\end{align}

In eq. (\ref{irreducible_part}), we denote the first arcs set as $\mathcal{R}_1$ and the second arcs set as $\mathcal{R}_2$. In $\mathcal{R}_1$, let the left and right most points be $p$ and $q$ respectively, while in $\mathcal{R}_2$, let the left and right most points be $k$ and $l$ respectively. Obviously $p, q, k, l$ are all the vertices of $\mathcal{B}$, hence we have $p<q<k<l$ or $k<l<p<q$, which means $\mathcal{R}_1$ and $\mathcal{R}_2$ are disjoint, moreover,
\begin{align*}
\mathcal{R}_1&=\{(i,j)\mid j<q+1\},\\
\mathcal{R}_2&=\{(i,j)\mid i\geq q+1\},
\end{align*}
or
\begin{align*}
\mathcal{R}_1&=\{(i,j)\mid i\geq p\},\\
\mathcal{R}_2&=\{(i,j)\mid j< p\}.
\end{align*}

It is contradicted to the irreducibility of $\mathcal{R}$. Therefore, $\mathcal{B}$ is a skeleton. It immediately leads that $\mathcal{S}$ is also a skeleton.

Then we discuss about
$\mathcal{R}\backslash\mathcal{S}$. Let arc
$\beta\in\mathcal{R}\backslash\mathcal{S}$. Trivially, $\beta$ must
be in the interval of the skeleton $\mathcal{S}$. It means that all
arcs in $\mathcal{R}\backslash\mathcal{S}$ are in the intervals of
the skeleton $\mathcal{S}$, in other words, in each interval of
$\mathcal{S}$, there is a structure of $3$-noncrossing perfect
matching, denoted by $\mathcal{F}_3$, including empty set.

Then we get the following conclusion. Let
$\mathcal{R}_{\mathcal{S}}$ be a subset of irreducible structure
$\mathcal{R}$, all of whose boundaries induce the skeleton structure
$\mathcal{S}$, then we have the symbolic relation,
\begin{equation*}
\mathcal{R}=\emptyset+\mathcal{R}_\emptyset+\left(\sum_{\mathcal{S}\in{\mathbb{S}}}\mathcal{R}_{\mathcal{S}}\right),
\end{equation*}
where $\mathbb{S}$ is the set of all the skeleton matching structures.

Furthermore, we denote the arc number of some $\mathcal{S}$ as  ${\rm a}(\mathcal{S})$, then
such $\mathcal{S}$ has $2{\rm a}(\mathcal{S})-1$ intervals (if $\mathcal{S}=\emptyset$, there will be only
one interval). According this we have the following
symbolic relation,
\begin{align}
\mathcal{R}&=\emptyset+\mathcal{A}\times\mathcal{F}_3+\sum_{\mathcal{S}\in\mathbb{S}} \mathcal{A}^{{\rm a}(\mathcal{S})}\times(\mathcal{F}_3)^{2{\rm a}(\mathcal{S})-1}\nonumber\\
&=\emptyset+\mathcal{A}\times\mathcal{F}_3+\sum_{h\geq2} {\sf S}(h)\left[\mathcal{A}^{h}\times(\mathcal{F}_3)^{2h-1}\right],\label{symbolicrelation}
\end{align}
where $\mathcal{A}$ is the class of arcs. According to eq. (\ref{symbolicrelation}) and $\mathcal{A}\doteq z$ we have the functional equation
\begin{equation}
{\bf Irr}(z)=1+\sum_{h\geq1}{\sf S}(h)z^h({\bf F}_3(z))^{2h-1}.
\end{equation}

\end{proof}

\begin{lemma}\label{irreducible2}
The generating function of ${\bf Irr}(z)$ satisfies the following equation,
\begin{equation*}
{\bf Irr}(z)=2-{1 \over {\bf F}_3(z)}.\label{irreducible}
\end{equation*}
\end{lemma}
\begin{proof}
Let $\mathcal{R}$ be the combinatorial class of irreducible $3$-noncrossing matching structures, and $\mathcal{F}_3$ be the class of $3$-noncrossing matching structures. Then we derive the following symbolic relation:
\begin{equation}
\mathcal{F}_3=\emptyset+(\mathcal{R}-\emptyset)\times\mathcal{F}_3.
\end{equation}
Therefore we have the functional equation that
\begin{equation}
{\bf F}_3(z)=1+({\bf Irr}(z)-1){\bf F}_3(z).
\end{equation}
Transform the equation above, we can get the eq. (\ref{irreducible}).
\end{proof}

\begin{theorem}
Let ${\bf S}(z)$ be the generating function of ${\sf S}(h)$, then, we have
\begin{equation}
{\bf S}(z{\bf F}_3(z)^{2})={\bf F}_3(z).
\end{equation}
\end{theorem}

\begin{proof}
According to Lemma \ref{irreducible1} and Lemma \ref{irreducible2}, we
easily obtain
\begin{equation*}
1+\sum_{h\geq 1}{\sf S}(h){\bf F}_3(z)^{2h-1}z^h=2-{1 \over {{\bf F}_3(z)}}.
\end{equation*}
Do some proper transformations for the equation above, we get
\begin{equation*}
1+\sum_{h\geq 1}{\sf S}(h){\bf F}_3(z)^{2h}z^h={\bf F}_3(z).
\end{equation*}
Since ${\sf S}(0)=1$, then
\begin{equation*}
\sum_{h\geq 0}{\sf S}(h){\bf F}_3(z)^{2h}z^h={\bf F}_3(z).
\end{equation*}
viz. ${\bf S}(z{\bf F}_3(z)^{2})={\bf F}_3(z)$.
\end{proof}

\section{Canonical $3$-noncrossing skeleton diagram}\label{S:ca-sk-gf}

In this section we mainly discuss about the way of obtaining the OGF of $3$-canonical, $3$-noncrossing skeleton diagrams,  ${\bf S}^{[4]}_{3}(z)=\sum_{n=0}^\infty {\sf S}^{[4]}_{3}(n) z^n$ . And from now on, we simply note  $3$-canonical as {\it canonical}.

\begin{theorem}\label{T:skeleta-gen}
Let $z$ be an indeterminate, then ${\bf S}^{[4]}_{3}(z)$, the generating
function of $3$-noncrossing skeleton diagrams with $\sigma\geq 3$, $\lambda\geq 4$, is
given by
\begin{equation}\label{s34GF}
{\bf S}^{[4]}_{3}(z)=(1-z){\bf G}\left(\left({{\sqrt{w_0(z)}z}\over{1-z}}\right)^2\right),
\end{equation}
where
\begin{eqnarray*}
{\bf G}(x)={\bf S}(x)-1-x,
w_0(z)={z^4 \over {1-z^2+z^6}}.
\end{eqnarray*}
\end{theorem}

\begin{proof}
\ \\
{\it Claim 1. Let ${\bf IS}(z)=\sum_{n\geq 2}{\sf IS}(h)z^h$ be the generating function of the $3$-noncrossing skeleton-shape with $h$ arcs. Then we have
\begin{equation}\label{E:IStoSk}
{\bf IS}(z)=\sum_{h\geq 2}{\sf S}(h)\left(z\over{1+z}\right)^h.
\end{equation}
}

Let $\gamma$ be a shape in the set of skeleton-shape $\mathcal{IS}(h)$, and $\mathcal{S}_\gamma$ be the skeleton matching inflated from a shape $\gamma$. Obviously $\mathcal{S}_\gamma$ is derived from inflating each shape-arc to a stack, then we have the symbolic relation below,
\begin{equation}
\mathcal{S}_\gamma=\left(\mathcal{A}\times \textsc{Seq}(\mathcal{A})\right)^h,
\end{equation}
where $\mathcal{A}$ is the class of arcs.

Since $\mathcal{A}=z$, we derive the generating function of $\mathcal{S}_\gamma$ that
\begin{equation}
\mathcal{S}_\gamma(z)=\left(z\over{1-z}\right)^h.
\end{equation}
Therefore we have
\begin{align*}
\sum_{h\geq 2}{\sf S}(h)z^h&=\sum_{h\geq 2}\sum_{\gamma\in\mathcal{IS}(h)}\mathcal{S}_\gamma(z)\\
&=\sum_{h\geq 2}{\sf IS}(h)\left(z\over{1-z}\right)^h\\
&={\bf IS}\left(z\over{1-z}\right).
\end{align*}
Transform the variable in the equation above, we can get eq. (\ref{E:IStoSk}).

{\it Claim 2. the generating function ${\bf S}^{[4]}_{3}(z)$ satisfies the following equation,
\begin{equation}\label{SK34toSk}
{\bf S}^{[4]}_{3}(z)=(1-z){\bf IS}\left(z^6\over{1-2z+2z^3-z^4-2z^7+z^8}\right).
\end{equation}}

Let ${\mathcal S}^{[4]}_{3}$ denote the set of $3$-noncrossing,
$3$-canonical skeleton structures with arc length $\geq$ 4 and ${\mathcal IS}$ the set of
all $k$-noncrossing skeleton-shapes and ${\mathcal IS}(h)$ those having $h$
arcs, see Figure~\ref{F:skeleton-shape}.

Then we have the surjective map
\begin{equation}\label{map}
\varphi\colon {\mathcal S}^{[4]}_{3}\rightarrow {\mathcal IS}.
\end{equation}

Eq. (\ref{map}) induces the partition
${\mathcal S}^{[4]}_{3}=\dot\cup_\gamma\varphi^{-1}(\gamma)$, where $\gamma$ is a skeleton-shape.
Then we have
\begin{equation}\label{E:sk-Hgf}
{\bf S}^{[4]}_{3}(z) =
\sum_{h\geq 2}\sum_{\gamma\in\,{\mathcal IS}(h)}
\mathbf{S}_\gamma(z).
\end{equation}
We proceed by computing the generating function $\mathbf{S}_\gamma(z)$.
We will construct $\mathbf{S}_\gamma(z)$ via simpler combinatorial classes
as building blocks considering the classes $\mathcal{M}$ (stems), $\mathcal{K}^{3}$ (stacks), $\mathcal{N}^{3}$ (induced stacks),
$\mathcal{L}$ (isolated vertices), $\mathcal{A}$ (arcs) and $\mathcal{Z}$
(vertices), where $\mathbf{Z}(z)=z$ and $\mathbf{R}(z)=z^2$.
We inflate $\gamma\in {\mathcal IS}(h)$ having $h$ arcs to a structure in two steps.\\

\noindent{\bf Step I:} we inflate any shape-arc to a stack of size at least
$3$ and subsequently add additional stacks. The latter are called
induced stacks \index{induced stack} and have to be separated by means of
inserting isolated vertices, see Fig.~\ref{F:addskeletonstack}.
\restylefloat{figure}\begin{figure}[h!t!b!p]
\centerline{\includegraphics[width=1\textwidth]{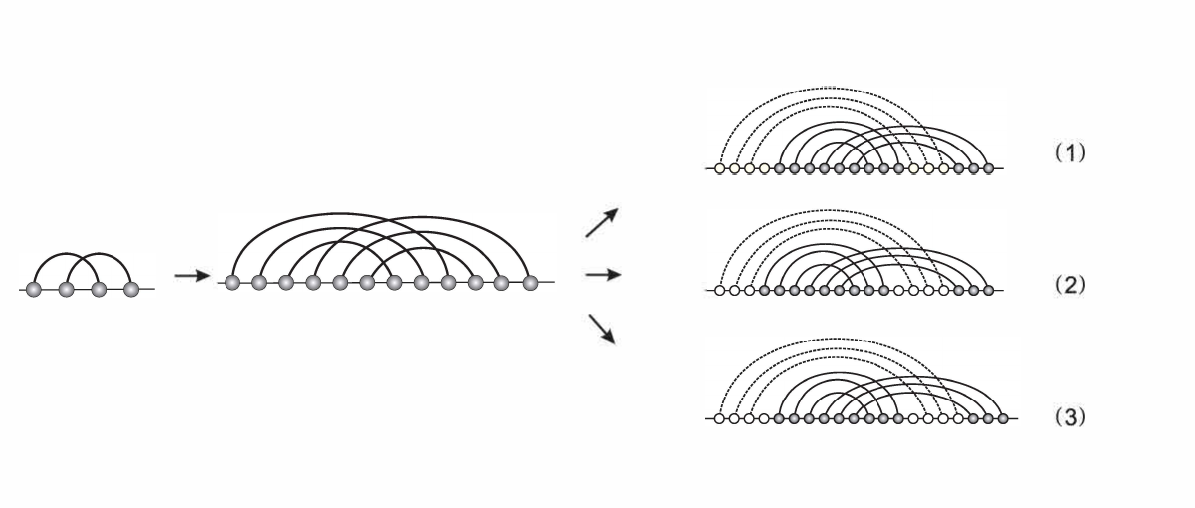}}
\caption{\small Step I:
a skeleton-shape (left) is inflated to a $3$-noncrossing,
canonical diagram. First, every arc in the skeleton-shape is inflated to
a stack of size at least three (middle), and then the skeleton-shape is inflated
to a new $3$-noncrossing, canonical skeleton structure (right) by adding
one stack of size three.
There are three ways to insert the intervals isolated vertices.}
\label{F:addskeletonstack}
\end{figure}
Note that during this first inflation step no intervals of isolated
vertices, other than those necessary for separating the nested stacks
are inserted. We generate
\begin{itemize}
\item isolated segments i.e.~sequences of isolated
vertices $\mathcal{L}= \textsc{Seq}(\mathcal{Z})$ with the generating function
\begin{eqnarray*}
 {\bf L}(z) & = &  \frac{1}{1-z}
\end{eqnarray*}
\item stacks, i.e.~pairs consisting of the minimal sequence of arcs
$\mathcal{R}^\sigma$ and an arbitrary extension consisting of arcs
of arbitrary finite length
\begin{equation*}
\mathcal{K}^{3}=
\mathcal{A}^{3}\times\textsc{Seq}\left(\mathcal{A}\right)
\end{equation*}
with the generating function
\begin{eqnarray*}
\mathbf{K}^3(z) & = & z^6\cdot \frac{1}{1-z^2}
\end{eqnarray*}
\item induced stacks, i.e.~stacks together with at least one nonempty
interval of isolated vertices on either or both its sides.
\begin{equation*}
\mathcal{N}^{3}=\mathcal{K}^{3}\times
\left(\mathcal{Z}\times\mathcal{L}
+\mathcal{Z}\times\mathcal{L}+\left(\mathcal{Z}\times
\mathcal{L}\right)^2\right),
\end{equation*}
with the generating function
\begin{equation*}
\mathbf{N}^3(z)=\frac{z^{6}}{1-z^2}\left(2\frac{z}{1-z}
+\left(\frac{z}{1-z}\right)^2\right)
\end{equation*}
\item stems, that is pairs consisting of stacks $\mathcal{K}^3$
and an arbitrarily long sequence of induced stacks
\begin{equation*}
\mathcal{M}^3=\mathcal{K}^{3}
\times \textsc{Seq}\left(\mathcal{N}^{3}\right),
\end{equation*}
with the generating function
\begin{eqnarray*}
\mathbf{M}^3(z)=\frac{\mathbf{K}^3(z)}{1-\mathbf{N}^3(z)}=
\frac{\frac{z^{6}}{1-z^2}}
{1-\frac{z^{6}}{1-z^2}\left(2\frac{z}{1-z}
+\left(\frac{z}{1-z}\right)^2\right)}.
\end{eqnarray*}
\end{itemize}

\noindent{\bf Step II:} here we insert additional isolated vertices at the
remaining $(2h-1)$ positions, see Fig.~\ref{F:addskeletonvertex}.
\restylefloat{figure}\begin{figure}[h!t!b!p]
\centerline{\includegraphics[width=1\textwidth]{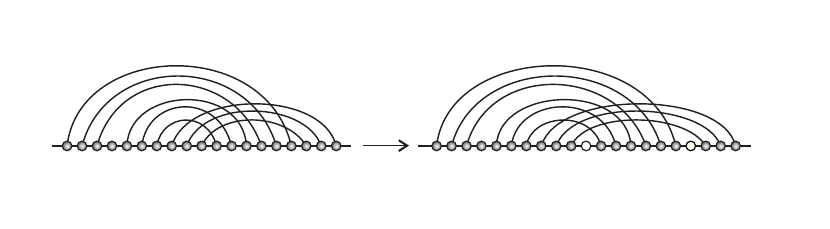}}
\caption{\small Step II: the diagram (left)
obtained in {\sf (1)} in Fig.~\ref{F:addskeletonstack} is inflated to a new
$3$-noncrossing, canonical skeleton diagram (right) by adding isolated
vertices (hollow discs).} \label{F:addskeletonvertex}
\end{figure}
Formally the second inflation is expressed via
\begin{itemize}
\item $\mathcal{J}=\mathcal{L}^{2h-1}$, where
\begin{eqnarray*}
\mathbf{J}(z) &= & \left(\frac{1}{1-z}\right)^{2h-1}
\end{eqnarray*}
\end{itemize}
Combining Step I and Step II we arrive at
\begin{eqnarray}
\mathcal{S}_{\gamma}&=&\left(\mathcal{M}^3\right)^h
\times\mathcal{L}^{2h-1}
\end{eqnarray}
and accordingly
\begin{eqnarray}
{\bf S}_{\gamma}(z)&=&
\left(\frac{\frac{z^{6}}{1-z^2}}
{1-\frac{z^{6}}{1-z^2}\left(2\frac{z}{1-z}
+\left(\frac{z}{1-z}\right)^2\right)}\right)^h
\left(\frac{1}{1-z}\right)^{2h-1}\nonumber\\
&=&(1-z)\left(\frac{z^{6}}{(1-z^2)(1-z)^2-(2z-z^2)
z^{6}}\right)^h\nonumber\\
&=&(1-z)\left(z^6\over{1-2z+2z^3-z^4-2z^7+z^8}\right)^h.
\end{eqnarray}
Since for any $\gamma,\gamma_1\in \mathcal{IS}(h)$ we have
$\mathbf{S}_\gamma(z)=\mathbf{S}_{\gamma_1}(z)$, we derive
\begin{equation*}
{\bf S}^{[4]}_{3}(z) = \sum_{h\geq 2}\sum_{\gamma\in\,{\mathcal
IS}(m)} \mathbf{S}_\gamma(z) =
\sum_{h\geq 2}{\sf IS}(h)\mathbf{S}_\gamma(z).
\end{equation*}
Setting
\begin{equation}\label{E:eta}
\eta_3(z)={z^6\over{1-2z+2z^3-z^4-2z^7+z^8}},
\end{equation}
we have, according to eq.~(\ref{E:sk-Hgf}) and Theorem~\ref{T:gfIk} the
following situation
\begin{eqnarray}
{\bf S}^{[4]}_{3}(z) &=&
\sum_{h\geq 2}{\sf IS}(h)\, {\bf S}_{\gamma}(z)\nonumber\\
&=&(1-z)\sum_{h\geq 2}{\sf IS}(h)\, \left(z^6\over{1-2z+2z^3-z^4-2z^7+z^8}\right)^h\nonumber\\
&=&(1-z)\, {\bf IS}\left(z^6\over{1-2z+2z^3-z^4-2z^7+z^8}\right),
\end{eqnarray}
Claim 2 is proved.

Now we substitute eq. (\ref{E:IStoSk}) into eq. (\ref{SK34toSk}), we arrive at eq. (\ref{s34GF}), where $w_0(z)={z^4 \over {1-z^2+z^6}}$. The theorem follows.

\end{proof}

\section{Asymptotic analysis}\label{S:asymptotic}

Denote $y=z{\bf F}_3^2(z)$, then we have
\begin{equation}
{\bf S}(y)={\bf F}_3(z).\label{skeleta-shape2}
\end{equation}

 We will use singularity analysis to analyze the asymptotic behavior of the coefficients of ${\bf S}(y)$. To apply singularity analysis to ${\bf S}(y)$, we need to know its radius of convergence $R$, the behavior of ${\bf S}(y)$ near $y=R$, the absence of other singularities of ${\bf S}$ on the circle of convergence, $\Delta$-analyticity of ${\bf S}(y)$ and the singular expansion of ${\bf S}(y)$. This information must be deduced from eq. (\ref{skeleta-shape2}) and a knowledge of ${\bf F}_3$. As is typical, we will see that ${\bf F}_3$ has a unique singularity. This information is then transferred to ${\bf S}$ via eq. (\ref{skeleta-shape2}), then singularity analysis applies.

\begin{lemma}[Existence and uniqueness of dominant singularity of ${\bf S}(y)$]\label{L:exs_and_uniq}
The dominant singularity of the generating function ${\bf S}(y)$ is $R=\rho{\bf F}_3^2(\rho)\approx 0.0729$, $\rho=1/16$, and ${\bf S}(y)$ has no singularity other from the circle of $\{y\mid |y|=R\}$, that is, ${\bf S}(y)$ is analytically continuable around any $y\neq \rho$ with $|y|\leq R$.
\end{lemma}

\begin{proof}

{\it Claim 1. The radius of convergence, $R$, is equal to $\rho{\bf F}^2_3(\rho)$, where $\rho$ is the radius of convergence of ${\bf F}_3$.}\\

Since series ${\bf S}(y)$ has nonnegative coefficients, we know that ${\bf S}(y)$ has a radius of convergence of $R$. since ${\bf F}_3$ has nonnegative coefficients, $\dif y/\dif z>0$ for $0<z<\rho$ and so $z$ is an analytic function of $y$ for $0<y<R$. We now show that $y=\rho{\bf F}_3^2(\rho)$ is a singularity of ${\bf S}$, from which $R=\rho{\bf F}_3^2(\rho)$ follows.

From Proposition \ref{P:fk}, we know that ${\bf F}_3(\rho)$, ${\bf F}_3'(\rho)$, ${\bf F}_3''(\rho)$, ${\bf F}_3'''(\rho)$ exist and are positive but ${\bf F}_3^{(4)}(z)\rightarrow \infty$ as $z\rightarrow \rho$. From eq. (\ref{skeleta-shape2}),
\begin{equation}\label{deriv}
{\bf S}^{(4)}(z{\bf F}_3^2(z))=\frac{U({\bf F}_3, {\bf F}_3',{\bf F}_3'',{\bf F}_3'''){\bf F}_3^{(4)}(z)+V({\bf F}_3, {\bf F}_3',{\bf F}_3'',{\bf F}_3''')}{[{\bf F}_3(z)({\bf F}_3(z)+2z {\bf F}_3'(z))]^{7}},
\end{equation}
where $U$ and $V$ are both multi-variable polynomials and $U|_{z=\rho}>0$.

Thus ${\bf S}^{(4)}(y)\rightarrow \infty$ as $y\rightarrow (\rho{\bf F}^2_3(\rho))^-$ and so $R=\rho{\bf F}_3^2(\rho)$.\\

{\it Claim 2. There are no other singularities of ${\bf S}(y)$ existed on the circle of convergence $|y|=R$.}\\

By eq. (\ref{skeleta-shape2}), singularities of ${\bf S}(y)$ are due to singularities of ${\bf F}_3$ and values of $y$ near which inverting $y=z {\bf F}_3^2(z)$ uniquely is impossible.

Let us consider singularities of ${\bf F}_3(z)$. We have known that ${\bf F}_3(z)$ is $\Delta$-analytic and its dominant singularity is $\rho$ and unique. Then let us consider inverting $y=z {\bf F}_3^2(z)$ on $|y|=R$ with $y\neq R$. Thus $z\neq \rho$ and so ${\bf F}_3$ is defined. By the implicit function theorem any singularities must be due to $\dif y/\dif z=0$. We now use ${\bf S}'(R)\neq \infty$. Since ${\bf S}(y)$ has nonnegative coefficients, no singularity on $|y|=R$ can dominate $y=R$ and so ${\bf S}'(y)\neq \infty$ on $|y|=R$. By eq. (\ref{skeleta-shape2}),
\begin{equation*}
{\bf S}'(y)=\frac{\dif {\bf S}}{\dif y}=\frac{\dif {\bf F}_3}{\dif z}\bigg/\frac{\dif y}{\dif z}.
\end{equation*}

At a singularity, $\dif y/\dif z=0$ and so ${\bf F}_3'=0$. Since $\dif y/\dif z={\bf F}_3^2+2z{\bf F}_3{\bf F}_3'$, ${\bf F}_3=0$. If $y=\sigma\neq R$ is a singularity with $|\sigma|=R$ and $r$ is the corresponding value of $z$, we have shown that ${\bf S}(y)={\bf F}_3(z)=(z-r)^k g(z)$ where $k\geq 1$, $g(r)\neq 0$ and $g$ is analytic near $r$. Such $r$ is existed, which means $r$ is in the analytic domain of ${\bf F}_3(z)$, $ \mathbb{C}\backslash [1/16, \infty)$, due to eq. (\ref{E:F3-gf}). The proof is following:

From eq. (\ref{skeleta-shape2}), we derive the following equation,
\begin{equation*}
z {\bf F}^2_3(z)=z {\bf S}^2(z {\bf F}_3^2(z)),
\end{equation*}
then we have
\begin{equation}\label{inverseofy}
z=\frac{y}{ {\bf S}^2(y)}
\end{equation}
and
\begin{equation}\label{skeleta-shape3}
{\bf S}(y)={\bf F}_3(y {\bf S}^{-2}(y)).
\end{equation}
Therefore, from eq. (\ref{inverseofy}), we know that for any given $\sigma\neq R$ on the circle of $|y|=R$, we have corresponding value of $z$, $r$, which is $\frac{\sigma}{ {\bf S}^2(\sigma)}$. ${\bf S}^2(\sigma)$ cannot be zero, or else, $|z|$ will be infinity, then by eq. (\ref{inverseofy}) and eq. (\ref{skeleta-shape3}) we have $\lim_{|z|\rightarrow \infty} y(z)=\lim_{y\rightarrow \sigma}\frac{y}{ {\bf S}^{2}(y)}{\bf F}_3^2\left(\frac{y}{ {\bf S}^{2}(y)}\right)=\sigma$. But $\lim_{|z|\rightarrow \infty} |y(z)|> 1.83> R$, it is contradicted to $|\sigma|=R$. Hence $r$ cannot be infinity. Moreover, if $r$ lies on the line of $[1/16,\infty)$, we have $|r{\bf F}^2_3(r)|=|\sigma|=R$. However, by the rigorous calculation of \textsf{Mathematica}, when $r\in(1/16,1/8)$,
 \begin{equation*}
 r{\bf F}^2_3(r)=(1/16){\bf F}^2_3(1/16)+\Xi(r)(r-1/16),
 \end{equation*}
 where $\Xi(r)$ satisfying $|\Xi(r)|>0$ and $|\arg\Xi(r)|<0.19$. Besides, $\min|r{\bf F}^2_3(r)|>0.12>(1/16){\bf F}^2_3(1/16)$ when $r>1/8-1/32$. The two calculation results immediately imply that $\min_{r \geq 1/16}|r{\bf F}^2_3(r)|=(1/16){\bf F}^2_3(1/16)=R$ and the minimum arrives if and only if $r=1/16$. Such case is contradicted to $|\sigma|=R$ and $\sigma\neq R$. Now we have proved the existence of $r$.

Let us continue the main proof. Since
\begin{equation*}
{\bf S}'(y)=\frac{\dif {\bf F}_3/\dif z}{\dif (z {\bf F}_3^2)/\dif z}=\frac{(z-r)^{k-1}}{(z-r)^{2k-1}}f(z),
\end{equation*}
where $f(r)\neq 0$. This gives ${\bf S}'(\sigma)=\infty$, a contradiction. Hence no such singularity $\sigma$ exists, Claim 2 follows.

\end{proof}

\begin{lemma}[$\Delta$-analytic continuation of ${\bf S}(y)$]
${\bf S}(y)$ can be analytically continued to a $\Delta$-domain at the
singularity of $R$.
\end{lemma}
\begin{proof}
{\it Claim 1. ${\bf S}(y)$ is analytic in a ``$\Delta$-neighborhood" for $R$, of the form
\begin{equation*}
\{|y-R|<\delta', \arg(y-R)\in (\phi,2\pi-\phi)\},
\end{equation*}
with $\delta'>0$ and $0<\phi<\pi/2$.}

 First, by Fusy's method \cite{Fusy:10}, the singular expansion of ${\bf S}(y)$ at $R$ holds in a neighborhood $\Omega$ of $R$, where $\Omega$ is the image of $y(z)$ mapping from a slit neighborhood $\Theta$, which is defined by $\{|z-\rho|<\delta\}\backslash\{|\arg(z-\rho)|\leq\phi'\}$, with $\delta>0$ and $0<\phi'<\pi/2$. Since $y'(z)$ converges to a positive constant at $z=\rho$, it is locally conformal at $\rho$. Hence $\Omega$ contains a ``$\Delta$-neighborhood" for $R$, of the form
\begin{equation}\label{E:delta-neigh}
N(\Delta)\equiv\{|y-R|<\delta',\text{arg}(y-R)\in(\phi,2\pi-\phi)\},
\end{equation}
with $\delta'>0$ and $0<\phi<\pi/2$.

Second, The claim that $y(z)$ is univalent (bijective and analytic) in $N(\Delta)$, is true. Since $y'(\rho)>0$, by Noshiro-Warschawski Theorem \cite{goodman:83}, we know that $y(z)$ is respectively univalent in $\Theta\cap\{\im z>0\}$ and $\Theta\cap\{\im z<0\}$ when $z$ is in some neighborhood of $\rho$. Due to ${\bf F}_3(z)$, $y$ has a singular expansion,
\begin{equation*}
y(z)=R+y'(\rho)(z-\rho)+o(z-\rho)
\end{equation*}
around $\rho$. Assuming $\delta$ in eq.~(\ref{E:delta-neigh}) has been chosen sufficiently small enough making $|z|<\rho$ when $|\arg (z-\rho)|>\frac{2\pi}{3}$, and
\begin{equation}\label{E:upper1}
\arg(y(z)-R)\in (0,\pi),
\end{equation}
when $\arg (z-\rho)\in(\phi',\frac{2\pi}{3}]$ , concluded by the same method for eq.(\ref{E:gamma-ineq}).

For the point $z\in\Theta$ of $\arg (z-\rho)\in(\frac{2\pi}{3}, \pi)$. We draw a straight line $\{s(t)\mid s(0)=z, s(1)=z_0,0\leq t\leq 1\}$, where $z_0$ satisfies $z_0\in\Theta$ and $\arg (z_0-\rho)=\frac{2\pi}{3}$. Clearly $y(z_0)$ is in the upper-half plain. Suppose $z$ maps to the lower-half plain. The smooth curve of $y(s(t))$ must intersect the line of $[0,R]$, which means there exists $t_0$ such that $y(s(t_0))=r_0$, $0\leq r_0\leq R$. Using eq. (\ref{inverseofy}) we solute that $s(t_0)=r_0/{\bf S}^2(r_0)$, which implies $s(t_0)$ is real. It is impossible since the line of $\{s(t)\}$ cannot touch the real axis due to the convexity of $\Theta\cap\{\im z>0\}$. Therefore, \begin{equation}\label{E:upper2}
y(z)\in \{z\mid \im z>0\}, \text{when $z\in \Theta$ and $\arg (z-\rho)\in(\frac{2\pi}{3}, \pi)$}.
\end{equation}

Eq. (\ref{E:upper1}) and eq. (\ref{E:upper2}) will lead us that $y(z)$ is mapping $\Theta\cap\{\im z>0\}$ to a upper-half plain. By the symmetric theorem for analytic function \cite{fang:96}, $y(z)$ is mapping $\Theta\cap\{\im z<0\}$ to a lower-half plain as well. Obviously, $y(z)$ is strictly monotonic increasing on the real segment of $[0,\rho]$ and $y([0,\rho])=[0,R]$, hence $y(z)$ is globally univalent in the whole $\Theta$.

According to eq. (\ref{skeleta-shape2}), the analyticity of inversion function of $y(z)$   and ${\bf F}_3$ respectively in ``$\Delta$-neighborhood" $N(\Delta)$ and $\Theta$ imply the truth of Claim 1.

\restylefloat{figure}\begin{figure}[h!t!b!p]
\begin{center}
\scalebox{0.8}[0.8]{\includegraphics{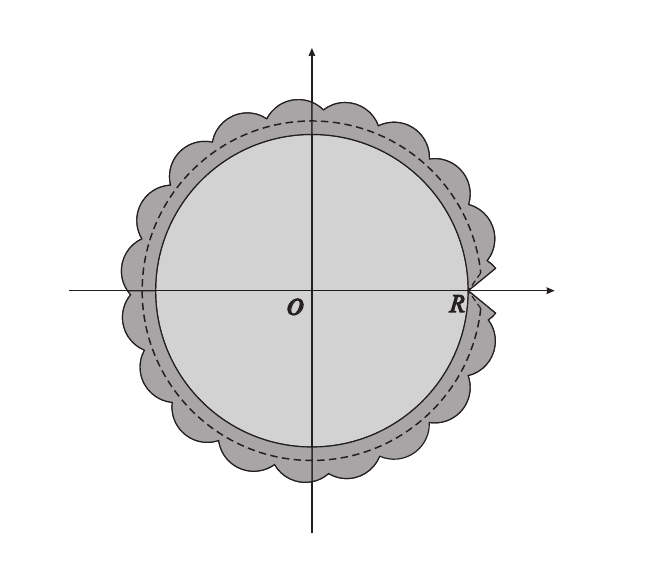}}
\end{center}
\caption{The construction of $\Delta$-domain for ${\bf S}(y)$: the dark grey part is the analytic continuous domain based on the disc of $U(0,R)$ (light grey). The dash line is the edge of $\Delta$-domain which is contained in the whole analytic domain (dark and light grey) of ${\bf S}(y)$.}\label{F:delta-analytic}
\end{figure}

{\it Claim 2. ${\bf S}(y)$ can be analytically continued to a $\Delta$-domain at the singularity of $R$.}

 The previous claim actually tells us ${\bf S}(y)$ can be analytically continued to a  ``$\Delta$-neighborhood" domain $N(\Delta)$, whose boundary makes an acute angle with the half line $\mathbb{R}_{\geq R}$.

  Claim 2 of Lemma \ref{L:exs_and_uniq} reveals that for any $y'$ on the circle of $|y|=R$ outside $N(\Delta)$, ${\bf S}$ can be analytically continued to $U(y',\delta_{y'})$ which contains in an image of $y(z)$ univalently mapping from a neighborhood $U(r,\varepsilon)$ where $y(r)=y'$. For the reason that $\{y\mid |y|=R\}$ without containing in $N(\Delta)$ is a closed curve $\ell$, we can find a finite neighborhoods, $\{U(y_i, \delta_{y_i})\}_{1\leq i\leq n}\subset \{U(y',\delta_{y'})\}_{|y'|=R}$, covering $\ell$.  Since
  \begin{equation*}
  U(y_i, \delta_{y_i})\cap U(y_{i+1}, \delta_{y_{i+1}})\cap U(0,R)\neq \emptyset,\quad 1\leq i\leq n-1,
  \end{equation*}
  we have that $\cup^n_{i=1} U(y_i, \delta_{y_i})$ is an analytic continuation domain for ${\bf S}(y)$.

  Combining these two facts and
   \begin{equation*}
   \bigcup^n_{i=1} U(y_i, \delta_{y_i})\cap U(0,R)\cap N(\Delta)\neq \emptyset,
   \end{equation*}
   we obtain a new analytic continuation domain for ${\bf S}(y)$, see Fig.~\ref{F:delta-analytic}.  Using the method of Step {\sf (c)} in Proposition~\ref{chapter2:algeasym}, we are able to find a $\Delta$-domain inside, which means we have proved ${\bf S}(y)$ is $\Delta$-analytic. Here the whole proof is complete.

\end{proof}

\begin{lemma}[Singular expansion of ${\bf S}(y)$]\label{L:sk-sing-expansion}
The singular expansion of ${\bf S}(y)$ for
$y\rightarrow R$ is given by
\begin{equation}
{\bf S}(y)=
\sum_{i=0}^{4} b_i(y-R)^i+c'(y-R)^{4}
\log (y-R)+o((y-R)^4)
\end{equation}
were $c'$ is
some negative constant, and $b_i=\frac1{i!}{\bf S}^{(i)}(R^-)$, $0\leq i\leq 3$.
\end{lemma}

\begin{proof}

{\it Claim 1.
The singular expansion of ${\bf S}(y)$ for
$y\rightarrow R$ is given by
\begin{equation}\label{sk-sing-expansion}
{\bf S}(y)=
\sum_{i=0}^{3} b_i(y-R)^i+(y-R)^{4}
\log (y-R)\left(c'+o(1)\right).
\end{equation}
 Furthermore, $c'$ is
some constant, and $b_i=\frac1{i!}{\bf S}^{(i)}(R^-)$, $0\leq i\leq 3$.
}\\

Since ${\bf F}_3^{(m)}(\rho^-)$, $0\leq m\leq 3$ exists, $y^{(m)}(\rho^-)$ exists. Therefore ${\bf S}^{(m)}(R^-)$ exists. Then we set
\begin{equation}
E(y)={\bf S}(y)-\sum_{i=0}^{3} b_i(y-R)^i,\label{inv-sing-expansion}
\end{equation}
where $b_i=\frac1{i!}{\bf S}^{(i)}(R^-)$, $0\leq i\leq 3$.

Furthermore, by eq. (\ref{E:Fk_sing_ex}), we have the singular expansion of $y(z)$ which is
\begin{equation*}
y(z)=
\sum_{i=0}^{3} a_i(z-\rho)^i+c_3''(z-\rho)^{4}
\log (z-\rho)\left(1+o(1)\right)
\end{equation*}
where $c_3''$ is
some negative constant, and $a_i=\frac1{i!}y^{(i)}(\rho^-)$, $0\leq i\leq 3$.

Now we begin to discuss the higher term of the expansion of ${\bf S}(y)$.

By eq. (\ref{inv-sing-expansion}), we know that

\begin{align}
{\bf F}_3(z)={\bf S}(y(z))&=\sum_{i=0}^{3} b_i(y(z)-y(\rho))^i+E(y(z))\nonumber\\
&=\sum_{i=0}^{3} b_i\left[\sum_{j=1}^{3} a_j(z-\rho)^j+c_3'(z-\rho)^4\log(z-\rho)(1+o(1))\right]^i\nonumber\\
&+E(y(z))\label{inv-cal2}.
\end{align}

From Proposition \ref{P:fk} and Fa\`{a} di Bruno's Formula \cite{Stanley:001}, for $1\leq m\leq 3$ we have
\begin{align}
\frac1{m!}{\bf F}_3^{(m)}(\rho^-)&=\sum_{i=0}^m \sum_{t_1+t_2+\cdots+t_m=i \atop t_1+2t_2+\cdots+mt_m=m}
\frac{1}{t_1!t_2!\cdots t_m!}{\bf S}^{(i)}(R^-)\prod_{j=1}^m \left(\frac{y^{(j)}(\rho^-)}{j!}\right)^{t_j}\nonumber\\
&=\sum_{i=0}^m \frac1 {i!}{\bf S}^{(i)}(R^-)[(z-\rho)^m]\left(\sum_{j=1}^m\frac{y^{(j)}(\rho^-)}{j!}(z-\rho)^j\right)^i\nonumber\\
&=\sum_{i=0}^m b_i[(z-\rho)^m]\left(\sum_{j=1}^m a_j(z-\rho)^j\right)^i.\label{faadibruno}
\end{align}

So by
Eq. (\ref{faadibruno}) and ${\bf F}_3(\rho^-)={\bf S}(R^-)=b_0$, it leads that the every coefficient of $(z-\rho)^m$, $0\leq m\leq 3$ in the expansion of each side in eq. (\ref{inv-cal2}) is eliminated. Therefore we have

\begin{equation}
(z-\rho)^4\log(z-\rho)(c''+o(1))=E(y(z)),\label{inv-eq}
\end{equation}
where $c''$ is some constant.

We divide $E(y)$ by $(y-R)^4 \log(y-R)$ and then
\begin{align}
\lim_{y\rightarrow R}\frac{E(y)}{(y-R)^4 \log(y-R)}&=
\lim_{z\rightarrow \rho}\frac{E(y(z))}{(y(z)-y(\rho))^4\log(y(z)-y(\rho))}\nonumber\\
&=\lim_{z\rightarrow \rho}\frac{(z-\rho)^4\log(z-\rho)(c''+o(1))}{(y(z)-y(\rho))^4 \log(y(z)-y(\rho))}\nonumber\\
&=\text{a constant},\label{E:errorterms-deduc}
\end{align}
since $\frac{z-\rho}{y(z)-y(\rho)}$ and $\frac{\log(z-\rho)}{\log(y(z)-y(\rho))}$ both exist as $z\rightarrow \rho$. Therefore
\begin{equation}
{\bf S}(y)=\sum_{i=0}^{3} b_i(y-R)^i+(y-R)^4\log(y-R)(c'+o(1)),\label{inv-sing-expansion2}
\end{equation}
whence Claim 3.\\

{\it Claim 2. $c'$ is negative.}\\

From eq. (\ref{deriv}) with $T={\bf F}_3(\rho^-)+2\rho{\bf F}'_3(\rho^-)$ and $F={\bf F}_3(\rho^-)$,
\begin{equation}\label{comb1}
{\bf S}^{(4)}(y)\sim K\cdot {\bf F}_3^{(4)}(z)/(FT)^{7}
\end{equation}
as $y\rightarrow R$, where $K$ is a positive constant equaled $U({\bf F}_3,{\bf F}'_3,{\bf F}''_3,{\bf F}'''_3)|_{z=\rho}=F^4T^2$. From $y=z{\bf F}_3^2(z)$ we have
\begin{equation}
z-\rho\sim (y-R)/(FT).
\end{equation}
From eq. (\ref{E:Fk_sing_ex}),

\begin{equation}\label{comb2}
{\bf F}^{(4)}_3(z)\sim  c'_3 \frac{\dif^4 [(z-\rho)^4\log(z-\rho)]}{\dif z^4}=c'_3 (4!\log(z-\rho)+M),
\end{equation}
where $M$ is a positive integer.

Combining eq. (\ref{comb1})--(\ref{comb2}) we obtain
\begin{equation}\label{comb3}
{\bf S}^{(4)}(y)\sim \frac{K\cdot c'_3 }{(FT)^{7}}(4!(\log(y-R)-\log(FT))+M)\sim \frac{K\cdot c'_3}{(FT)^{7}}\frac{\dif^4 [(y-R)^4\log(y-R)]}{\dif y^4}.
\end{equation}
Differentiating eq. (\ref{sk-sing-expansion}) four times and comparing to eq.~(\ref{comb3}), we know that
\begin{equation}\label{E:multiplier}
c'=\frac{K\cdot c'_3}{(FT)^{7}},
\end{equation}
which means $c'$ is negative, whence Claim 2.
Based on the above claims, we do the final proof, we set
\begin{equation*}
E'(y)={\bf S}(y)-\sum_{i=0}^{3} b_i(y-R)^i-c'(y-R)^{4}
\log (y-R),
\end{equation*}
where
Substituting $y=z{\bf F}^2_3(z)$ into the equation above, we have
\begin{align}
E'(y(z))&={\bf F}_3(z)-\sum_{m=0}^{3} b_m(y(z)-R)^m-c'(y(z)-R)^{4}
\log (y(z)-R)\nonumber\\
&=\sum_{m=0}^3 \left[\frac1{m!}{\bf F}_3^{(m)}(\rho^-)-\sum_{i=0}^m b_i[(z-\rho)^m]\left(\sum_{j=1}^m a_j(z-\rho)^j\right)^i\right] (z-\rho)^m\nonumber\\
&+\left[c'_3-c'(FT)^4-2 b_1\rho F c'_3\right]\log(z-\rho)\nonumber\\
&+\theta((z-\rho)^4), \label{E:errorterms}
\end{align}
where $g(z)=\theta(f(z))$ means $\lim_{z\rightarrow\rho}|g(z)/f(z)|$ is a positive constant. In the view of the equation of eq.~(\ref{faadibruno}), eq.~(\ref{E:multiplier}) and $b_1=\frac{{\bf F}_3'(\rho)}{FT}$, we imply that the former two terms of eq.~(\ref{E:errorterms}) are actually zero, which means
\begin{equation*}
E'(y(z))=\theta((z-\rho)^4).
\end{equation*}
Using the same deducting method of eq.~(\ref{E:errorterms-deduc}), we have
\begin{equation*}
E'(y)=\theta((y-R)^4)
\end{equation*}
The proof is complete.
\end{proof}

    Currently, we have the $\Delta$-analytic property and the singular expansion of ${\bf S}(y)$. The finale is applying the transfer rule (Corollary \ref{chapter2:sim-transfer}) which yields directly the following asymptotic estimate for the number of $3$-noncrossing skeleton matching.

\begin{theorem}[Asymptotic number of skeleton matching]
The asymptotic behavior of the number of $3$-noncrossing skeleton matching with $h$ arcs is given by
\begin{equation*}
{\sf S}(h)\ \sim\ C\cdot h^{-5}R^{-h}, \end{equation*}
where $C\approx 3.03096$, see Fig.~\ref{F:skeleta-asym-check}.
\end{theorem}

The multiplier $C$ is not estimated but precisely computed as follow with the help of eq.~(\ref{deriv}) and eq.~(\ref{E:multiplier}),
\begin{equation*}
C=-\frac{K c'_3}{(FT)^7}\cdot R^4\cdot4!=\frac{24(-512+165\pi)^5}{3125(256-81\pi)^5\pi}.
\end{equation*}

\begin{theorem}\label{T:canonical-sk-asym}
The number of canonical $3$-noncrossing skeleton diagrams of length $n$ with arc length $\geq 4$, ${\sf S}^{[4]}_{3}(n)$, has asymptotic of
\begin{equation}
{\sf S}^{[4]}_{3}(n)\ \sim\ C' n^{-5} \eta^{-n},
\end{equation}
where $C'\approx 7892.16$ and $\eta\approx 0.4934$ is the minimal positive real solution of $\vartheta(z)=R$, where
\begin{equation*}
\vartheta(z)=\left({\sqrt{w_0(z)}z}\over{1-z}\right)^2
\end{equation*}
and $w_0(z)={z^4 \over {1-z^2+z^6}}$. See Fig.~\ref{F:skeleta-asym-check}.
\end{theorem}

\begin{proof}
Pringsheim's
Theorem \cite{Tichmarsh:39} guarantees that ${\bf S}^{[4]}_{3}(z)$ has a
dominant real positive singularity $\eta$. We verify that $\eta$, which is the unique solution of
minimum modulus of the equation $\vartheta(z)=R$, is the
unique dominant singularity of ${\bf S}^{[4]}_{3}(z)$,  and
$\vartheta'(\eta)> 0$.

Here we discuss a claim that $|\vartheta(z)|$ is constraint in the closed disc with radius of $R$ when $\{|z|\leq \eta\}$. By maximum modulus theorem, we only need to focus on $\{|z|=\eta\}$. Since
\begin{equation}\label{E:theta-norm}
|\vartheta(z)|=\frac{\eta^6}{|1-z|^2|1-z^2+z^6|}.
\end{equation}
Studying each factor of the denominator of eq. (\ref{E:theta-norm}) by replacing $z$ to $\eta(\cos x+\i\sin x)$, we have the following functions corresponding to $x$,
\begin{eqnarray*}
|1-z|^2&=&1-2\eta\cos x+\eta^2,\\
|1-z^2+z^6|^2&=&1 + \eta^4 + \eta^{12} - 2 \eta^2 \cos 2 x - 2 \eta^8 \cos 4 x + 2 \eta^6\cos 6 x.
\end{eqnarray*}
Finding critical points for every function above and we can easily conclude that both functions reach the minimum when $x=0$. Therefore the maximum of $|\vartheta(z)|$ for $|z|=\eta$ is $|\vartheta(\eta)|$, which is also equal to $R$, whence the claim.

Furthermore,
\begin{equation*}
{\bf G}(x)={\bf S}(x)-1-x,
\end{equation*}
which means ${\bf G}(x)$ has the same expansion formation as ${\bf S}(x)$. According to
Proposition~\ref{chapter2:algeasym} we therefore have
\begin{equation*}
{\sf S}^{[4]}_{3}(n)\sim  C'\, n^{-5}\,
(\eta^{-1})^n.
\end{equation*}
Here $C'$ can be precisely computed as follow,
\begin{equation*}
C'=(1-\eta)\cdot\vartheta'(\eta)^4\cdot\frac{820125 \pi^7}{256 (256 - 81 \pi)^5 (-512 + 165 \pi)^3}\cdot 4!\cdot\eta^4,
\end{equation*}
which is concluded by eq. (\ref{s34GF}), eq. (\ref{sk-sing-expansion}) and Corollary~\ref{chapter2:sim-transfer}. The proof of Theorem~\ref{T:canonical-sk-asym} is complete.
\end{proof}

\restylefloat{figure}\begin{figure}
\centering
\includegraphics[width=1\textwidth]{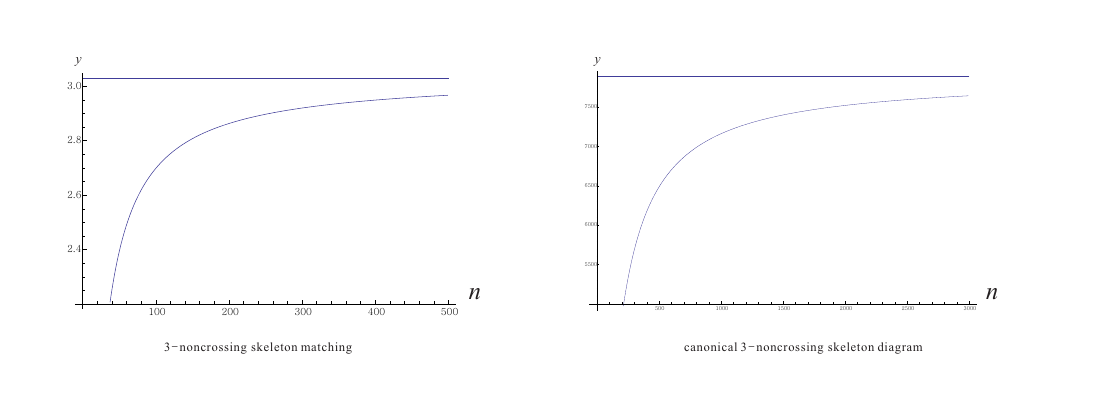}
\caption{The computational experiment of asymptotic enumeration: The straight lines in both figures are $y=C$ (left) and $y=C'$ (right) respectively. The curves are plotted by two lists of points, ${\sf S}(n)/(n^{-5}R^{-n}), n=1,\ldots, 500$ (left) and ${\sf S}^{[4]}_3(n)/(n^{-5}\eta^{-n}), n=1,\ldots, 3000$ (right). We notice that each curve is well approaching the corresponding straight line, which means the asymptotic enumeration makes sense.\label{F:skeleta-asym-check}}
\end{figure}

\section{Discussion}

So far, we have concluded the asymptotic properties of the number of $3$-noncrossing skeleton matching and canonical diagrams with length $n$. The two results solves the complexity problem of Huang's folding algorithm \cite{Huang:09}, which means that we have had a strict mathematical way to inform how much time we cost in Huang's algorithm. In this section we will make a exploration to the skeleton diagrams. First we are going to take a look at the statistics of arcs in canonical $3$-noncrossing skeleton diagrams, then we prospect the further research to the asymptotic analysis for $k$-noncrossing skeleton diagrams.

\subsection{Statistics of arcs in canonical $3$-noncrossing skeleton diagrams}\label{S:statistics}

In this section, we discuss the distribution of arcs in canonical $k$-noncrossing skeleta with certain number of vertices. Let $\mathbb{A}_{n}(S)$ denote the number of arcs in a $3$-noncrossing skeleton diagrams, $S$ and let $\mathcal{A}(n,h)$ and $\textsf{A}(n,h)$ denote the set and number of canonical $3$-noncrossing skeleton diagrams, having exactly $h$ arcs. Now we will study $\mathbb{A}_n$, where $\mathbb{P}(\mathbb{A}_n=h)=\frac{\textsf{A}(n,h)}{\textsf{S}_3^{[4]}(n)}$.

To compute the distribution of $3$-noncrossing skeleta diagrams, i.e. $\mathbb{A}_n$, the first step is to compute the bivariate generating function
\begin{equation*}
{\bf A}(z,u)=\sum_{n \geq 0} \sum_{0\leq h \leq
\frac{n}{2}} {\sf A}(n,h)\,u^h z^n.
\end{equation*}

\begin{theorem}\label{T:bivarate-g-arcs}
Let $u$ and $z$ be indeterminants. Then we have the identity of formal power series
\begin{equation}\label{E:bivarate-g-arcs}
{\bf A}(z,u) =(1-z){\bf G}\left(\frac{w(u,z)z^6}{(1-z)^2}\right),
\end{equation}
where $w(u,z)$ is given by
\begin{equation*}
w(u,z)=\frac{u^3}{1-u z^2+u^3 z^6}
\end{equation*}
and
\begin{equation*}
{\bf G}(x)={\bf S}(x)-1-x.
\end{equation*}
Considered as a relation between analytic functions, eq. (\ref{E:bivarate-g-arcs}) holds for $u=\exp^s$ and $|s|\leq \varepsilon$ for $\varepsilon$ sufficiently small and $|z|\leq \frac{1}{2}$.
\end{theorem}

\begin{proof}
Let $\mathcal{S}_3^{[4]}(\gamma,h,n)$ denote the set of canonical $3$-noncrossing structures, having length $n$ and $h$ arcs, contained in the preimage of a fixed skeleton shape, $\gamma\in\mathcal{IS}(m)$. Then $\mathcal{S}_3^{[4]}=\dot{\cup}\psi^{-1}(\gamma)$ and $\psi^{-1}(\gamma)=\dot{\cup}_n\mathcal{S}_3^{[4]}(\gamma,h,n)$ where $\psi:\mathcal{S}_3^{[4]}\rightarrow \mathcal{IS}$ is the surjective projection into skeleton shape. Then
\begin{eqnarray*}
{\bf A}(z,u) & = &
\sum_{m\geq 0}\sum_{\gamma\in\,{\mathcal IS}(m)}\underbrace{
\sum_{n,h}\vert \mathcal{S}_{3}^{[4]}(\gamma,h,n)\vert z^nu^h}_{
\mathbf{A}_\gamma(z,u)},
\end{eqnarray*}
where $\mathbf{A}_\gamma(z,u)$ is the bivariate generating function of
$3$-noncrossing canonical skeleton diagrams having
the shape $\gamma$.
A structure inflated from $\gamma$ has $s$ stems and $(2s-1)$ intervals
of isolated vertices.
We build these structures in a modular way via the combinatorial classes
$\mathcal{M}$ (stems), $\mathcal{K}^3$ (stacks),
$\mathcal{N}^{3}$ (induced stacks), $\mathcal{L}$ (isolated vertices),
$\mathcal{R}$ (labelled arcs) and $\mathcal{Z}$ (vertices), where
$\mathbf{Z}(z)=z$ and $\mathbf{R}(z,u)=uz^2$.
We proceed in complete analogy to the proof of Theorem~\ref{T:skeleta-gen}, in fact
all we have to is to substitute $\mathbf{R}(z,u)=uz^2$, i.e.~the
bivariate generating function of labelled arcs for
$\mathbf{R}(z)=z^2$. Accordingly we generate

\begin{itemize}
\item isolated segments i.e.~sequences of isolated vertices
$\mathcal{L}= \textsc{Seq}(\mathcal{Z})$, with the generating function
\begin{eqnarray*}
 {\bf L}(z) & = &  \frac{1}{1-z}
\end{eqnarray*}
\item stacks , i.e.~pairs consisting of the minimal sequence of arcs
$\mathcal{R}^3$ and an arbitrary extension consisting of arcs
of arbitrary finite length
$\mathcal{K}^{3}=\mathcal{R}^{3}\times\textsc{Seq}
\left(\mathcal{R}\right)$, with the generating function
\begin{eqnarray*}
\mathbf{K}^\tau(z,u) & = & \frac{(uz^2)^{3}}{1-uz^2}
\end{eqnarray*}
\item induced stacks, i.e.~stacks together with at least one nonempty
interval of isolated vertices on either or both its sides
\begin{equation*}
\mathcal{N}^{3}=\mathcal{K}^{3}\times
\left(\mathcal{Z}\times\mathcal{L}
+\mathcal{Z}\times\mathcal{L}+\left(\mathcal{Z}\times
\mathcal{L}\right)^2\right),
\end{equation*}
with the generating function
\begin{equation*}
\mathbf{N}^3(z,u)=\frac{(uz^2)^{3}}{1-uz^2}\left(2\frac{z}{1-z}
+\left(\frac{z}{1-z}\right)^2\right)
\end{equation*}
\item stems that is pairs consisting of stacks $\mathcal{K}^\tau$
and an arbitrarily long sequence of induced stacks
\begin{equation*}
\mathcal{M}^3=\mathcal{K}^{3}
\times \textsc{Seq}\left(\mathcal{N}^{3}\right),
\end{equation*}
with the generating function
\begin{eqnarray*}
\mathbf{M}^3(z,u)=\frac{\mathbf{K}^3(z,u)}
{1-\mathbf{N}^3(z,u)}=
\frac{\frac{(uz^2)^{3}}{1-uz^2}}
{1-\frac{(uz^2)^{3}}{1-uz^2}\left(2\frac{z}{1-z}
+\left(\frac{z}{1-z}\right)^2\right)}.
\end{eqnarray*}
\end{itemize}
Plainly, the second inflation is identical to that of Theorem~4.4.
Combining Step I and Step II we derive
\begin{eqnarray*}
\mathcal{A}_{\gamma}&=&\left(\mathcal{M}^3\right)^s
\times\mathcal{L}^{2s-1}
\end{eqnarray*}
and compute

\begin{eqnarray*}
{\bf A}_{\gamma}(z,u)&=
\left(\frac{\frac{(uz^2)^3}{1-uz^2}}
{1-\frac{(uz^2)^3}{1-uz^2}\left(2\frac{z}{1-z}
+\left(\frac{z}{1-z}\right)^2\right)}\right)^h
\left(\frac{1}{1-z}\right)^{2h-1}\nonumber\\
&=
(1-z)\left(\frac{u^3 z^6}{1-2z+(1-u)z^2+2uz^3-uz^4-2u^3z^7+u^3 z^8}\right)^h.
\end{eqnarray*}
Since for any $\gamma,\gamma_1\in\mathcal{IS}(h)$,
${\bf A}_{\gamma}(z,u)={\bf A}_{\gamma_1}(z,u)$ holds we obtain
\begin{eqnarray*}
{\bf A}(z,u)=\sum_{h\geq 0}\sum_{\gamma\in\,\mathcal{IS}(h)}
\mathbf{A}_\gamma(z,u) =
\sum_{h= 0}{\sf IS}(h)\,\mathbf{A}_\gamma(z,u).
\end{eqnarray*}
We set
\begin{equation*}
\eta(z,u)=\frac{u^3 z^6}{1-2z+(1-u)z^2+2uz^3-uz^4-2u^3z^7+u^3 z^8}.
\end{equation*}
Then we have ${\bf A}_{3}(z,u) =
\sum_{h\geq 2} {\sf IS}(h){\bf A}_{\gamma}(z,u)$ and
according to eq. (4.3.9),
\begin{equation*}
{\bf IS}(x)=\sum_{h\geq 2}{\sf IS}(h) x^h=\sum_{h\geq 2}{\sf S}(h)\left(x\over{1+x}\right)^h.
\end{equation*}
Therefore we arrive, setting $x=\eta(z,u)$, we have
\begin{equation*}
{\bf A}(z,u) =(1-z){\bf G}\left(\frac{w(u,z)z^6}{(1-z)^2}\right),
\end{equation*}
where $w(u,z)$ is given by
\begin{equation*}
w(u,z)=\frac{u^3}{1-u z^2+u^3 z^6}.
\end{equation*}

The proof of the theorem is complete.

\end{proof}

Theorem \ref{T:bivarate-g-arcs} puts us to use singularity analysis in order to compute the asymptotic distribution of the r.v. $\mathbb{A}_{n}$. Next we discuss about the singularities of a specific parametrization of ${\bf A}(z,u)$. We set $u=\exp^s$ and consider
\begin{equation*}
{\bf A}^*(z,u)=\sum_{n\geq 0} \alpha_{n}(s)z^n,
\end{equation*}
where $\alpha_{n}(s)=\sum_{h\leq \frac{n}{2}}{\sf A}(n,h)\exp^{sh}$. The following analysis gives us an asymptotic enumeration for the coefficients of ${\bf A}^*(z,s)$ in order to establish the central limit theorem, Theorem \ref{T:normal}, for the distribution of the number of arcs.

\begin{proposition}\label{P:h-arc}
Suppose $\epsilon > 0$, $k\ge 2$ and $u=e^{s}$, where $|s|<\varepsilon$. \\
{\sf (a)} Any dominant singularity of ${\bf A}^*(z,s)$ is a
singularity of
\begin{equation*}
{\bf G}\left(\frac{w(\exp^{s},z)z^6}{(1-z)^2}\right).
\end{equation*}
Let $\gamma(s)$ be a solution of the equation
\begin{equation}\label{E:solution-2}
\frac{w(u,z)z^6}{(1-z)^2}-R=0,
\end{equation}
where $R=256\left(\frac{11}{8}-\frac{64}{15\pi}\right)^2$ such that $\gamma(0)$ is the minimal real positive solution of
eq.~(\ref{E:solution-2}). Then $\gamma(s)$ is analytic  in $s$
and a dominant singularity of ${\bf A}^*(z,s)$.\\
{\sf (b)} $\gamma(s)$ is the unique dominant singularity of
${\bf A}^*(z,s)$ and
\begin{equation}\label{E:uniform-2}
{[z^n] {\bf A}^*(z,s)}\sim
a(s)\, n^{-5}
\left( \frac{1}{\gamma(s)}
\right)^n,
\end{equation}
for some {$a(s)\in\mathbb{C}$,} uniformly in $s$ contained in
a neighborhood of $0$.
In particular, the subexponential factors of the coefficients of
${\bf A}^*(z,s)$ coincide with those of ${\bf
G}(z)$ and are independent of $s$.
\end{proposition}

Proposition \ref{P:h-arc} and Theorem \ref{T:normal} will lead us to the following central limit theorem for the distribution of arcs in canonical $3$-noncrossing skeleton diagrams.

\begin{theorem}\label{T:distr-asym}
Suppose that the random variable $\mathbb{A}_n(S)$ denotes the number of arcs in a canonical $3$-noncrossing skeleton diagram, $S$. Then there exists a pair $(\mu,\sigma)$ such that the normalize random variable $\mathbb{A}^*_n$ has asymptotically normal distribution with parameter $(0,1)$, where $\mu$ and $\sigma^2$ are given by
\begin{equation}\label{E:was-arc}
\mu=
-\frac{\gamma'(0)}{\gamma(0)}=0.384971,
\qquad \qquad \sigma^2=
\left(\frac{\gamma'(0)}{\gamma(0)}
\right)^2-\frac{\gamma''(0)}{\gamma(0)}=0.0686453,
\end{equation}
where $\gamma(s)$ is the unique dominant singularity of ${\bf A}(z,e^s)$.
\end{theorem}

Fig.~\ref{F:arc-distr} are shown to authenticate that the asymptotic arc distribution of canonical $3$-noncrossing skeleton diagram is almost coincided with the actual frequencies as compute for a certain length of a diagram.

\restylefloat{figure}\begin{figure}[t]
\begin{center}
\includegraphics[width=1\textwidth]{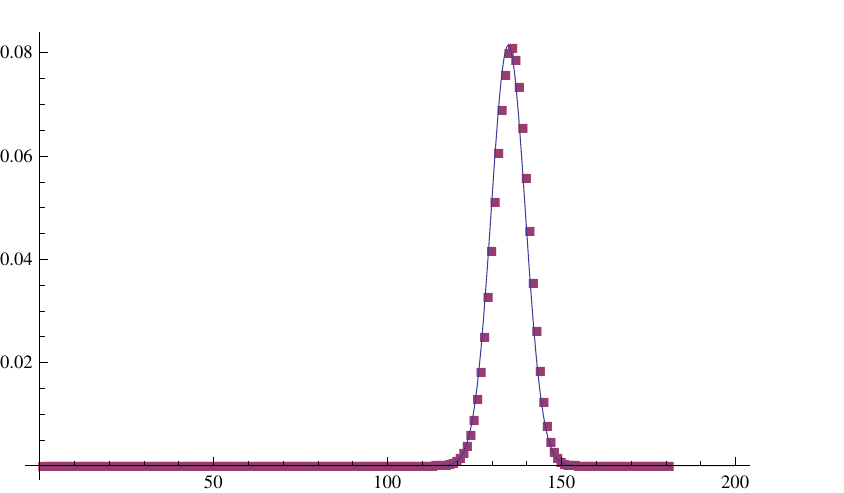}
\hskip15pt
\end{center}
\caption{Central limit theorems of Theorem~\ref{T:distr-asym}
versus exact enumeration data for 3-noncrossing canonical skeleton diagrams with length
$n=350$. We display the asymptotic arc distributions (solid
curves) and actual frequencies (square) as computed for $n={350}$.
\label{F:arc-distr}}
\end{figure}

\subsection{$k$-noncrossing skeleton diagrams}

 For the further study, we plan to extend the space of $3$-noncrossing skeleton diagrams to $k$-noncrossing, where $k\geq 4$. The functional equation for the generating function of $k$-noncrossing skeleton matching is not hard to obtain. Using the same method in Section \ref{S:skeleton-shape} by just replacing ``3" to ``$k$", we have
 \begin{equation}
 {\bf S}_k(z{\bf F}^2_k(z))={\bf F}_k(z),
 \end{equation}
 where ${\bf S}_k$ is the generating function of $k$-noncrossing matching. Then we derive the canonical $k$-noncrossing skeleton diagrams as
 \begin{equation}\label{E:s_kgf}
{\bf S}^{[4]}_{k}(z)=(1-z){\bf G}_k\left(\left({{\sqrt{w_0(z)}z}\over{1-z}}\right)^2\right),
\end{equation}
where
\begin{eqnarray*}
{\bf G}_k(x)={\bf S}_k(x)-1-x,
w_0(z)={z^4 \over {1-z^2+z^6}}.
\end{eqnarray*}

The core to do the asymptotic analysis of eq. (\ref{E:s_kgf}) is to
analyze ${\bf S}_k$. If we follow the same method in Section \ref{S:asymptotic}, we will discover that the property of analytic continuation for ${\bf F}_k(z), k\geq 4$ is unknown. This problem induces a big trouble for discussing the analytic continuation of ${\bf S}_k$. The other problem is that we do not know the explicit value of ${\bf F}_k(\rho_k)$, which directly leads the unknown of singularity of ${\bf S}_k$. Once these problems were figured out, the {\sf cross} problem would be totally solved. Definitely, it is a challenge work because the proof might be totally different from that used to prove Lemma~\ref{L:exs_and_uniq} and Lemma~\ref{L:sk-sing-expansion}.

\renewcommand\it{}
\newpage
\def\papername{Bibliography}\label{bibliographypage}
\bibliographystyle{plain}

\bibliography{thesis-bio}

\end{document}